\documentclass[aihp]{imsart}

\RequirePackage[T1]{fontenc}
\RequirePackage[utf8]{inputenc}
\RequirePackage{microtype}
\RequirePackage{amsthm,amsmath,amsfonts,amssymb}
\RequirePackage{mathtools,mathrsfs,dsfont,bbm,latexsym,extarrows}
\RequirePackage[title]{appendix}
\RequirePackage[subfigure]{tocloft}
\RequirePackage{graphicx,subfig,booktabs,makecell}
\RequirePackage[inline]{enumitem}
\RequirePackage{xcolor}
\RequirePackage[numbers,sort&compress]{natbib}
\RequirePackage[colorlinks,citecolor=blue,urlcolor=blue]{hyperref}
\startlocaldefs
\numberwithin{equation}{section}
\theoremstyle{plain}
\newtheorem{theorem}{Theorem}[section]
\newtheorem{corollary}[theorem]{Corollary}
\newtheorem{lemma}[theorem]{Lemma}
\newtheorem{proposition}[theorem]{Proposition}
\theoremstyle{definition}
\newtheorem{definition}[theorem]{Definition}

\newtheorem{remark}[theorem]{Remark}

\def\E{\mathbb{E}}
\def\d{\mathrm{d}}

\def\R{\mathbb{R}}
\def\D{\textbf{D}}
\endlocaldefs

\begin{document}
\let\manuscriptTOC\tableofcontents

\begin{frontmatter}
\title{Strong existence for SDEs with singular drifts driven by fractional Brownian motion}
\runtitle{Strong existence for SDEs with singular drifts driven by fractional Brownian motion}

\begin{aug}
\author[A]{\inits{J.}\fnms{Jiazhen}~\snm{Gu}\ead[label=e1]{jiazhengu@nuaa.edu.cn}}
\author[A]{\inits{Q.}\fnms{Qian}~\snm{Yu}\ead[label=e2]{qyumath@163.com}}
\address[A]{School of Mathematics, Nanjing University of Aeronautics and Astronautics,
Nanjing, Jiangsu 211106, P.R. China\printead[presep={,\ }]{e1,e2}}
\end{aug}

\begin{abstract}
Let $B^H$ be a $d$-dimensional fractional Brownian motion with Hurst
parameter $H\in(0,1/2)$. We prove the existence of a strong solution to
$$
    \d X_t=b(t,X_t)\d t+\d B_t^H,\qquad X_0=x_0\in\mathbb R^d,
$$
when $b\in L^q([0,T];L^p(\mathbb R^d;\mathbb R^d)),
2\le p,q<\infty$
and the subcritical Ladyzhenskaya--Prodi--Serrin condition
$\frac1q+\frac{Hd}{p}<1-H$
is satisfied. Thus, within the finite-exponent range $p,q\ge2$, the
strong-existence result covers the full subcritical region, including
the regime $p<2Hd$, where a uniform strong Novikov condition may fail.
The proof is based on a compactness criterion for Wiener functionals.
Its main ingredients are an exponentially quantitative derivative
estimate for the finite-dimensional densities of fractional Brownian
motion, a conditional moment estimate obtained through a martingale
representation of fractional Brownian motion, and a frequency-localized
truncation of the Cameron--Martin shift. These estimates yield uniform
Malliavin--Sobolev bounds for smooth approximations of the equation.
\end{abstract}



\end{frontmatter}
\let\tableofcontents\manuscriptTOC
	
	\tableofcontents
	
	\begingroup
	\renewcommand{\thefootnote}{}
	\footnotetext{\textit{E-mail address}: jiazhengu@nuaa.edu.cn (J. Gu), qyumath@163.com (Q. Yu).}
	\endgroup
	
	\section{Introduction}
	We consider the stochastic differential equation (SDE)
	\begin{equation}\label{sec1-eq.1}
		\d X_t=b(t,X_t)\d t+\d B_t^H,\qquad X_0=x_0\in\R^d
	\end{equation}
	where $b(t,\cdot)\in L^p_xL^q_t:=L^q([0,T];L^p(\R^d)),p,q\geq1$ and $B^H$ is an $\mathcal{F}_t$-adapted fractional Brownian motion (fBm) with Hurst parameter $H\in(0,1)$ defined on a complete filtered probability space $(\Omega,\mathcal{F},\mathbb{P},(\mathcal{F}_t)_{t\in[0,T]})$. Assume that $b$ satisfies the Ladyzhenskaya-Prodi-Serrin (LPS) condition \cite[Example 1.2]{galeati2022solution}
	\begin{equation}\label{cond1}
		\tag{H1}
		\kappa_H:=1-H-\frac1q-\frac{Hd}{p}>0.
	\end{equation}

A fundamental problem is to determine whether the subcritical condition
\eqref{cond1} is sufficient for strong well-posedness of \eqref{sec1-eq.1}. In the present
paper, we focus on the strong-existence part of this problem. More
precisely, we construct a solution on the prescribed fractional Brownian
probability space without invoking pathwise uniqueness. The uniqueness
problem in the full subcritical regime remains open.

	\subsection{Main Result}
	Our main result is stated as follows.
	\begin{theorem}\label{sec1-thm.1}
Let $H\in(0,1/2)$ and let
$    b\in L^q([0,T];L^p(\mathbb R^d;\mathbb R^d)),
    2\le p,q<\infty.$
Assume \eqref{cond1}.
Let $B^H$ be an $(\mathcal F_t)$-fractional Brownian motion in the sense
of Definition \ref{sec1-def1.3}. Then, for every $x_0\in\mathbb R^d$,
there exists an $\mathbb R^d$-valued continuous process
$X=(X_t)_{t\in[0,T]}$, adapted to the completed natural filtration of
$B^H$, such that
$$
    \int_0^T |b(s,X_s)|\d s<\infty
    \qquad\mathbb{P}\text{-a.s.},
$$
and
$$
    X_t=x_0+\int_0^t b(s,X_s)\d s+B_t^H
$$
holds simultaneously for all $t\in[0,T]$ on an event of probability one.
In particular,  \eqref{sec1-eq.1} admits a strong solution.
	\end{theorem}

\begin{remark}\label{rem:existence-only}
Theorem \ref{sec1-thm.1} is a strong-existence result. It does not assert pathwise uniqueness and therefore should not be interpreted as a strong-well-posedness theorem. In particular, the comparison with the classical result of Krylov and R\"ockner concerns the construction of a
strong solution, rather than the full existence-and-uniqueness theory.
\end{remark}
	
	\subsection{Motivation and Previous Results}
	An ill-posed deterministic differential system might become well-posed by the addition of irregular noise. For instance, the differential equation $\d X_t=\mathrm{sign}(X_t)|X_t|^\alpha\d t,\alpha\in(0,1),X_0=0$ has infinitely many solutions, while the perturbed equation $\d X_t=\mathrm{sign}(X_t)|X_t|^\alpha\d t+\d W_t$, where $W$ is a Brownian motion, has a unique solution. This phenomenon is called \textit{regularization by noise}. We refer to \cite{flandoli2011random, flandoli2010well, zhang2016stochastic} for more examples and developments.
	
	Inspired by these examples, regularization by noise has received extensive research attention. In the case where the noise is a standard Brownian motion, the pioneering works of Zvonkin \cite{zvonkin1974transformation} and Veretennikov \cite{veretennikov1981strong} established strong existence and uniqueness of solutions to the SDE
	\begin{equation}\label{sec1-eq.3}
		\d X_t=b(t,X_t)\d t+\d W_t,\qquad X_0=x\in\R^d
	\end{equation} 
	when the drift coefficient $b$ is merely bounded and measurable. This result was extended by Krylov and R\"ockner \cite{krylov2005strong} to integrable drifts $b(t,\cdot)\in L_x^pL_t^q$ with $p\geq 2$, $q>2$ and 
	$$\frac{2}{q}+\frac{d}{p}<1.$$
	Recently, Krylov \cite{krylov2021strong} has made significant progress regarding the critical condition $2/q+d/p=1$, establishing the strong well-posedness of \eqref{sec1-eq.3} for the case that $b(t,x)=b(x)\in L^d(\R^d)$ with $d\geq 3$. Subsequently, in \cite{krylov2025weak}, he addressed situations where the equation coefficients are time-dependent and provided a detailed discussion of conditional and unconditional strong uniqueness. Following his works, R\"ockner and Zhao \cite{rockner2023sdes,rockner2025sdes} established weak and strong well-posedness of \eqref{sec1-eq.3} under the critical condition $2/q+d/p=1$ with $p,q\in(2,\infty)$ and $d\geq3$. The method they employed to construct strong solutions relies on a compactness criterion for random fields in Wiener spaces, which provided profound inspiration for our work. We will discuss this in detail later.
	
	In stark contrast to the martingale setting, regularization-by-noise phenomena in non-martingale frameworks, particularly those involving SDEs driven by fBm, remain far less well understood. A key reason for this gap is that all proofs in the aforementioned articles rely heavily on It\^o's calculus, which cannot be applied directly beyond the semimartingale setting. Since it has been illustrated in \cite{catellier2016averaging} that a rougher driving noise may allow for a more irregular drift $b$ such that the SDE remains well-posed, investigating the case of SDEs driven by fBm with $H<1/2$ becomes particularly important.
	
	Regularization by fractional Brownian noise was initiated by Nualart and Ouknine in \cite{nualart2002regularization,nualart2003stochastic}, who developed Girsanov's transformation for fBm. Building on their seminal works, Ba\~nos et al. \cite{banos2020strong} combined Girsanov's theorem with Malliavin calculus to construct strong solutions, provided that the coefficient $b\in L^1(\R^d;L^\infty([0,T];\R^d))\cap L^\infty(\R^d;L^\infty([0,T];\R^d))$ and the Hurst parameter $H<1/(2d+4)$. In addition to these two works, many relevant investigations utilized Girsanov's theorem to analyze weak and strong existence of the SDE \eqref{sec1-eq.1}. We highlight several representative works as follows: 
	\cite{banos2022strong},  \cite{butkovsky2021approximation},
	\cite{butkovsky2024regularization},
	\cite{galeati2023distribution}, \cite{hao2024distribution}.
	
	An alternative way to study regularization by fBm is the stochastic sewing technique, which was introduced by Catellier and Gubinelli \cite{catellier2016averaging}, and L\^e \cite{le2020stochastic}. Following their prior works, stochastic sewing techniques have been further advanced and applied to the study of equation \eqref{sec1-eq.1} in \cite{anzeletti2023regularisation}, \cite{galeati2022solution}, \cite{goudenege2025numerical}. A breakthrough in establishing the well-posedness of \eqref{sec1-eq.1} in the subcritical regime was obtained by \cite{butkovsky2023stochastic}, which proved that the equation \eqref{sec1-eq.1} admits a weak solution if $b$ is time-independent and belongs to $L^p(\R^d)$ with $Hd/p<1-H$. A subsequent work \cite{butkovsky2023weak} considered the case of time-dependent drift and showed the weak existence of a solution to the SDE \eqref{sec1-eq.1} under the condition
	\begin{equation}\label{sec1-eq.4}
		\frac{1-H}q+\frac{Hd}p<1-H,
	\end{equation}
	which is even better than the LPS condition \eqref{cond1}. \cite{butkovsky2023weak} also showed that weak existence might fail if this condition is not satisfied.
	
	In summary, we present previous results in the following table and conduct a systematic comparison between these results and ours.
	
	\begin{table}[htbp]
		\centering  
		\caption{Comparison to Previous Results (Parameter Range for Weak / Strong Solution)}  
		\label{comparison}  
		\renewcommand{\arraystretch}{1.5}
		\makebox[\linewidth][c]{%
			\begin{tabular}{|l|c|c|}
				\hline 
				\textbf{References} & \textbf{Weak Solution} & \textbf{Strong Solution} \\
				\hline
				Krylov $\&$ R\"ockner \cite{krylov2005strong} & / & $\frac2q+\frac dp<1$, $p\geq2$, $q>2$, $H=\frac12$ \\
				\hline
				Anzeletti et al. \cite{anzeletti2023regularisation} & $q=\infty$, $\frac dp<\frac1{2H}-\frac12$ & $q=\infty$, $\frac dp<\frac1{2H}-1$, $H<\frac12$\\
				\hline
				Butkovsky $\&$ Gallay \cite{butkovsky2023weak} & $\frac{1-H}q+\frac{Hd}p<1-H$, $H\in (0,1)$ & / \\
				\hline
				Butkovsky et al. \cite{butkovsky2024regularization} & $\left\{\begin{array}{l}
					p\geq 2dH, \\
					\frac{Hd}p+\frac1q<1-H, H<\frac12,\\
					\frac{Hd}p<\frac{p^{-1}-q^{-1}}{p^{-1}-2^{-1}}\left(\frac12-H\right) \\
					\qquad \text{ if } p<(1-H)^{-1} \text{ and } q>2
				\end{array}\right.$ & same\\
				\hline
				Butkovsky et al. \cite{butkovsky2023stochastic} & $q=\infty$, $\frac{Hd}p<1-H$, $H\in(0,1)$ & / \\
				\hline
				Catellier $\&$ Gubinelli \cite{catellier2016averaging} & $q=\infty$, $\frac{Hd}p<\frac1{2H}-1$, $H<\frac12$ & / \\
				\hline
				\footnotemark[1]Galeati $\&$ Gerencs\'er \cite{galeati2022solution} & \makecell{$\frac 1q+\frac{Hd}p<1-H$,\\ $\frac dp<\frac1{2H}-\frac12$, $H\in(0,1)$} & \makecell{$\frac1{\min\{q,2\}}+\frac{Hd}p<1-H$,\\ $H\in(0,\infty)\setminus\mathbb{N}_+$}\\
				\hline
				L\^e \cite{le2020stochastic} & \makecell{$\frac1q+\frac{Hd}p<\frac12$, $H<\frac12$,\\ $p,q\in[2,\infty]$} & \makecell{$\frac1q+\frac{Hd}p<\frac12-H$, $H<\frac12$,\\ $p,q\in[2,\infty]$} \\
				\hline
			\end{tabular}%
		}
	\end{table}

\begin{remark}
\label{rem:comparison}
Galeati and Gerencs\'er \cite{galeati2022solution} established strong well-posedness under
$$
    \frac{1}{\min\{q,2\}}+\frac{Hd}{p}<1-H.
$$
When $q\le2$, this condition coincides with the LPS condition
$\frac1q+\frac{Hd}{p}<1-H.$
When $q>2$, however, their condition becomes
$\frac12+\frac{Hd}{p}<1-H,$
which is generally stronger than the LPS condition. The contribution of
Theorem \ref{sec1-thm.1} is therefore the construction of a strong solution
throughout the remaining finite-exponent region
$$
    2\le p,q<\infty,
    \qquad
    \frac1q+\frac{Hd}{p}<1-H.
$$
The price is that pathwise uniqueness is not obtained in this enlarged
region.

Consequently, when Theorem \ref{sec1-thm.1} is combined with the
strong-well-posedness result of Galeati and Gerencs\'er, strong existence
is available for all finite $q>1$ in the sector $p\ge2$ covered by the
corresponding subcritical conditions. This combined conclusion should
not be attributed to Theorem \ref{sec1-thm.1} alone.
\end{remark}

	\footnotetext[1]{FBm of parameter $H=1$ is somewhat ill-defined; however, one can extend the definition to all values of $H\in(0,\infty)\setminus\mathbb{N}$ inductively as in \cite{picard2010representation} by $B_t^{H+1}=\int_0^t B_s^H\d s$.}
	
	As mentioned in preceding literature review, Butkovsky and Gallay \cite{butkovsky2023weak} established weak existence of \eqref{sec1-eq.1} under a milder condition \eqref{sec1-eq.4}. Based on weak existence, most of the works discussed above obtain strong solutions by employing Yamada-Watanabe argument \cite{yamada1971uniqueness}, which asserted that weak existence combined with pathwise uniqueness implies strong existence. There is, however, another approach that constructs strong solutions by employing a compactness criterion in Wiener spaces (see Theorem \ref{sec2-thm.2}). This method was initiated by Meyer-Brandis and Proske \cite{meyer2010construction}, and subsequently developed in \cite{banos2024}, \cite{banos2020strong}, \cite{banos2022strong},  \cite{menoukeu2013variational} and \cite{rockner2025sdes}, in both the Brownian and fractional Brownian motion settings. Our work also builds on this compactness-based approach.

Main contributions and proof strategy.
The proof contains three ingredients that may be of independent interest.

First, in Proposition \ref{sec2-prop.4}, we derive a
derivative estimate for the joint density of
$$
    (B_{t_1}^H,\ldots,B_{t_n}^H)
$$
whose constant grows at most exponentially in $n$. The quantitative dependence $C^n$ is essential when the estimate is applied to the arbitrarily high-order iterated integrals arising from the Malliavin derivative of the approximating solutions.

Second, a direct use of Girsanov's theorem would require uniform exponential integrability of the Cameron-Martin shifts. As shown in Proposition \ref{sec4-prop.3}, such a strong Novikov estimate may fail
when $p<2Hd$. To overcome this difficulty, we introduce a dyadic spectral decomposition of the shift and localize the compactness argument to sets on which each frequency block is controlled.

Third, Proposition \ref{sec4-prop.5} provides a conditional
estimate for $|b(t,X_t)|^\rho$ with $\rho\in[(1-H)^{-1},2)$. Its proof uses the martingale $r\mapsto\mathbb E_r[B_t^H]$ and an auxiliary semimartingale, allowing It\^o's formula to be used despite the non-semimartingale nature of
fractional Brownian motion. Combined with the block-integral estimates of Section~\ref{sec5}, this yields the uniform Malliavin--Sobolev bounds required by the compactness criterion.

The resulting limit is constructed directly on the probability space generated by $B^H$; no Yamada--Watanabe argument or pathwise uniqueness result is used.

	The rest of this paper is organized as follows. In the remainder of this section, we introduce some notation, conventions and definitions used in the sequel. In Section \ref{sec2}, we recall some basic tools, including Malliavin calculus, a compactness criterion in Wiener spaces and Girsanov's theorem for fBm. In particular, we derive a key estimate for the derivatives of the joint density of $(B_{t_1}^H,B_{t_2}^H,\cdots,B_{t_n}^H)$. Section \ref{sec3} outlines the proof strategy, while Sections \ref{sec4} and \ref{sec5} are devoted to establishing the uniform estimates required for the compactness argument. The proof of the main result is completed in Section \ref{sec6}. Several technical lemmas are collected in Appendix \ref{A}, while the proof of Lemma 5.2 is deferred to Appendix \ref{B}. Throughout the paper, the letter $C$, possibly with subscripts, denotes a generic positive finite constant and may change from line to line.
	
	\subsection{Notation, Conventions and Definitions}\label{sec1-3}
	
	We close this section by introducing some notation, conventions and definitions.
	\begin{itemize}
		\item For $r\geq1$, $r'$ denotes its conjugate exponent with $1/r+1/r'=1$.
		\item For $m\in\mathbb{N}$ and $0\leq\theta<t\leq T$, set $$\Delta_{\theta,t}^m:=\{(s_1,\cdots,s_m)\in[0,T]^m: \theta\leq s_1< \cdots< s_m\leq t\},$$ 
		and write $\d s=\d s_1\cdots \d s_m$.
		\item For $x=(x^{(1)},\cdots,x^{(d)})^T\in\R^d$, $x^{(i)}$ denotes the $i$-th component of $x$. 
		\item For a differentiable map $f:\R^{d_1}\ni x\mapsto(f^{(1)}(x),\cdots,f^{(d_2)}(x))^{\mathrm{T}}\in\R^{d_2}$, the Jacobian matrix is denoted by $\nabla f$, that is $\nabla f=\left(\frac{\partial}{\partial x_j}f^{(i)}(x)\right)_{i=1,\cdots,d_2,j=1,\cdots,d_1}$.
		\item For $f:[0,T]\times\R^d\mapsto\R^d$, define
		$$\Vert f\Vert_{L^p_x}(t):=\left(\int_{\R^d}|f(t,x)|^p\d x\right)^{\frac1p},\text{ and }\;\Vert f\Vert_{L^p_xL_t^q}:=\left(\int_0^T\Vert f\Vert_{L^p_x}^q(t)\d t\right)^{\frac1q}.$$
		\item Let $(\mathcal{F}_t)_{t\in[0,T]}$ be a complete filtration. The conditional expectation $\E[\cdot|\mathcal{F}_s]$ is denoted by $\E_s[\cdot]$.
	\end{itemize}
	\begin{definition}\label{sec1-def1.3}
		$($$\mathcal{F}_t$-fractional Brownian motion \cite[Definition 1]{nualart2002regularization} $)$ Let $(\Omega,\mathcal{F},\mathbb{P},(\mathcal{F}_t)_{t\in[0,T]})$ be a complete filtered probability space. We say that $B^H$ $(H<1/2)$ is an \textit{$\mathcal{F}_t$-fractional Brownian motion} if there exists an $\mathcal{F}_t$-Brownian motion $W$ such that $B^H$ can be formulated as
		$$B_t^H=\int_0^tK_H(t,s)\d W_s,$$
		where for $0\leq s< t\leq T$,
		\begin{align*}
			K_H(t,s):=C_H\left[\left(\frac{t}{s}\right)^{H-\frac12}(t-s)^{H-\frac12}+\left(\frac12-H\right) s^{\frac12-H}\int_s^tr^{H-\frac32}(r-s)^{H-\frac12}\d r\right].
		\end{align*}
	\end{definition}
	\begin{definition} 
		Let
$(\Omega,\mathcal F,\mathbb P,(\mathcal F_t)_{t\in[0,T]})$
be a complete filtered probability space satisfying the usual conditions.
A weak solution to \eqref{sec1-eq.1} consists of an
$(\mathcal F_t)$-fractional Brownian motion $B^H$ and an
$(\mathcal F_t)$-adapted continuous process $X$ such that
$$
    \int_0^T |b(s,X_s)|\d s<\infty
    \qquad\text{a.s.},
$$
and
$$
    X_t=x_0+\int_0^t b(s,X_s)\d s+B_t^H,
    \qquad t\in[0,T],
$$
holds simultaneously for all $t\in[0,T]$ almost surely.
A solution is called strong if the underlying fractional Brownian motion is prescribed in advance and $X$ is adapted to the completed natural filtration
$$
    \mathcal F_t^{B^H}
    :=\sigma(B_s^H:0\le s\le t)\vee\mathcal N,
$$
where $\mathcal N$ denotes the collection of $\mathbb P$-null sets.
	\end{definition}
	\begin{definition}\label{sec1-def.5}
		$($\cite[Section 2.1 and 2.2]{nochetto2015pde}$)$ Let $E$ be a Hilbert space. For $s\in(0,1)$ and $r\in[1,\infty)$, define
		$$W^{s,r}([0,T];E):=\left\{f\in L^r([0,T];E):[f]_{W^{s,r}}^r:=\int_0^T\int_0^T\frac{\Vert f(t)-f(u)\Vert_E^r}{|t-u|^{1+sr}}\d u\d t<\infty\right\},$$
		equipped with the norm $\Vert f\Vert_{W^{s,r}}^r:=\Vert f\Vert_{L^r}^r([0,T];E)+[f]_{W^{s,r}}^r$.
		
		Let
		$$\varphi_k(t)=\sqrt{\frac2T}\sin\frac{k\pi t}{T},\qquad\lambda_k=\frac{k^2\pi^2}{T^2},\qquad k\in\mathbb{N},$$
		be the Dirichlet eigenbasis and eigenvalues on $[0,T]$. For $f\in L^2([0,T];E)$, $f$ admits a spectral decomposition $f=\sum_{k=1}^\infty f_k\varphi_k$ in $L^2([0,T];E)$, where $f_k=\int_0^Tf(t)\varphi_k(t)\d t$. For $s\in\R$, the \textit{spectral Sobolev space} $\mathbb{H}_{([0,T];E)}^s$ is defined as the completion of the space
		$$\mathcal{S}_E:=\left\{g=\sum_{k=1}^N g_k\varphi_k:N\in\mathbb{N},\ g_k\in E\right\},$$
		under the norm $\Vert g\Vert_{\mathbb{H}_{([0,T];E)}^s}^2
		:=\sum_{k=1}^N\lambda_k^s\Vert g_k\Vert_E^2$. When $E=\R^d$ or $\R^{d\times d}$, we simply write $W^{s,r}$ and $\mathbb{H}^s$.
	\end{definition}
	\section{Preliminaries }\label{sec2}
	We first introduce the basic framework of Malliavin calculus in this section. The reader can refer to \cite{biagini2008stochastic,nualart2006malliavin} for more details. Let $B^H=\{(B_t^{H,(1)},\cdots,B_t^{H,(d)})\}_{t\in[0,T]}$ be a $d$-dimensional $\mathcal{F}_t$-adapted fractional Brownian motion with Hurst parameter $H<1/2$. Let $\mathcal{E}$ be the space of $\R^d$-valued indicator functions on $[0,T]$. Then, the reproducing kernel Hilbert space $\mathcal{H}$ associated with $B^H$ is defined as the closure of $\mathcal{E}$ for the scalar product
	$$\langle(\mathbbm{1}_{[0,t_1]},\cdots,\mathbbm{1}_{[0,t_d]}),(\mathbbm{1}_{[0,s_1]},\cdots,\mathbbm{1}_{[0,s_d]})\rangle_{\mathcal{H}}=\sum_{i=1}^dR(t_i,s_i),\qquad0\leq t_i,s_i\leq T,i=1,\cdots,d$$
	where 
	$$R(t,s)=\frac12\left(s^{2H}+t^{2H}-|t-s|^{2H}\right).$$
	Consider the linear operator $\mathcal{K}_H^*$ from the space $\mathcal{E}$ to $L^2([0,T];\R^d)$ defined by
	$$(\mathcal{K}_H^*\phi)(s)=K_H(T,s)\phi(s)+\int_s^T(\phi(t)-\phi(s))\frac{\partial K_H}{\partial t}(t,s)\d t.$$
	We can see that $\mathcal{K}_H^*$ induces an isometry between $\mathcal{E}$ and $L^2([0,T];\R^d)$ that can be extended to the Hilbert space $\mathcal{H}$. From this, we define the Wiener integral $B^H(\phi)=\int_0^T\langle(\mathcal{K}_H^*\phi)(s),\d W_s\rangle_{\R^d}$. Denote by $\mathcal{S}_H$ the set of smooth cylindrical random variables of the form
	$$F=f\left(B^H(\phi_1),\cdots,B^H(\phi_n)\right)$$
	where $n\geq1$, $\phi_i\in\mathcal{H}$ and $f:\R^n\mapsto\R$ is a  bounded $C^\infty$ function with bounded derivatives. Then, the Malliavin derivative with respect to fBm is given by
	$$\D^H_tF=\sum_{i=1}^n\phi_i(t)\frac{\partial}{\partial x_i}f\left(B^H(\phi_1),\cdots,B^H(\phi_n)\right),\qquad F\in\mathcal{S}_H.$$
	More generally, we can introduce iterated derivatives by $\D^{H,k}_{t_1,\cdots,t_k}F=\D^H_{t_1}\cdots \D^H_{t_k}F$. For $p_1\geq1$, we define $\mathbb{D}^{k,p_1}_H$ as the closure of the class of cylindrical random variables with respect to the norm
	$$\Vert F\Vert_{k,p_1}=\left(\E[|F|^{p_1}]+\sum_{i=1}^k\E\Big[\left\Vert \D^{H,i} F\right\Vert^{p_1}_{\mathcal{H}^{\otimes i}}\Big]\right)^{\frac1{p_1}}.$$ 
	and $\mathbb{D}^\infty_H:=\cap_{p_1\geq1}\cap_{k\geq1}\mathbb{D}^{k,p_1}_H$. It is well known that the Malliavin calculus is readily applicable with respect to $W$. To distinguish the Malliavin derivatives associated with $B^H$ and those with $W$, we denote the Malliavin derivative (and the corresponding Sobolev spaces) for $W$ by $\D$ (and by $\mathbb{D}^{k,p}$, respectively). The relation between the two operators $\D^H$ and $\D$ is given by the following (see e.g. \cite[Proposition 5.2.1]{nualart2006malliavin})
	\begin{proposition}\label{sec2-prop.2.1}
		For any $F\in \mathbb{D}^{1,2}=\mathbb{D}^{1,2}_H$, we have $\boldsymbol{\mathrm{D}} F=\mathcal{K}_H^*\boldsymbol{\mathrm{D}}^HF$.
	\end{proposition}
	\subsection{A Compactness Criterion for $L^2$ Random Fields}\label{sec2-2}
	The following result was established by \cite{da_prato1992compact}, which provides a compactness criterion for random variables in $L^2(\Omega)$. This criterion plays a crucial role in constructing strong solutions to the SDE \eqref{sec1-eq.1}.
	\begin{theorem}\label{sec2-thm.2}
		Assume that the sequence of $\mathcal{F}_T$-measurable random variables $\{F_n\}_{n=1}^\infty\subset L^2(\Omega)$ satisfies the following conditions: $\sup_n\E[\Vert F_n\Vert^2]<\infty$
		and there exists a constant $\beta>0$ such that $\sup_n \E[\Vert \boldsymbol{\mathrm{D}}_\cdot F_n\Vert_{W^{\beta,2}}^2]<\infty$.
		Then, the sequence $\{F_n\}_{n=1}^\infty$ is relatively compact in $L^2(\Omega)$.
	\end{theorem} 
	\begin{remark}\label{sec2-rm.3}
		Let $\mathcal{H}_\beta:=\{h\in\mathcal{H}:\mathcal{K}_H^*h\in W^{\beta,2}\}$ equipped with an inner product $\langle \mathcal{K}_H^*f,\mathcal{K}_H^*g\rangle_{W^{\beta,2}}$. Then, by Proposition \ref{sec2-prop.2.1}, the second condition of the above theorem is equivalent to $\sup_n\E[\Vert \boldsymbol{\mathrm{D}}_\cdot^H F_n\Vert_{\mathcal{H}_\beta}^2]$ $<\infty$ for some $\beta>0$.
	\end{remark}
	\subsection{Girsanov's Transformation}
	Girsanov's theorem for fBm was introduced by \cite{nualart2002regularization}. Before we state it explicitly, we recall some basic definitions on fractional calculus. More details can be found in \cite{kilbas2006theory}. Let $\alpha\in(0,1)$. For $f\in L^r([a,b])$ with $1\leq r<1/\alpha$, we define
	\begin{equation*}
		(I^\alpha_{a+}f)(x)=\frac1{\Gamma(\alpha)}\int_a^xf(t)(x-t)^{\alpha-1}\d t
	\end{equation*}
	as the \textit{left-sided Riemann-Liouville fractional integral} of order $\alpha$. On the other hand, the fractional derivative was introduced as the inverse operation. Denote by $I^\alpha_{a+}(L^r)$ the image of $L^r([a,b])$ by the operator $I^\alpha_{a+}$. Then, the \textit{left-sided Riemann-Liouville fractional derivative} of order $\alpha$
	is defined by
	\begin{equation*}
		(D^\alpha_{a+}f)(x)=\frac1{\Gamma(1-\alpha)}\frac{\d }{\d x}\int_a^xf(t)(x-t)^{-\alpha}\d t,
	\end{equation*}
	where $f\in I^\alpha_{a+}(L^r)$. Consider the linear operator $K_H$ on $L^2([0,T];\R^d)$ defined by $(K_H\phi)(t):=\int_0^tK_H(t,s)\phi(s)\d s$. With the notation introduced above, $K_H$ admits the following representation: $(K_H\phi)(t)=I_{0+}^{2H}t^{1/2-H}I_{0+}^{1/2-H}t^{H-1/2}\phi,\forall\phi\in L^2([0,T];\R^d)$. Recall by construction that $D^{\alpha}_{a+}(I^\alpha_{a+}f)=f,\forall f\in L^r([a,b])$. Hence, the inverse of $K_H$ is given by
	$$(K_H^{-1}\phi)(t)=t^{\frac12-H}D_{0+}^{\frac12-H}t^{H-\frac12}D_{0+}^{2H}\phi,\qquad\phi\in I_{0+}^{H+\frac12}(L^2([0,T];\R^d)).$$
	If $\phi$ is absolutely continuous, we have
	$$(K_H^{-1}\phi)(t)=t^{H-\frac12}I_{0+}^{\frac12-H}t^{\frac12-H}\phi'$$
	
	Now, we present Girsanov's theorem for fBm.
	\begin{theorem}\label{sec2-theorem.1}
		Let $B_t^H=\int_0^tK_H(t,s)\d W_s$ be an $\mathcal{F}_t$-fractional Brownian motion with Hurst parameter $H<1/2$. Consider the shifted process
		$\tilde{B}_t^H:=B_t^H+\int_0^tu_s\d s$,
		where $\{u_s\}_{s\in[0,T]}$ is an $\mathcal{F}_t$-adapted process with integrable trajectories. Assume that:
		\begin{enumerate}
			\item $\int_0^{\cdot}u_s\d s\in I_{0+}^{H+1/2}(L^2([0,T];\R^d))$ almost surely;
			\item $\E[\xi_T]=1$, where
			\begin{equation*}
				\xi_T=\exp\left(-\int_0^TK_H^{-1}\left(\int_0^{\cdot}u_r\d r\right)(s)\d W_s-\frac12\int_0^T\left|K_H^{-1}\left(\int_0^{\cdot}u_r\d r\right)(s)\right|^2\d s\right).
			\end{equation*}
		\end{enumerate}  
		Then, $\tilde{W}_t:=W_t+\int_0^tK_H^{-1}(\int_0^\cdot u_r\d r)(s)\d s$ is a Brownian motion and $\tilde{B}_t^H=\int_0^tK_H(t,s)\d\tilde{W}_s$ is an $\mathcal{F}_t$-fractional Brownian motion with Hurst parameter $H$ under a new probability $\mathbb{Q}$ defined by $\d \mathbb{Q}=\xi_T\d \mathbb{P}$.
	\end{theorem}
	\subsection{Density Estimate}
	In this subsection, we will use properties of high-dimensional Hermite polynomials to derive estimates for the derivatives of the joint density of fBm. Although we expect that upper bounds can be established for derivatives of all orders, dealing with this general case will be much more difficult. Therefore, we will consider a simple situation, which suffices for the integral estimates in the remainder of the paper. 
	
	For $0< t_1<\cdots<t_n\leq T$, $p_{t_1,\cdots,t_n}(x_1,\cdots,x_n)$ and $q_{t_1,\cdots,t_n}(y_1,\cdots,y_n)$ denote the joint density of $$(B_{t_1}^H,\cdots,B_{t_n}^H)\quad\text{and}\quad(B_{t_1}^H,B_{t_2}^H-B_{t_1}^H,\cdots,B_{t_n}^H-B_{t_{n-1}}^H),$$
	respectively, where $\{B_t^H\}_{t\in[0,T]}$ is a $d$-dimensional fBm with $H<1/2$. For simplicity, we write $p_t=p_{t_1,\cdots,t_n}$ and $q_t=q_{t_1,\cdots,t_n}$. Then, $$p_t(x_1,\cdots,x_n)=q_t(x_1,x_2-x_1,\cdots,x_n-x_{n-1}).$$
	For $\alpha=(\alpha_1,\cdots,\alpha_n)\in\{1,\cdots,d\}^n$, define the derivative operators by $\mathcal{D}_i^{(\alpha_i)}:=\partial_{y_i^{(\alpha_i)}}-\partial_{y_{i+1}^{(\alpha_i)}}$ for $i=1,\cdots,n-1$ and $\mathcal{D}_n^{(\alpha_n)}:=\partial_{y_n^{(\alpha_n)}}$. Then, 
	\begin{equation}\label{sec2-eq.1}
		\partial_{x_1^{(\alpha_1)}}\cdots\partial_{x_n^{(\alpha_n)}}p_t(x_1,\cdots,x_n)=\mathcal{D}_1^{(\alpha_1)}\cdots\mathcal{D}_n^{(\alpha_n)}q_t(x_1,x_2-x_1,\cdots,x_n-x_{n-1}).
	\end{equation}
	In the following, we will show that: 
	\begin{proposition}\label{sec2-prop.4} 
		There exists a constant $C$ independent of $n$ such that 
		\begin{equation} 
			\left|\mathcal{D}_1^{(\alpha_1)}\cdots\mathcal{D}_n^{(\alpha_n)}q_t(y_1,\cdots,y_n)\right|\leq C^n\prod_{j=1}^n\Delta_j^{-Hd}\ell_j^{-H}\exp\left(-\frac{|y_j|^2}{8\Delta_j^{2H}}\right) 
		\end{equation}  
		where $\Delta_j=t_j-t_{j-1}$ and $\ell_j=\min\{\Delta_j,\Delta_{j+1}\}$ for $j=1,\cdots,n-1$ with the convention that $t_0=0$ and $\ell_n=\Delta_n$. 
	\end{proposition} 
	
	Before proving Proposition \ref{sec2-prop.4}, we recall the definition and properties of high-dimensional Hermite polynomials. 
	\begin{definition}
		Let $\Sigma\in\R^{n\times n}$ be a positive definite matrix and let $\alpha=(\alpha_1,\cdots,\alpha_n)\in \mathbb{N}_0^n$ be a multi-index. Then, the Hermite polynomial $H_\alpha(x;\Sigma)$ with matrix $\Sigma$ and index $\alpha$ is defined as
		$$H_{\alpha}(x;\Sigma):=(-1)^{|\alpha|}\frac{\partial^\alpha p_\Sigma(x)}{p_\Sigma(x)},$$
		where $p_{\Sigma}(x)$ is the density function of the centered Gaussian random vector in $\mathbb{R}^n$ with covariance matrix $\Sigma$.
	\end{definition}
	
	As shown in \cite[Proposition 6]{rahman2017wiener}, the Hermite polynomial $H_{\alpha}(x,\Sigma)$ has the following generating function formula 
	\begin{equation}\label{sec2-eq.3}
		\exp\left(\lambda^T\Sigma^{-1}x-\frac12 \lambda^T\Sigma^{-1}\lambda\right)=\sum_{\alpha\in\mathbb{N}_0^n}\frac{\lambda^\alpha}{\alpha!} H_\alpha(x;\Sigma).
	\end{equation}
	
	Write $y:=(y_1^T,\cdots,y_n^T)^T$. Denote by $\Gamma$ and $\Sigma$ the covariance matrices of $$(B_{t_1}^{H,(1)},\cdots,B_{t_n}^{H,(1)})\quad\text{and}\quad(B_{t_1}^{H,(1)},B_{t_2}^{H,(1)}-B_{t_1}^{H,(1)},\cdots,B_{t_n}^{H,(1)}-B_{t_{n-1}}^{H,(1)}),$$
	respectively. Then, the covariance matrix of $(B_{t_1}^H,B_{t_2}^H-B_{t_1}^H,\cdots,B_{t_n}^H-B_{t_{n-1}}^H)$ is $\boldsymbol{\Sigma}:=\Sigma\otimes I_d$. Let $\{e_i^{(n)}\}_{i=1}^n$ and $\{e_i^{(d)}\}_{i=1}^d$ denote the standard bases of $\R^n$ and $\R^d$, respectively. Set $d_i:=e_i^{(n)}-e_{i+1}^{(n)}$ for $i=1,\cdots,n-1$ and $d_n:=e_n^{(n)}$, and introduce $D:=(d_1,\cdots,d_n)$, $D_\alpha:=(d_1\otimes e_{\alpha_1}^{(d)},\cdots,d_n\otimes e_{\alpha_n}^{(d)})$, $z:=D_\alpha^T\boldsymbol{\Sigma}^{-1}y$ and $M:=D_\alpha^T\boldsymbol{\Sigma}^{-1}D_\alpha$.
	\begin{lemma}\label{sec2-lm.6} 
		Let $\alpha_0\in\mathbb{N}_0^n$ be the all-ones vector. It holds that  
		\begin{equation} 
			\mathcal{D}_1^{(\alpha_1)}\cdots\mathcal{D}_n^{(\alpha_n)}q_t(y)=(-1)^n q_t(y)H_{\alpha_0}(M^{-1}z;M^{-1}). 
		\end{equation} 
	\end{lemma} 
	\begin{proof} 
		Define $\Phi(\lambda)=q_t(y-D_\alpha\lambda)$ for $\lambda\in\R^n$. By a direct computation, we have 
		$$\partial_{\lambda_1}\cdots\partial_{\lambda_n}\Phi(\lambda)\big|_{\lambda=0}=(-1)^n\mathcal{D}_1^{(\alpha_1)}\cdots\mathcal{D}_n^{(\alpha_n)}q_t(y).$$  
		Since 
		$$\Phi(\lambda)=q_t(y)\exp\left(\lambda^Tz-\frac12 \lambda^TM\lambda\right),$$
		the generating function formula \eqref{sec2-eq.3} completes the proof. 
	\end{proof} 
	
	Define $L=\mathrm{diag}\{\ell_1^H,\cdots,\ell_n^H\}$, where $\ell_i$ is the same quantity as in Proposition \ref{sec2-prop.4}. 
	\begin{lemma}\label{sec2-lm.2}
		For $H<1/2$, we have $\Vert L^TML\Vert_2\leq c_H$, where $\Vert\cdot\Vert_2$ denotes the spectral norm and $c_H=\frac1{1-2^{2H-1}}$.
	\end{lemma}
	\begin{proof}
		Note that $\Sigma=D\Gamma D^T$. Hence $D^T\Sigma^{-1}D=\Gamma^{-1}$. Moreover, $M_{ij}=(\Gamma^{-1})_{ij}\mathbbm{1}_{\{\alpha_i=\alpha_j\}}$. For $k=1,\cdots,d$, define the diagonal projection matrix
		$P_k:=\mathrm{diag}\{\mathbbm{1}_{\{\alpha_1=k\}},\cdots,\mathbbm{1}_{\{\alpha_n=k\}}\}$.
		Then, $M=\sum_{k=1}^dP_k\Gamma^{-1}P_k$. Given that $L$ commutes with $P_k$, it suffices to show that $\Gamma\succeq c_H^{-1}L^2$. Indeed, this yields
		$$L^T\Gamma^{-1}L\preceq c_HI_n.$$
		Consequently,
		$$L^TML=\sum_{k=1}^dP_k^TL^T\Gamma^{-1}LP_k\preceq c_H\sum_{k=1}^dP_k=c_HI_n,$$
		and therefore $\Vert L^TML\Vert_2\leq c_H$.
		
		Consider $x^T(\Gamma-c_H^{-1}L^2)x=x^TD^{-1}\Sigma(D^T)^{-1}x-c_H^{-1}x^TL^2x$. Let $u=(D^T)^{-1}x$. Then, $u_i=x_i+x_{i+1}+\cdots+x_n$ and $x^T\Gamma x=u^T\Sigma u$. From the anti-persistent property of fBm with $H<1/2$, we have $\Sigma_{ij}<0$ for $i\neq j$. Let 
		\begin{align*} 
			r_i&:=\sum_{j=1}^n\Sigma_{ij}=\mathrm{Cov}(B_{t_i}^{H,(1)}-B_{t_{i-1}}^{H,(1)},B_{t_n}^{H,(1)})\\ 
			&=\frac12\left[t_i^{2H}-t_{i-1}^{2H}+(t_n-t_{i-1})^{2H}-(t_n-t_i)^{2H}\right]\geq 0. 
		\end{align*} 
		In particular, we have
		$$r_n=\frac12\left[t_n^{2H}-t_{n-1}^{2H}+\Delta_n^{2H}\right]\geq \frac12 \ell_n^{2H}.$$
		If $\Delta_i\leq\Delta_{i+1}$, then 
		\begin{align*}
			-\Sigma_{i,i+1}&=\frac12\left[\Delta_i^{2H}+\Delta_{i+1}^{2H}-(\Delta_i+\Delta_{i+1})^{2H}\right]\\
			&=\frac{\ell_i^{2H}}2\left[1+\left(\frac{\Delta_{i+1}}{\Delta_i}\right)^{2H}-\left(1+\frac{\Delta_{i+1}}{\Delta_i}\right)^{2H}\right]\geq (1-2^{2H-1})\ell_i^{2H}.
		\end{align*}
		For the case $\Delta_i\geq \Delta_{i+1}$, we have $-\Sigma_{i,i+1}\geq (1-2^{2H-1})\ell_i^{2H}$ by the same reasoning. Therefore,
		\begin{align*}
			u^T\Sigma u&=\sum_{i=1}^nr_i u_i^2+\sum_{1\leq i<j\leq n}(-\Sigma_{ij})(u_i-u_j)^2\\
			&\geq r_nu_n^2+\sum_{i=1}^{n-1}(-\Sigma_{i,i+1})(u_i-u_{i+1})^2\\
			&\geq (1-2^{2H-1})\left(\ell_n^{2H}x_n^2+\sum_{i=1}^{n-1}x_i^2\ell_i^{2H}\right)=c_H^{-1}x^TL^2x.
		\end{align*}
		Thus, we have proved that $\Gamma\succeq c_H^{-1}L^2$. The proof is complete.
	\end{proof}
	
	With the above preparations in place, we begin the proof of Proposition \ref{sec2-prop.4}.
	\\\textit{Proof of Proposition \ref{sec2-prop.4}}. Let $\tilde{z}=Lz$ and $\hat z=(L^TML)^{-1}\tilde{z}$. Then, from the generating function formula \eqref{sec2-eq.3}, it holds that $$H_{\alpha_0}(M^{-1}z;M^{-1})=H_{\alpha_0}(\hat z;(L^TML)^{-1})\prod_{j=1}^n\ell_j^{-H}.$$
	Hence, it remains to estimate $H_{\alpha_0}(\hat z;(L^TML)^{-1})$. By Cauchy's integral formula and \eqref{sec2-eq.3}, we have
	$$H_{\alpha_0}(\hat z;(L^TML)^{-1})=\frac1{(2\pi\imath)^{|\alpha_0|}}\oint_{|\lambda_1|=R_1}\cdots\oint_{|\lambda_n|=R_n}\frac{\exp\left(\lambda^T\tilde{z}-\frac12 \lambda^TL^TML\lambda\right)}{\lambda_1^2\cdots\lambda_n^2}\d \lambda.$$
	By Lemma \ref{sec2-lm.2}, we have $|\lambda^TL^TML\lambda|\leq c_H|\lambda|^2$ and
	\begin{align*} 
		|\lambda^T\tilde{z}|&\leq C_\varepsilon|\lambda|^2+\varepsilon|Lz|^2\\ 
		&= C_\varepsilon |\lambda|^2+\varepsilon |L^TD_\alpha^T\boldsymbol{\Sigma}^{-\frac12}\boldsymbol{\Sigma}^{-\frac12}y|^2\\ 
		&\leq C_\varepsilon |\lambda|^2+\varepsilon \Vert L^TML\Vert_2\cdot y^T\boldsymbol{\Sigma}^{-1}y.  
	\end{align*} 
	Choosing $R_1=\cdots=R_n=1$ and taking $\varepsilon$ with $c_H\varepsilon<1/4$, we obtain
	$$|H_{\alpha_0}(\hat z;(L^TML)^{-1})|\leq C^n\exp\left(\frac14 y^T\boldsymbol{\Sigma}^{-1} y\right).$$
	Using Lemma \ref{sec2-lm.6} yields that
	$$\left|\mathcal{D}_1^{(\alpha_1)}\cdots\mathcal{D}_n^{(\alpha_n)}q_t(y_1,\cdots,y_n)\right|\leq \frac{C^n(2\pi)^{-\frac{nd}{2}}}{(\det(\Sigma))^{d/2}}\prod_{j=1}^n\ell_j^{-H}\exp\left(-\frac14y^T\boldsymbol{\Sigma}^{-1} y\right).$$
	From the identity
	$$\det(\Sigma)=\mathrm{Var}(B_{t_1}^{H,(1)})\mathrm{Var}(B_{t_2}^{H,(1)}-B_{t_1}^{H,(1)}\mid B_{t_1}^{H,(1)})\cdots\mathrm{Var}(B_{t_n}^{H,(1)}-B_{t_{n-1}}^{H,(1)}\mid B_{t_1}^{H,(1)},\cdots,B_{t_{n-1}}^{H,(1)})$$
	(see e.g. \cite[Lemma B.15]{das2025exponential}) and the local non-determinism, we have 
	$$\det(\Sigma)\geq C_H^n\prod_{j=1}^n\Delta_j^{2H}.$$ 
	Hence, it remains to prove $y^T\boldsymbol{\Sigma}^{-1}y\geq \frac12\sum_{j=1}^n|y_j|^2/\Delta_j^{2H}$. Let $\Lambda=\mathrm{diag}\{\Delta_1^{H},\cdots,\Delta_n^{H}\}$ and $R:=\Lambda^{-1}\Sigma\Lambda^{-1}$ be the correlation matrix of increments of fBm. It follows from the anti-persistent property of fBm with $H<1/2$ that $R$ can be decomposed into $R=I-A$, where $A$ has non-negative entries. For $0\neq u\in\R^n$, set $v=(|u_1|,\cdots,|u_n|)^T$. Then, 
	$$u^T(I+A)u\geq v^T(I-A)v> 0.$$
	Hence, we have that $I+A$ is positive definite, which, together with the positive definiteness of $R=I-A$, implies that $\rho(A)\leq 1$ and hence $\rho(R)\leq 2$. As a result, writing $y^{(k)}=(y_1^{(k)},\cdots,y_n^{(k)})^T$,  
	$$y^T\boldsymbol{\Sigma}^{-1} y=\sum_{k=1}^d (y^{(k)})^T\Sigma^{-1}y^{(k)}\geq\frac12\sum_{k=1}^d\sum_{j=1}^n\frac{|y_j^{(k)}|^2}{\Delta_j^{2H}}=\frac12\sum_{j=1}^n\frac{|y_j|^2}{\Delta_j^{2H}}.$$ 
	The proof is complete. 

	\begin{remark}\label{sec2-rm.9}
		Combining \eqref{sec2-eq.1} with Proposition \ref{sec2-prop.4} yields
		\begin{equation} 
			|\partial_{x_1^{(\alpha_1)}}\cdots \partial_{x_n^{(\alpha_n)}}p_t(x_1,\cdots,x_n)|\leq C^n\prod_{j=1}^n\Delta_j^{-Hd}\ell_j^{-H}\exp\left(-\frac{|x_j-x_{j-1}|^2}{8\Delta_j^{2H}}\right).
		\end{equation} 
	\end{remark}
	
	\begin{remark}
		In fact, one can follow the same approach in \cite[Theorem 2.4]{lou2017local} to obtain more general derivative estimates. However, it is more delicate to show that the constant in Remark \ref{sec2-rm.9} is in the form of $C^n$.
	\end{remark}
	
	\section{Sketch of the Proof}\label{sec3}
	Let $\{b_n\}_{n=1}^\infty\subset C_c^\infty([0,T]\times\R^d)$ be an approximating sequence of $b$ such that $\lim_{n\rightarrow\infty}\Vert b_n-b\Vert_{L^p_xL^q_t}=0$ and $\sup_n\Vert b_n\Vert_{L^p_xL^q_t}\leq \Vert b\Vert_{L^p_xL^q_t}$. By a standard result on SDEs, for each smooth coefficient $b_n$, there exists a unique strong solution $X^n$ to the SDE
	\begin{equation}\label{approx-eq}
		\d X_t^n=b_n(t,X_t^n)\d t+\d B_t^H,\qquad X_0^n=x_0\in\R^d.
	\end{equation}
	
	The key idea of our approach is to apply Theorem \ref{sec2-thm.2} to conclude that $\{X_t^n\}_{n=1}^\infty$ is relatively compact in $L^2(\Omega)$ for each fixed $t$. This allows us, together with a diagonal argument and suitable time-regularity estimates, to extract a subsequence converging to a limit process $X$, which will be shown in Section \ref{sec6} to be a strong solution to \eqref{sec1-eq.1}. We are therefore reduced to verifying the two conditions in Theorem \ref{sec2-thm.2}.  
	
	The first condition of Theorem \ref{sec2-thm.2} is not difficult to verify. Indeed, Lemma 3.11 and Corollary 3.14 in \cite{butkovsky2023weak} give the following estimate:
	\begin{lemma}\label{sec3-lm.1}
		Assume that $\kappa_H+H/q>0$. Let $b:[0,T]\times \R^d\rightarrow\R^d$ be a continuously differentiable function with compact support, and let $X$ be the strong solution to the corresponding SDE. Suppose that $\Vert b\Vert_{L_x^pL_t^q}\leq M$. Then, for $\varphi\in L_x^pL_t^q$ and $m\geq2$, there exists a constant $C=C(H,d,p,q,T,m,M)$ such that
		$$\left\Vert\int_s^t\varphi(r,X_r)\d r\right\Vert_{L^m(\Omega)}\leq C\Vert \varphi\Vert_{L_x^pL_t^q}(t-s)^{1-H-\frac{Hd}p-\frac{1-H}q}.$$
	\end{lemma}
	
	\begin{proof}
		Set $\psi_t:=X_t-x_0-B_t^H$. By Corollary 3.14 in \cite{butkovsky2023weak}, together with the assumption $\kappa_H+H/q>0$ and $\Vert b\Vert_{L_x^pL_t^q}\leq M$, we have $\Vert[\psi]_{1\text{-var};[0,T]}\Vert_{L^m(\Omega)}\leq C$. Then, Lemma 3.11 in \cite{butkovsky2023weak} yields
		\begin{align*}
			\left\Vert\int_s^t\varphi(r,X_r)\d r\right\Vert_{L^m(\Omega)}
			&\leq C\Vert\varphi\Vert_{L_x^pL_t^q}(t-s)^{1-\frac1q-\frac{Hd}{p}}
			\left[1+(t-s)^{-H(1-\frac1q)}\right]\\
			&\leq C\Vert\varphi\Vert_{L_x^pL_t^q}
			(t-s)^{1-H-\frac{Hd}{p}-\frac{1-H}{q}}.
		\end{align*}
		The proof is complete.
	\end{proof}
	
	Therefore, the main difficulty lies in verifying the second condition. According to Remark \ref{sec2-rm.3}, it suffices to consider $\D^H X_t^n$. By the chain rule for the Malliavin derivative, we have
	$$\D^H_\theta X_t^n=I_d+\int_{\theta}^t\nabla b_n(s,X_s^n)\D^H_\theta X_s^n\d s$$
	for a.e. $0<\theta<t$, and $\D^H_\theta X_t^n=0$ for a.e. $\theta>t$. By iteration, we obtain
	\begin{equation}\label{Malliavin-eq}
		\D^H_\theta X_t^n=\mathbbm{1}_{[0,t]}(\theta) I_d+\sum_{m=1}^\infty\int_{\Delta_{\theta,t}^m}\nabla b_n(s_m,X_{s_m}^n)\cdots \nabla b_n(s_1,X_{s_1}^n)\d s.
	\end{equation}
	It is worth pointing out that, throughout the sequel, all the Malliavin derivatives are defined in the original probability space before any change of measure.
	
	Define $v_t^n:=K_H^{-1}(\int_0^\cdot b_n(s,X_s^n)\d s)(t)$, and consider the exponential martingale 
	$$\xi_T^n:=\exp\left(-\int_0^T \langle v_s^n,\d W_s\rangle-\frac12\int_0^T|v_s^n|^2\d s\right).$$ 
	Since $b_n$ is bounded for each $n$, the strong Novikov condition holds. By Theorem \ref{sec2-theorem.1}, we can define a probability measure $\mathbb{Q}_n$ by $\d \mathbb{Q}_n=\xi_T^n\d \mathbb{P}$, under which $\tilde{W}_t^n:=W_t+\int_0^tv_s^n \d s$ is a Brownian motion and the shifted process 
	$$\tilde{B}^{H,n}_t:=B_t^H+\int_0^tb_n(s,X_s^n)\d s=X_t^n-x_0$$ 
	is an fBm. Set $\eta_T^n:=\d\mathbb{P}/\d\mathbb{Q}_n$. Given that $X^n=x_0+\tilde{B}^{H,n}$ under $\mathbb{Q}_n$ and $\tilde{B}^{H,n}$ under $\mathbb{Q}_n$ has the same law as $B^H$ under $\mathbb{P}$, by a change of measure from $\mathbb{P}$ to $\mathbb{Q}_n$ and the Cauchy-Schwarz inequality, we have $$\sup_n\E[\Vert \D_\cdot^H X_t^n\Vert_{\mathcal{H}_\beta}^2]\leq \left(\sup_n\E^{\mathbb{Q}_n}[|\eta_T^n|^2]\right)^{\frac12}\left(\sup_n\E\left[\Vert \mathbbm{1}_{[0,t]}(\cdot) I_d+\mathcal{I}_{\cdot,t}^n\Vert_{\mathcal{H}_\beta}^4\right]\right)^{\frac12},$$ where 
	\begin{equation}\label{sec3-eq.2}
		\mathcal{I}_{\theta,t}^n:=\sum_{m=1}^\infty\int_{\Delta_{\theta,t}^m}\nabla b_n(s_m,x_0+B_{s_m}^H)\cdots\nabla b_n(s_1,x_0+B_{s_1}^H)\d s
	\end{equation} 
	for $\theta<t$. Therefore, if the uniform bound for $\E^{\mathbb{Q}_n}[|\eta_T^n|^2]$ is available, it would remain to estimate $\sup_n\E[\Vert \mathbbm{1}_{[0,t]}(\cdot)I_d+\mathcal{I}_{\cdot,t}^n\Vert_{\mathcal{H}_\beta}^4]$. However, $\sup_n\E^{\mathbb{Q}_n}[|\eta_T^n|^2]$ might fail under the whole regime of the LPS condition (see Proposition \ref{sec4-prop.3}). We thus truncate the $L^2$-norm of $v^n$. On the truncation set, $\eta_T^n$ admits moments of all orders with the required uniform control, which allows us to apply the Cauchy-Schwarz inequality safely.
	
	Recall that, for $f\in L^2([0,T];E)$, the spectral decomposition with respect to the Dirichlet basis $\{\varphi_k\}_{k=1}^\infty$ introduced in Definition \ref{sec1-def.5} is given by $f=\sum_{k=1}^\infty f_k\varphi_k$, where $f_k=\int_0^Tf(t)\varphi_k(t)\d t$. Set
	$$\Lambda_0:=\{k:\sqrt{\lambda_k}<2\},\qquad
	\Lambda_i:=\{k:2^i\leq\sqrt{\lambda_k}<2^{i+1}\},\qquad i\geq1,$$
	and define
	$$\Pi_i f(t):=\sum_{k\in\Lambda_i}f_k\varphi_k,\qquad i\geq0.$$
	By Parseval's identity, we have $\Vert f\Vert_{L_t^2}^2=\sum_{i=0}^\infty\Vert\Pi_i f\Vert_{L_t^2}^2$. Apply the above decomposition to $v_\cdot^n$, and set $E_i^n:=\Vert\Pi_i v_\cdot^n\Vert_{L_t^2}^2$ and $R_{i,L}:=L2^{-2\eta i}$ ($\eta>0$ will be specified precisely later). We now define the truncation function by
	$$\Theta_L^n:=\prod_{j=0}^\infty\chi\left(\frac{E_j^n}{R_{j,L}}\right),$$
	where $\chi\in C_c^\infty(\R),\mathrm{supp}\chi\subset[-1,2],0\leq\chi\leq 1$ and $\chi(x)=1$ for $x\in[0,1]$.
	
	We then introduce $\{F_L^n:=X_t^n\Theta_L^n\}_{n=1}^\infty$ and claim that the relative compactness of $\{F_L^n\}_{n=1}^\infty$ implies that of $\{X_t^n\}_{n=1}^\infty$ if $\sup_n\E[\Vert v_\cdot^n\Vert_{\mathbbm{H}^\delta}^2]<\infty$ for some $\delta>0$.
	\begin{lemma}\label{sec3-lm.2}
		Assume that $\sup_n\E[\Vert v_\cdot^n\Vert_{\mathbbm{H}^\delta}^2]<\infty$ for some $\delta>0$. Let $0<\eta<\delta$. Suppose that, for all $L>0$, $\{F_L^n\}_{n=1}^\infty$ is relatively compact in $L^2(\Omega)$. Then, $\{X_t^n\}_{n=1}^\infty$ is relatively compact in $L^2(\Omega)$.
	\end{lemma}
	\begin{proof}
		Firstly, we note that, for all $i\geq0$ and $f\in \mathbb{H}^\delta$,
		\begin{align}\label{sec3-eq.3}
			\Vert \Pi_i f\Vert_{L_t^2}^2&=\sum_{k\in\Lambda_i}|f_k|^2=\sum_{k\in\Lambda_i}\lambda_k^{-\delta}\lambda_k^\delta|f_k|^2\nonumber\\
			&\leq c_{T,\delta}2^{-2\delta i}\sum_{k\in\Lambda_i}\lambda_k^\delta|f_k|^2\leq c_{T,\delta}2^{-2\delta i}\Vert f\Vert_{\mathbb{H}^\delta}^2,
		\end{align}
		where $c_{T,\delta}:=1\vee\lambda_1^{-\delta}$. As a result, we have
		$$\sup_n\E[E_i^n]\leq c_{T,\delta}2^{-2\delta i}\sup_n \E[\Vert v_\cdot^n\Vert_{\mathbb{H}^\delta}^2]\leq C2^{-2\delta i}.$$
		By Markov's inequality, 
		$$\sup_n\mathbb{P}(E_i^n\geq R_{i,L})\leq\frac{\sup_n\E[E_i^n]}{L2^{-2\eta i}}\leq CL^{-1}2^{-2(\delta-\eta)i},$$
		which leads to
		$$\sup_n\mathbb{P}(\Theta_L^n\neq 1)\leq \sum_{i=0}^\infty \sup_n\mathbb{P}(E_i^n\geq R_{i,L})\leq CL^{-1}\frac{1}{1-2^{-2(\delta-\eta)}}\rightarrow 0$$
		as $L$ tends to infinity. Then, by Lemma \ref{sec3-lm.1} and the Cauchy-Schwarz inequality, we obtain
		\begin{align*}
			\sup_n\E[|F_L^n-X_t^n|^2]&\leq \sup_n\E[|X_t^n|^2\mathbbm{1}_{\{\Theta_L^n\neq 1\}}]\\
			&\leq  \left(\sup_n\E[|X_t^n|^4]\right)^{\frac12}\left(\sup_n\mathbb{P}(\Theta_L^n\neq1)\right)^{\frac12}\overset{L\rightarrow\infty}\longrightarrow0.
		\end{align*}
		Hence, the relative compactness of $\{F_L^n\}_{n=1}^\infty$ implies that of $\{X_t^n\}_{n=1}^\infty$. The proof is complete.
	\end{proof}
	
	Therefore, we can reduce the validation of the relative compactness of $\{X_t^n\}_{n=1}^\infty$ to that of $\{F_L^n\}_{n=1}^\infty$. By the product rule and the chain rule for the Malliavin derivative, we have
	$\D_\theta^H F_L^n=\Theta_L^n\D_\theta^HX_t^n+X_t^n\D_\theta^H\Theta_{L}^n$ and 
	\begin{equation}\label{sec3-eq.7}
		\D_\theta^H\Theta_L^n=\sum_{i=0}^\infty \frac{\D_\theta^HE_i^n}{R_{i,L}}\chi'\left(\frac{E_i^n}{R_{i,L}}\right)\prod_{j\neq i}\chi\left(\frac{E_j^n}{R_{j,L}}\right).
	\end{equation}
	Here, the Malliavin differentiability of $\Theta_L^n$ and the identity above will be justified in Lemma \ref{sec3-lm.4}. Since $\Vert v_\cdot^n\Vert_{L_t^2}^2=\sum_{j=0}^\infty E_j^n$,
	the support property of $\chi$ implies that both $\Theta_L^n$ and $\D_\cdot^H\Theta_L^n$ vanish outside
	$$A^n_L:=\left\{\Vert v_\cdot^n\Vert_{L_t^2}^2\leq \frac{2L}{1-2^{-2\eta}}\right\}.$$
	Hence, the Cauchy-Schwarz inequality yields
	\begin{align}\label{sec3-eq.4}
		&\E[\Vert \D_\cdot^H F_L^n\Vert_{\mathcal{H}_\beta}^2]
		=\E^{\mathbb{Q}_n}[\eta_T^n\Vert \D_\cdot^H F_L^n\Vert_{\mathcal{H}_\beta}^2]\nonumber\\
		&\qquad\leq
		\begin{aligned}[t]
			& C\E^{\mathbb{Q}_n}[\eta_T^n|\Theta_L^n|^2
			\Vert \mathbbm{1}_{[0,t]}(\cdot)I_d+\mathcal{I}_{\cdot,t}^n\Vert_{\mathcal{H}_\beta}^2]\\
			& \mathllap{{}+{}}C\E^{\mathbb{Q}_n}\left[\eta_T^n|x_0+\tilde{B}_t^H|^2
			\left\Vert \sum_{i=0}^\infty
			\frac{\D_\cdot^HE_i^n}{R_{i,L}}
			\chi'\left(\frac{E_i^n}{R_{i,L}}\right)
			\prod_{j\neq i}\chi\left(\frac{E_j^n}{R_{j,L}}\right)
			\right\Vert_{\mathcal{H}_\beta}^2\right]
		\end{aligned}\nonumber\\
		&\qquad\leq
		\begin{aligned}[t]
			& C\big(\E^{\mathbb{Q}_n}[|\eta_T^n|^2\mathbbm{1}_{A^n_L}]\big)^{\frac12}
			\big(\E[\Vert  \mathbbm{1}_{[0,t]}(\cdot)I_d+\mathcal{I}_{\cdot,t}^n\Vert_{\mathcal{H}_\beta}^4]\big)^{\frac12}\\
			& \mathllap{{}+{}}C
			\big(\E^{\mathbb{Q}_n}[|\eta_T^n|^4\mathbbm{1}_{A^n_L}]\big)^{\frac14}
			\big(\E[|x_0+B_t^H|^8]\big)^{\frac14}
		\end{aligned}\nonumber\\
		&\qquad\quad\times
		\left(\E^{\mathbb{Q}_n}\left[
		\left(\sum_{i=0}^\infty
		\frac{\Vert\D_\cdot^HE_i^n\Vert_{\mathcal{H}_\beta}}{R_{i,L}}
		\left|\chi'\left(\frac{E_i^n}{R_{i,L}}\right)\right|
		\right)^4\right]\right)^{\frac12}.
	\end{align}
	It is clear that $\E[|x_0+B_t^H|^8]$ is finite. Note that
	$$|\eta_T^n|^4=\exp\left(\int_0^T \langle 4v_s^n,\d\tilde{W}^{n}_s\rangle-\frac12\int_0^T|4v_s^n|^2\d s\right)\exp\left(6\Vert v_\cdot^n\Vert_{L_t^2}^2\right).$$
	By the strong Novikov condition, the first term on the right-hand side of the above equality is a martingale. Thus, we obtain
	$$\E^{\mathbb{Q}_n}[|\eta_T^n|^4\mathbbm{1}_{A^n_L}]\leq\exp\left(\frac{12L}{1-2^{-2\eta}}\right).$$
	By Jensen's inequality, $\E^{\mathbb{Q}_n}[|\eta_T^n|^2\mathbbm{1}_{A^n_L}]$ is also bounded by $e^{cL}$ for some constant $c>0$. Moreover, by H\"older's inequality, we have
	\begin{align}\label{sec3-eq.5}
		&\E^{\mathbb{Q}_n}\left[\left(\sum_{i=0}^\infty \frac{\Vert\D_\cdot^HE_i^n\Vert_{\mathcal{H}_\beta}}{R_{i,L}}\left|\chi'\left(\frac{E_i^n}{R_{i,L}}\right)\right|\right)^4\right]\nonumber\\
		&\qquad=\sum_{i_1,\cdots,i_4=0}^\infty\E^{\mathbb{Q}_n}\left[\prod_{k=1}^4\frac{\Vert\D_\cdot^HE_{i_k}^n\Vert_{\mathcal{H}_\beta}}{R_{i_k,L}}\left|\chi'\left(\frac{E_{i_k}^n}{R_{i_k,L}}\right)\right|\right]\nonumber\\
		&\qquad\leq\sum_{i_1,\cdots,i_4=0}^\infty\prod_{k=1}^4\left(\E^{\mathbb{Q}_n}\left[\left|\frac{\chi'\left(\frac{E_{i_k}^n}{R_{i_k,L}}\right)}{R_{i_k,L}}\right|^4\Vert \D_\cdot^H E_{i_k}^n\Vert_{\mathcal{H}_\beta}^4\right]\right)^{\frac14}.
	\end{align}
	
	The preceding estimates \eqref{sec3-eq.4} and \eqref{sec3-eq.5} lead to the following sufficient criterion for the relative compactness of $\{F_L^n\}_{n=1}^\infty$.
	\begin{lemma}\label{sec3-lm.3}
		Assume that for some $\delta>0$ and $\beta>0$, the following three estimates hold:
		\begin{enumerate}
			\item $\sup_n\E[\Vert \mathcal{I}^n_{\cdot,t}\Vert_{\mathcal{H}_\beta}^4]<\infty;$
			\item $\sup_n\E^{\mathbb{Q}_n}[\Vert v_\cdot^n\Vert_{\mathbbm{H}^\delta}^{2m}]<\infty\qquad\forall\,m\in\mathbb{N};$
			\item $\sup_n\E^{\mathbb{Q}_n}[\Vert \mathbf{D}^H v^n\Vert_{\mathbbm{H}^{-2}([0,T];\mathcal{H}_\beta)}^8]<\infty.$
		\end{enumerate}
		Set $0<\eta<\delta$. Then, $\{F^n_L\}_{n=1}^\infty$ is relatively compact in $L^2(\Omega)$.
	\end{lemma}
	\begin{proof}
		By the fact that $\mathcal{K}_H^*\mathbbm{1}_{[0,t]}(\cdot)=K_H(t,\cdot)$ and Lemma \ref{appendix-lm.6}, the first condition gives $\sup_n\E[\Vert \mathbbm{1}_{[0,t]}(\cdot)I_d+\mathcal{I}_{\cdot,t}^n\Vert_{\mathcal{H}_\beta}^4]<\infty$. Note that for all $i\geq0$ and $f\in L^2([0,T];\mathcal{H}_\beta)$,
		\begin{align}\label{sec3-eq.6}
			\Vert \Pi_i f\Vert_{L^2([0,T];\mathcal{H}_\beta)}^2&=\sum_{k\in\Lambda_i}\Vert f_k\Vert_{\mathcal{H}_\beta}^2=\sum_{k\in\Lambda_i}\lambda_k^2\lambda_k^{-2}\Vert f_k\Vert_{\mathcal{H}_\beta}^2\nonumber\\
			&\leq 2^{4i+4}\sum_{k\in\Lambda_i}\lambda_k^{-2}\Vert f_k\Vert_{\mathcal{H}_\beta}^2\leq 2^{4i+4}\Vert f\Vert_{\mathbb{H}^{-2}([0,T];\mathcal{H}_\beta)}^2.
		\end{align}
		In addition, by the definition of $\Pi_i$ and the chain rule for Malliavin derivative, we have
		$$\D_\theta^H E_i^n=2\int_0^T\langle \Pi_iv_s^n,\Pi_i\D_\theta^Hv_s^n\rangle\d s.$$
		Then, combining the Cauchy-Schwarz inequality with the estimate \eqref{sec3-eq.6} yields
		\begin{align*}
			\Vert \D_\cdot^H E_i^n\Vert_{\mathcal{H}_\beta}&\leq 2\Vert \Pi_i v_\cdot^n\Vert_{L_t^2}\Vert \Pi_i\D^H v^n\Vert_{L^2([0,T];\mathcal{H}_\beta)}\\
			&\leq 2^{2i+3}\sqrt{E_i^n}\Vert \D^Hv^n\Vert_{\mathbb{H}^{-2}([0,T];\mathcal{H}_\beta)}.
		\end{align*}
		Moreover, by Markov's inequality, \eqref{sec3-eq.3}, and the second condition, we obtain
		\begin{equation}\label{sec3-eq.8}
			\sup_n\mathbb{Q}_n(E_i^n\geq R_{i,L})\leq \frac{2^{-2m\delta i}\sup_n\E^{\mathbb{Q}_n}[\Vert v_\cdot^n\Vert_{\mathbbm{H}^\delta}^{2m}]}{L^m2^{-2m\eta i}}\leq CL^{-m}2^{-2m(\delta-\eta)i}
		\end{equation}
		for all $m\in\mathbb{N}$. Note that $\chi'(E_i^n/R_{i,L})$ can be nonzero only if $R_{i,L}\leq E_i^n\leq 2R_{i,L}$. Then, choosing $m=M$ such that $M(\delta-\eta)-4(2+\eta)>0$ and applying Cauchy-Schwarz inequality, we have
		\begin{align*}
			&\sup_n\E^{\mathbb{Q}_n}\left[\left|\frac{\chi'\left(\frac{E_i^n}{R_{i,L}}\right)}{R_{i,L}}\right|^4\Vert \D_\cdot^H E_i^n\Vert_{\mathcal{H}_\beta}^4\right]\\
			&\qquad\leq C2^{8i}R_{i,L}^{-2}\sup_n\E^{\mathbb{Q}_n}\left[\Vert \D^Hv^n\Vert_{\mathbb{H}^{-2}([0,T];\mathcal{H}_\beta)}^4\mathbbm{1}_{\{E_i^n\geq R_{i,L}\}}\right]\\
			&\qquad\leq C 2^{4(2+\eta)i}\left(\sup_n\E^{\mathbb{Q}_n}\left[\Vert \D^Hv^n\Vert_{\mathbb{H}^{-2}([0,T];\mathcal{H}_\beta)}^8\right]\right)^{\frac12}\left(\sup_n\mathbb{Q}_n(E_i^n\geq R_{i,L})\right)^{\frac12}\\
			&\qquad\leq C 2^{-M(\delta-\eta)i+4(2+\eta)i},
		\end{align*}
		where the last inequality follows from the third condition and \eqref{sec3-eq.8}. As a consequence, by \eqref{sec3-eq.5}, we have
		$$\sup_n\E^{\mathbb{Q}_n}\left[\left(\sum_{i=0}^\infty \frac{\Vert\D_\cdot^HE_i^n\Vert_{\mathcal{H}_\beta}}{R_{i,L}}\left|\chi'\left(\frac{E_i^n}{R_{i,L}}\right)\right|\right)^4\right]\leq C\prod_{k=1}^4\left(\sum_{i_k=0}^\infty 2^{-\frac{M}{4}(\delta-\eta)i_k+(2+\eta)i_k}\right)<\infty.$$
		Consequently, it follows from \eqref{sec3-eq.4} that $\sup_n\E[\Vert \D_\cdot^H F_L^n\Vert_{\mathcal{H}_\beta}^2]<\infty$. Note that $\sup_n\E[|F_L^n|^2]\leq \sup_n \E[|X_t^n|^2]$. In conjunction with Lemma \ref{sec3-lm.1} and Theorem \ref{sec2-thm.2}, we conclude that $\{F_L^n\}_{n=1}^\infty$ is relatively compact in $L^2(\Omega)$. The proof is complete.
	\end{proof}
	
	We now supply a rigorous justification for the Malliavin differentiability of $\Theta_L^n$ and the identity \eqref {sec3-eq.7}.
	\begin{lemma}\label{sec3-lm.4}
		Assume that the second and third conditions in Lemma \ref{sec3-lm.3} are satisfied. Then, for each $n$, $\Theta_L^n$ is Malliavin differentiable and the identity \eqref{sec3-eq.7} holds.
	\end{lemma}
	\begin{proof}
		For $N\in\mathbb{N}_0$, set
		$$\Theta_{L,N}^n:=\prod_{j=0}^N\chi\left(\frac{E_j^n}{R_{j,L}}\right).$$
		Since $0\leq\chi\leq1$, the sequence $\{\Theta_{L,N}^n\}_{N\geq0}$ converges pointwise to $\Theta_L^n$ and is bounded by one. Hence, by the dominated convergence theorem,
		$$\Theta_{L,N}^n\xrightarrow{N\rightarrow\infty}\Theta_L^n\qquad\text{in }L^2(\Omega).$$
		For each $N$, the product rule and the chain rule give
		$$\D_\theta^H\Theta_{L,N}^n=\sum_{i=0}^N \frac{\D_\theta^HE_i^n}{R_{i,L}}\chi'\left(\frac{E_i^n}{R_{i,L}}\right)\prod_{j\neq i}^N\chi\left(\frac{E_j^n}{R_{j,L}}\right).$$
		Set
		$$U_\theta^n:=\sum_{i=0}^\infty \frac{\D_\theta^HE_i^n}{R_{i,L}}\chi'\left(\frac{E_i^n}{R_{i,L}}\right)\prod_{j\neq i}\chi\left(\frac{E_j^n}{R_{j,L}}\right),$$
		and
		$$V_L^n:=\sum_{i=0}^\infty\frac{\Vert\D_\cdot^HE_i^n\Vert_{\mathcal{H}_\beta}}{R_{i,L}}\left|\chi'\left(\frac{E_i^n}{R_{i,L}}\right)\right|.$$
		By \eqref{sec3-eq.8} and the Borel-Cantelli lemma, for $\mathbb{Q}_n$-almost every $\omega$, there exists $N(\omega)$ such that $E_i^n(\omega)<R_{i,L}$ for all $i\geq N(\omega)$. Since $\chi'=0$ on $(0,1)$, for all $N\geq N(\omega)$, we have 
		$\D^H\Theta_{L,N}^n(\omega)=U^n(\omega)$ in $\mathcal{H}_\beta$. As a consequence,
		$$\Vert\D_\cdot^H\Theta_{L,N}^n-U_\cdot^n
		\Vert_{\mathcal{H}_\beta}\xrightarrow{N\to\infty}0,
		\quad\mathbb{Q}_n\text{-a.s.}$$
		Note that
		$$\Vert\D^H\Theta_{L,N}^n-U^n\Vert_{\mathcal{H}_\beta}\leq2V_L^n.$$
		\eqref{sec3-eq.5} and the estimate obtained in the proof of Lemma \ref{sec3-lm.3}, which does not rely on \eqref{sec3-eq.7}, show that $V_L^n\in L^4(\mathbb{Q}_n)$. Therefore, by the dominated convergence theorem,
		$$\E^{\mathbb{Q}_n}\left[
		\Vert\D_\cdot^H\Theta_{L,N}^n-U_\cdot^n\Vert_{\mathcal{H}_\beta}^4
		\right]\xrightarrow{N\rightarrow\infty}0.$$
		Given that $\eta_T^n\in L^2(\mathbb{Q}_n)$, the Cauchy-Schwarz inequality yields
		\begin{align*}
			\E\left[\Vert\D_\cdot^H\Theta_{L,N}^n-U_\cdot^n
			\Vert_{\mathcal{H}_\beta}^2
			\right]\leq\left(\E^{\mathbb{Q}_n}[|\eta_T^n|^2]\right)^{\frac12}\left(\E^{\mathbb{Q}_n}\left[\Vert\D_\cdot^H\Theta_{L,N}^n-U_\cdot^n\Vert_{\mathcal{H}_\beta}^4\right]\right)^{\frac12}\rightarrow0
		\end{align*}
		as $N$ tends to infinity, which implies that $\D^H\Theta_{L,N}^n$ converges to $U^n$ in $L^2(\Omega;\mathcal{H}_\beta)$. Since
		$\mathcal{H}_\beta\hookrightarrow\mathcal{H}$ and the Malliavin derivative operator
		$$\D^H:\mathbb{D}_H^{1,2}\subset L^2(\Omega)
		\longrightarrow L^2(\Omega;\mathcal{H})$$
		is closed, we conclude that $\Theta_L^n\in\mathbb{D}_H^{1,2}$ and
		$\D_\theta^H\Theta_L^n=U_\theta^n$ for almost every $\theta\in[0,T]$, $\mathbb{P}$-a.s. The proof is complete.
	\end{proof}
	
	In summary, Lemmas \ref{sec3-lm.2}-\ref{sec3-lm.4} reveal that establishing the relative compactness of $\{X_t^n\}_{n=1}^\infty$ for each $t$ requires
	\begin{itemize}
		\item bounding $\sup_n\E^{\mathbb{Q}_n}[\Vert v_\cdot^n\Vert_{\mathbb{H}^\delta}^{2m}]$ and $\sup_n \E[\Vert v_\cdot^n\Vert_{\mathbb{H}^\delta}^2]$ for some $\delta>0$;
		\item deriving uniform estimates for $\E[\Vert \mathcal{I}^n_{\cdot,t}\Vert_{\mathcal{H}_\beta}^4]$ and $\E^{\mathbb{Q}_n}[\Vert \D^H v^n\Vert_{\mathbb{H}^{-2}([0,T];\mathcal{H}_\beta)}^8]$.
	\end{itemize}
	Sections \ref{sec4} and \ref{sec5} are devoted to the derivation of these two desired bounds.
	\section{Bounds for $\sup_n \E^{\mathbb{Q}_n}[\Vert v_\cdot^n\Vert_{\mathbb{H}^\delta}^{2m}]$ and $\sup_n \E[\Vert v_\cdot^n\Vert_{\mathbb{H}^\delta}^2]$}\label{sec4}
	\subsection{Discussion on $\sup_n \E^{\mathbb{Q}_n}[\Vert v_\cdot^n\Vert_{\mathbb{H}^\delta}^{2m}]$ and the Strong Novikov Condition}
	The first lemma of this section can be obtained by adapting the proof of \cite[Lemma 6.4]{le2020stochastic}.
	\begin{lemma}\label{sec4-lm.1}
		Let $b(t,\cdot)\in L_x^pL_t^q$ be a non-negative function. Assume that $Hd/p+1/q<1$. Then, for any $s_0<t$, we have
		\begin{align}\label{sec4-eq.2}
			\mathcal{I}_n&:=\E_{s_0}\left[\int_{\Delta_{s_0,t}^n}\prod_{j=1}^n b(s_j,B_{s_j}^H)\d s\right]\nonumber\\
			&\leq C_{H,d,p}^n
			\frac{\Gamma(1-q'\frac{Hd}{p})^{\frac{n}{q'}}
				(t-s_0)^{n(1-\frac{Hd}{p}-\frac1q)}}
			{\Gamma\left(n(1-q'\frac{Hd}{p})+1\right)^{\frac1{q'}}(n!)^{\frac1q}}
			\Vert b\Vert_{L_x^pL_t^q}^n.
		\end{align}
		Here, we adopt the convention that $\E_0[\cdot]=\E[\cdot]$.
	\end{lemma}
	\begin{proof}
		Note that $B_t^H-\E_sB_t^H=\int_s^tK_H(t,r)\d W_r$ has a centered normal distribution with covariance matrix $\sigma_{s,t}^2I_d$ and is independent of $\mathcal{F}_s$, where $\sigma_{s,t}^2:=\mathrm{Var}(B_t^H-\E_sB_t^H)$. Hence,
		$$\E_s[b(t,B_t^H)]=\left(p_{\sigma_{s,t}^2}*b(t,\cdot)\right)(\E_sB_t^H),$$
		where $p_\sigma$ denotes the $d$-dimensional Gaussian density with mean zero and covariance matrix $\sigma I_d$. By Lemma \ref{appendix-lm.5} and the local non-determinism property, we have
		\begin{align}\label{sec4-eq.1}
			\E_s[b(t,B_t^H)]
			\leq \Vert p_{\sigma_{s,t}^2}*b(t,\cdot)\Vert_{L_x^\infty}
			\leq C_{H,d,p}(t-s)^{-\frac{Hd}{p}}
			\Vert b\Vert_{L_x^p}(t).
		\end{align}
		Then, by conditioning successively on $\mathcal{F}_{s_{n-1}},\cdots,\mathcal{F}_{s_1}$ and using the heat kernel estimate \eqref{sec4-eq.1}, we obtain
		$$\E_{s_0}\left[\prod_{j=1}^n b(s_j,B_{s_j}^H)
		\right]	\leq C_{H,d,p}^n\prod_{j=1}^n\Vert b\Vert_{L_x^p}(s_j)(s_j-s_{j-1})^{-\frac{Hd}{p}}.$$
		Consequently, invoking Fubini's theorem, H\"older's inequality and Lemma \ref{appendix-lm.4} yields the desired \eqref{sec4-eq.2}
		\begin{align*}
			\mathcal{I}_n
			&\leq C_{H,d,p}^n\int_{\Delta_{s_0,t}^n}\prod_{j=1}^n
			\Vert b\Vert_{L_x^p}(s_j)(s_j-s_{j-1})^{-\frac{Hd}{p}}\d s\\
			&\leq \frac{C_{H,d,p}^n}{(n!)^{\frac1q}}
			\Vert b\Vert_{L_x^pL_t^q}^n
			\left(\int_{\Delta_{s_0,t}^n}\prod_{j=1}^n
			(s_j-s_{j-1})^{-q'\frac{Hd}{p}}\d s\right)^{\frac1{q'}}\\
			&=C_{H,d,p}^n
			\frac{\Gamma(1-q'\frac{Hd}{p})^{\frac{n}{q'}}
				(t-s_0)^{n(1-\frac{Hd}{p}-\frac1q)}}
			{\Gamma\left(n(1-q'\frac{Hd}{p})+1\right)^{\frac1{q'}}(n!)^{\frac1q}}
			\Vert b\Vert_{L_x^pL_t^q}^n,
		\end{align*}
		thanks to the condition $Hd/p+1/q<1$. The proof is complete.
	\end{proof}
	\begin{lemma}\label{sec4-lm.2}
		Let $v_t=K_H^{-1}(\int_0^\cdot b(s,B_s^H)\d s)(t)$ for $b(t,\cdot)\in L_x^p L_t^q$. Assume \eqref{cond1}, $p,q\in[(1-H)^{-1},\infty)$ and $H<1/2$. Then, for $m\in\mathbb{N}$, there exists a constant $C=C(H,d,p,q,T)$ such that
		\begin{equation}
			\E[\Vert v_\cdot\Vert_{L_t^2}^{2m}]\leq C^{2m}m^{\frac{2Hd}p m}\Vert b\Vert_{L_x^pL_t^q}^{2m}
		\end{equation}
	\end{lemma}
	\begin{proof}
		By the definition of $K_H^{-1}$ for $H<1/2$, we have
		$$|v_t|\leq C_H\int_0^t(t-s)^{-\frac12-H}|b(s,B_s^H)|\d s.$$
		As a result, by Fubini's theorem
		\begin{align*}
			\int_0^T|v_t|^2\d t&\leq C_H \int_0^T\int_0^t\int_0^t(t-s)^{-\frac12-H}(t-r)^{-\frac12-H}|b(s,B_s^H)||b(r,B_r^H)|\d s\d r\d t\\
			&\leq C_H\int_0^T\int_0^T|b(s,B_s^H)||b(r,B_r^H)|\d s\d r\int_{s\vee r}^T(t-s)^{-\frac12-H}(t-r)^{-\frac12-H}\d t\\
			&\leq C_H\int_0^T\int_0^T|s-r|^{-2H}|b(s,B_s^H)||b(r,B_r^H)|\d s\d r.
		\end{align*}
		Here, we give a brief explanation for the last inequality. Assume that $r<s$ and denote $|s-r|$ by $h$. Then, changing the variable $t-s=hu$ leads to 
		$$\int_s^T(t-s)^{-\frac12-H}(t-r)^{-\frac12-H}\d t=h^{-2H}\int_0^{\frac{T-s}{h}}u^{-\frac12-H}(u+1)^{-\frac12-H}\d u\leq C_Hh^{-2H}.$$
		Consequently, by Hardy-Littlewood inequality, we have
		\begin{equation*}
			\int_0^T|v_t|^2\d t\leq C_{H,T}\left(\int_0^T|b(s,B_s^H)|^{\frac1{1-H}}\d s\right)^{2(1-H)}.
		\end{equation*}
		Note that $|b|^{\frac1{1-H}}\in L_t^{q_H}L_x^{p_H}$, where $p_H=p(1-H)$ and $q_H=q(1-H)$ satisfy all the conditions of Lemma \ref{sec4-lm.1} thanks to \eqref{cond1}. Hence, applying Jensen's inequality and Lemma \ref{sec4-lm.1} yields
		\begin{align*}
			\E[\Vert v_\cdot\Vert_{L_t^2}^{2m}]&\leq C^m\left(\E\left[\left(\int_0^T|b(s,B_s^H)|^{\frac1{1-H}}\d s\right)^{2m}\right]\right)^{1-H}\\
			&\leq C^m\left((2m)!\E\left[\int_{\Delta_{0,T}^{2m}}\prod_{j=1}^{2m}|b(s_j,B_{s_j}^H)|^{\frac1{1-H}}\d s\right]\right)^{1-H}\\
			&\leq C^{2m}\Vert b\Vert_{L_x^pL_t^q}^{2m} (2m!)^{1-H-\frac1q}\left(\frac{\Gamma(1-q_H'\frac{Hd}{p_H})^{\frac{2m}{q_H'}}}{\Gamma\left(2m(1-q_H'\frac{Hd}{p_H})+1\right)^{\frac1{q_H'}}}\right)^{1-H}.
		\end{align*}
		Finally, by Stirling's formula $\Gamma(s)=\sqrt{2\pi}(1+O(1/s))e^{-s}s^{s-1/2}$ and $(2m)!=\Gamma(2m+1)$, we complete the proof.
	\end{proof}
	\begin{corollary}\label{sec4-col.3}
		Let $v_t$ be the same quantity as in Lemma \ref{sec4-lm.2}. Assume \eqref{cond1}, $p,q\in[(1-H)^{-1},\infty)$ and $H<1/2$. Suppose that $\Vert b\Vert_{L_x^pL_t^q}\leq M$. Then, for $0<\delta<\min\{1/2-H,\kappa_H,1-H-1/p,1-H-1/q\}$, $\E[\Vert v_\cdot\Vert_{W^{\delta,2}}^{2m}]$ is bounded by a constant $C=C(H,d,p,q,T,m,M)$ for all $m\in\mathbb{N}$.
	\end{corollary}
	\begin{proof}
		For notational convenience, we set $\mathcal{R}_Hf(t)=K_H^{-1}(\int_0^\cdot f(s)\d s)(t)$ for $f$ such that $\int_0^\cdot f(s)\d s\in I_{0+}^{H+1/2}(L^2([0,T]))$. Then, $v_t=\mathcal{R}_Hb(\cdot,B_\cdot^H)(t)$. There exists a $\delta_0$ such that $\delta<\delta_0<\min\{1/2-H,\kappa_H,1-H-1/p,1-H-1/q\}$. Set $\bar{H}=H+\delta_0$. By the choice of $\delta_0$ and the definition of $\kappa_H$, we have $\bar H<1/2$, $p,q\geq(1-\bar H)^{-1}$, and $\bar H+Hd/p+1/q<1$. First, we show 
		\begin{equation}\label{sec4-eq.5}
			|\mathcal{R}_Hf(t+h)-\mathcal{R}_Hf(t)|\leq Ch^{\delta_0}(\mathcal{R}_{\bar{H}}|f|(t)+\mathcal{R}_{\bar{H}}|f|(t+h))
		\end{equation}
		for $0<t<T-h$.
		
		By the definition of $K_H^{-1}$, we have
		\begin{align*}
			\mathcal{R}_Hf(t+h)-\mathcal{R}_Hf(t)&=[(t+h)^{H-\frac12}-t^{H-\frac12}]\int_0^t(t-s)^{-\frac12-H}s^{\frac12-H}f(s)\d s\\
			&+(t+h)^{H-\frac12}\int_0^t[(t+h-s)^{-\frac12-H}-(t-s)^{-\frac12-H}]s^{\frac12-H}f(s)\d s\\
			&+(t+h)^{H-\frac12}\int_t^{t+h}(t+h-s)^{-\frac12-H}s^{\frac12-H}f(s)\d s=:I_1+I_2+I_3.
		\end{align*}
		For the $I_1$ term, Lemma \ref{appendix-lm.1} implies $|(t+h)^{H-1/2}-t^{H-1/2}|\leq Ch^{\delta_0}t^{H-1/2-\delta_0}$. Since
		$$t^{H-1/2-\delta_0}(t-s)^{-1/2-H}s^{1/2-H}\leq C t^{\bar H-1/2}(t-s)^{-1/2-\bar H}s^{1/2-\bar H},$$
		the $I_1$ term is bounded by $Ch^{\delta_0}\mathcal{R}_{\bar H}|f|(t)$.
		
		For the $I_2$ term, Lemma \ref{appendix-lm.1} gives $|(t+h-s)^{-1/2-H}-(t-s)^{-1/2-H}|\leq Ch^{\delta_0}(t-s)^{-1/2-\bar H}$, and hence $|I_2|\leq Ch^{\delta_0}\mathcal{R}_{\bar H}|f|(t)$. Similarly, $|I_3|\leq Ch^{\delta_0}\mathcal{R}_{\bar H}|f|(t+h)$.
		
		From \eqref{sec4-eq.5}, we obtain 
		$$\Vert v_{\cdot+h}-v_\cdot\Vert_{L_t^2}\leq Ch^{\delta_0}\Vert\mathcal{R}_{\bar{H}}|b(\cdot,B_\cdot^H)|\Vert_{L_t^2}.$$
		Then, by Minkowski's inequality, we have
		\begin{align*}
			\big\Vert \Vert v_\cdot\Vert_{W^{\delta,2}} ^2\big\Vert_{L^m(\Omega)}&\leq \big\Vert \Vert v_\cdot\Vert_{L_t^2} ^2\big\Vert_{L^m(\Omega)}+2\left\Vert\int_0^T h^{-1-2\delta}\Vert v_{\cdot+h}-v_\cdot\Vert_{L_t^2}^2\d h \right\Vert_{L^m(\Omega)}\\
			&\leq C\big\Vert \Vert v_\cdot\Vert_{L_t^2}^2 \big\Vert_{L^m(\Omega)}+C\int_0^Th^{-1-2\delta+2\delta_0}\big\Vert\Vert\mathcal{R}_{\bar{H}}|b(\cdot,B_\cdot^H)|\Vert_{L_t^2}^2\big\Vert_{L^m(\Omega)}\d h.
		\end{align*}
		Given that the argument used in Lemma \ref{sec4-lm.2}, with the singularity exponent in $\mathcal R_H$ replaced by $\bar H$ while the driving process remains $B^H$, applies under the conditions above, both terms on the right-hand side admit finite moments of all orders. This completes the proof.
	\end{proof}
	\begin{remark}
		Given that $\mathbb{H}^\delta$ and $W^{\delta,2}$ have equivalent norms for $\delta<1/2$ (see \cite[eq. (2.13)]{nochetto2015pde} for instance), the estimate in Corollary \ref{sec4-col.3} holds for $\E[\Vert v_\cdot\Vert_{\mathbb{H}^\delta}^{2m}]$. 
	\end{remark}
	The following proposition gives the sharp condition guaranteeing the strong Novikov condition for any drift in $L_x^pL_t^q$.
	\begin{proposition}\label{sec4-prop.3}
		Let $v_t$ be the same quantity as in Lemma \ref{sec4-lm.2}. Assume \eqref{cond1}, $p,q\geq(1-H)^{-1}$ and $H<1/2$. If $p\geq2Hd$, then the strong Novikov condition
		$$\E\left[\exp\left(\frac12\Vert v_\cdot\Vert_{L_t^2}^2\right)\right]<\infty$$
		holds for any $b\in L_x^pL_t^q$. If $p<2Hd$, there exists $b\in L_x^pL_t^q$ for which this expectation is infinite.
	\end{proposition}
	\begin{proof}
		We prove the stronger statement with $1/2$ replaced by any $\lambda>0$. Suppose that $p>2Hd$. By Taylor's expansion, Lemma \ref{sec4-lm.2} and Stirling's formula, we have
		\begin{align*}
			\E[\exp(\lambda\Vert v_\cdot\Vert_{L_t^2}^2)]
			&=1+\sum_{m=1}^\infty\frac{\lambda^m}{m!}\E[\Vert v_\cdot\Vert_{L_t^2}^{2m}]\\
			&\leq 1+\sum_{m=1}^\infty(C\lambda\Vert b\Vert_{L_x^pL_t^q}^{2})^m
			m^{(\frac{2Hd}{p}-1)m}<\infty.
		\end{align*}
		thanks to the condition $p>2Hd$. Now suppose that $p=2Hd$, and fix the constant $C$ appearing in the last inequality. Decompose $b$ as $b_1^M+b_2^M$, where $b_1^M=b\mathbbm{1}_{\{|b|\leq M\}}$ and $b_2^M=b\mathbbm{1}_{\{|b|> M\}}$ such that $\Vert b_2^M\Vert_{L_x^pL_t^q}^2<(4\lambda C)^{-1}$. Then, using the elementary inequality $|a+b|^2\leq 2|a|^2+2|b|^2$, we have
		\begin{align*}
			\Vert v_\cdot\Vert_{L_t^2}^2\leq 2C_HT^{2-2H}M^2+2\int_0^T\left|K_H^{-1}\left(\int_0^\cdot b_2^M(s,B_s^H)\d s\right)(t)\right|^2\d t
		\end{align*} 
		As a consequence, 
		\begin{align*}
			\E[\exp(\lambda\Vert v_\cdot\Vert_{L_t^2}^2)]&\leq e^{2\lambda C_{H,T}M^2}\E\left[\exp\left(2\lambda\int_0^T\left|K_H^{-1}\left(\int_0^\cdot b_2^M(s,B_s^H)\d s\right)(t)\right|^2\d t\right)\right]\\
			&\leq e^{2\lambda C_{H,T}M^2}\left(1+\sum_{m=1}^\infty(2\lambda C)^m\Vert b_2^M\Vert_{L_x^pL_t^q}^{2m}\right)<\infty.
		\end{align*} 
		
		On the other hand, we construct a counterexample to show that $\E[\exp(\lambda\Vert v_\cdot\Vert_{L_t^2}^2)]<\infty$ may fail for some $b$ when $p<2Hd$. Without loss of generality, we assume $T=1$. Let $I_n=(2^{-n},2^{-n+1}]$ for $n\geq1$, $\ell_n=2^{-n}$ and $r_n=\ell_n^\alpha$ for some $\alpha>H$ satisfying
		$$\frac1q-\frac12+H+\alpha \left(\frac{d}{p}-\frac1{2H}\right)>0.$$
		Let $A_n=n\ell_n^{H-1/2-\alpha/(2H)}$ and choose $\chi\in C_c^\infty((-2,2)^d)$ such that $0\leq \chi\leq 1$ and $\chi=1$ on $[-1,1]^d$. Set
		$$f(t,x)=\sum_{n=1}^\infty A_n\mathbbm{1}_{I_n}(t)\chi\left(\frac x{r_n}\right)e_1$$
		where $e_1=(1,0,\ldots,0)^T\in\R^d$, and set $v_t^f=K_H^{-1}(\int_0^\cdot f(s,B_s^H)\d s)(t)$. We will show that $f\in L_x^pL_t^q$ but $\E[\exp(\lambda\Vert v_\cdot^f\Vert_{L_t^2}^2)]=\infty$. By triangle inequality, we have
		\begin{align*}
			\Vert f\Vert_{L_x^pL_t^q}&\leq \sum_{n=1}^\infty A_n\Vert \mathbbm{1}_{I_n}(\cdot)\Vert_{L_t^q}\Vert \chi(\cdot/r_n)\Vert_{L_x^p}\\
			&\leq C_d\sum_{n=1}^\infty A_n \ell_n^{\frac1q}r_n^{\frac dp}=\sum_{n=1}^\infty n\ell_n^{\frac1q-\frac12+H+\alpha(\frac dp-\frac1{2H})}<\infty.
		\end{align*}
		Hence, $f\in L_x^pL_t^q$. However,
		\begin{align*}
			\Vert v_\cdot^f\Vert_{L_t^2}^2&\geq \int_0^T s^{2H-1}\left(\int_0^s(s-r)^{-\frac12-H}r^{\frac12-H}A_n\mathbbm{1}_{I_n}(r)\chi(B_r^H/r_n)\d r\right)^2\d s\\
			&\geq A_n^2\mathbbm{1}_{\{\sup_{t\in I_n}\Vert B_t^H\Vert_\infty\leq r_n\}}\int_0^Ts^{2H-1}\left(\int_{2^{-n}}^{2^{-n+1}\wedge s}(s-r)^{-\frac12-H}r^{\frac12-H}\d r\right)^2\d s
		\end{align*}
		Note that
		\begin{align*}
			&\int_0^Ts^{2H-1}\left(\int_{2^{-n}}^{2^{-n+1}\wedge s}(s-r)^{-\frac12-H}r^{\frac12-H}\d r\right)^2\d s\\
			&\qquad\geq \int_{2^{-n}}^{2^{-n+1}}s^{2H-1}\left(\int_{2^{-n}}^{s}(s-r)^{-\frac12-H}r^{\frac12-H}\d r\right)^2\d s\\
			&\qquad\geq \ell_n^{1-2H}\int_{2^{-n}}^{2^{-n+1}}s^{2H-1}(s-2^{-n})^{1-2H}\d s=C_H\ell_n^{2-2H}.
		\end{align*}
		Hence, we have
		$$\Vert v_\cdot^f\Vert_{L_t^2}^2\geq c_Hn^2\ell_n^{1-\frac\alpha{H}}\mathbbm{1}_{\{\sup_{t\in I_n}\Vert B_t^H\Vert_\infty\leq r_n\}}$$
		for all $n\geq1$. Recall the small ball probability of fractional Brownian motion
		\begin{equation}
			-\log\mathbb{P}\left(\sup_{0\leq t\leq 1}\Vert B_t^H\Vert_{\infty}\leq \varepsilon\right)\asymp \varepsilon^{-\frac1 H}
		\end{equation}
		(see \cite{stolz1996} for instance). Then
		\begin{align*}
			\E[\exp(\lambda\Vert v_\cdot^f\Vert_{L_t^2}^2)]&\geq \exp\left(c_H\lambda n^2\ell_n^{1-\frac{\alpha}H}\right)\mathbb{P}\left(\sup_{0\leq t\leq 2^{-n+1}}\Vert B_t^H\Vert_\infty\leq r_n\right)\\
			&= \exp\left(c_H\lambda n^2\ell_n^{1-\frac{\alpha}H}\right)\mathbb{P}\left(\sup_{0\leq t\leq 1}\Vert B_t^H\Vert_\infty\leq 2^{H(n-1)}r_n\right)\\
			&\geq \exp\left(c_H\lambda n^2\ell_n^{1-\frac\alpha{H}}-C_d\ell_n^{1-\frac{\alpha}{H}}\right)\overset{n\rightarrow\infty}{\longrightarrow} \infty
		\end{align*}
		since $\alpha>H$. Here, we use the fact that $\{\sup_{0\leq t\leq 2^{-n+1}}\Vert B_t^H\Vert_\infty\leq r_n\}\subset\{\sup_{t\in I_n}\Vert B_t^H\Vert_\infty\leq r_n\}$ in the first inequality and $H$-self-similarity of fBm in the first equality. As a consequence, $\E[\exp(\lambda\Vert v_\cdot^f\Vert_{L_t^2}^2)]=\infty$ for all $\lambda>0$. The proof is complete.
	\end{proof}
	
	\subsection{A Martingale Approach to Estimate $\sup_n \E[\Vert v_\cdot^n\Vert_{\mathbb{H}^\delta}^2]$}
	
	In this subsection, we will adopt the \textit{martingale approach for fBm} developed in \cite{vienszhang2019} to show the finiteness of $\sup_n \E[\Vert v_\cdot^n\Vert_{L_t^2}^2]$. In the following, we assume that $b:[0,T]\times\R^d\rightarrow\R^d$ is a smooth function with compact support and $X$ is the corresponding strong solution. The key observation is that for a fixed $t$,
	$$M_r^t:=\E_r[B_t^H],\qquad0\leq r\leq t$$
	is a martingale and thus the auxiliary process
	$$Z_r^t=x_0+\int_0^rb(u,X_u)\d u+M_r^t$$ 
	is a semimartingale. Then, the classical It\^o formula can be applied to $Z_r^t$.  
	
	\begin{proposition}\label{sec4-prop.5}
		Assume \eqref{cond1}, $H<1/2$ and let $\rho\in[(1-H)^{-1},2)$ satisfy $p,q\in(\rho,\infty)$ and $\rho(Hd/p+1/q)<1$. Suppose that $\Vert b\Vert_{L_x^pL_t^q}\leq M$. Then, for $0<s<t$, there exists a constant $C=C(H,d,p,q,\rho,T,M)$ such that
		\begin{equation}\label{sec4-eq.7}
			\E_s[|b(t,X_t)|^\rho]\leq C(t-s)^{-\frac{Hd\rho}{p}}\Vert b\Vert_{L_x^p}^\rho(t)\quad\mathbb{P}\text{-a.s.}
		\end{equation}
	\end{proposition} 
	\begin{proof}
		For notational simplicity, we set $\alpha=\rho/p$, $\beta=Hd\rho/p$ and $\sigma_r:=\mathrm{Var}(B_t^H-\E_rB_t^H)=\int_r^tK_H^2(t,u)\d u$. Define
		$F(r,x):=(p_{\sigma_r}*|b(t,\cdot)|^p)(x)$ for $r<t$. Since $b$ is continuous, $\lim_{r\uparrow t}F(r,x)=|b(t,x)|^p$. For $\delta>0$, we set $F_\delta:=F+\delta$ and $G_\delta:=F_\delta^\alpha$. By a direct computation, we have
		\begin{align*}  
			&\partial_r G_\delta=-\frac\alpha2 K_H^2(t,r)F_\delta^{\alpha-1}\Delta_x F\\
			&\nabla_xG_\delta=\alpha F_\delta^{\alpha-1}\nabla F\\
			&\Delta_xG_\delta=\alpha F_\delta^{\alpha-1}\Delta_x F-\alpha(1-\alpha)F_\delta^{\alpha-2}|\nabla_x F|^2\\
			&\d Z_r^t=b(r,X_r)\d r+\d M_r^t
		\end{align*}
		and 
		$$\d M_r^t=K_H(t,r)\d W_r,\qquad\d \langle M^t\rangle_r=K_H^2(t,r)\d r.$$
		Then, for $s<\tau<t$, It\^o's formula gives
		\begin{align}\label{sec4-eq.8}
			G_\delta(\tau,Z_\tau^t)&=G_\delta(s,Z_s^t)+\alpha\int_s^\tau K_H(t,r) F_\delta^{\alpha-1}\langle\nabla_x F,\d W_r\rangle\nonumber\\
			&+\int_s^\tau\left[\alpha F_\delta^{\alpha-1}\langle\nabla_x F,b(r,X_r)\rangle-\frac{\alpha(1-\alpha)}{2}F_\delta^{\alpha-2}K_H^2(t,r)|\nabla_x F|^2\right]\d r.
		\end{align}
		Note that, for any $r\in[s,\tau]$,
		\begin{align*}
			\int_s^r K_H^2(t,u)F_\delta^{2(\alpha-1)}|\nabla_x F|^2\d u&\leq \delta^{2(\alpha-1)}\int_s^r K_H^2(t,u)|\nabla_x F|^2\d u\\
			&\leq C_{H,d,p}\delta^{2(\alpha-1)}\Vert b\Vert_{L_x^p}^{2p}(t)\int_s^\tau K_H^2(t,u)\sigma_u^{-d-1}\d u\\
			&\leq C_{H,d,p}\delta^{2(\alpha-1)}(t-\tau)^{-2Hd-2H}\Vert b\Vert_{L_x^p}^{2p}(t)\int_s^\tau K_H^2(t,u)\d u<\infty,
		\end{align*}
		where we use the estimate
		$$\Vert\nabla_x p_{\sigma_r}*|b(t,\cdot)|^p\Vert_{L_x^\infty}\leq C_d\sigma_r^{-\frac{d+1}{2}}\Vert b\Vert_{L_x^p}^p(t)$$
		given by Lemma \ref{appendix-lm.5} and the local non-determinism $\sigma_r\geq C_H(t-r)^{2H}$. Given that $Z_r^t$ is an adapted process, we conclude that
		$$\left\{\int_s^r\tau K_H(t,u) F_\delta^{\alpha-1}\langle\nabla_x F,\d W_u\rangle\right\}_{r\in[s,\tau]}$$
		is a martingale for fixed $\tau\in(s,t)$. Therefore, taking the conditional expectation $\E_s[\cdot]$ on both sides of \eqref{sec4-eq.8}, then letting $\tau\uparrow t$ and applying Fatou's lemma, we obtain
		\begin{align}\label{sec4-eq.9}
			\E_s[(|b(t,X_t)|^p+\delta)^\alpha]\leq G_\delta(s,Z_s^t)+\int_s^t\E_s[\Phi_\delta(r)]\d r\quad\mathbb{P}\text{-a.s.}
		\end{align}
		where
		$$\Phi_{\delta}(r):=\alpha F_\delta^{\alpha-1}\langle\nabla_x F,b(r,X_r)\rangle-\frac{\alpha(1-\alpha)}{2}F_\delta^{\alpha-2}K_H^2(t,r)|\nabla_x F|^2.$$
		
		For $x,y\in\R^d$ and $1\leq \vartheta<2$, suppose that $|x|\leq R$. One can prove that
		$$\langle x,y\rangle-|x|^2\leq C_\vartheta R^{2-\vartheta}|y|^\vartheta.$$
		Indeed,
		$$\langle x,y\rangle-|x|^2\leq \min\left\{|x||y|,\frac{|y|^2}{4}\right\}.$$
		If $|y|\leq R$, then $|y|^2=|y|^{2-\vartheta}|y|^\vartheta\leq R^{2-\vartheta}|y|^\vartheta$. If $|y|>R$, then $|x||y|=|x||y|^\vartheta|y|^{1-\vartheta}\leq R^{2-\vartheta}|y|^\vartheta$. Set $\vartheta=\rho$,
		\begin{align*}
			x=\sqrt{\frac{\alpha(1-\alpha)}{2}}K_H(t,r)F_\delta^{\frac{\alpha-2}2}\nabla_x F,\qquad y=\alpha\sqrt{\frac2{\alpha(1-\alpha)}}\frac1{K_H(t,r)}F_\delta^{\frac\alpha2}b(r,X_r).
		\end{align*}
		Let $F_c=p_{c\sigma_r}*|b(t,\cdot)|^p$ for some $c>1$. By \eqref{appendix-eq.4}, we can set
		$$R=C_c\sqrt{\frac{\alpha(1-\alpha)}{2}}K_H(t,r)\sigma_r^{-\frac12}F_\delta^{\frac{\alpha-1}{2}}F_c^{\frac12}$$
		for some $C_c>0$ such that $|x|\leq R$. Then, we have
		$$\Phi_\delta(r)=\langle x,y\rangle-|x|^2\leq C_{H,d,p,c}(K_H(t,r))^{2(1-\rho)}\sigma_r^{-\frac{(2-\rho)}2}F_\delta^{\alpha-\frac{2-\rho}{2}}F_c^{1-\frac{\rho}2}|b(r,X_r)|^\rho.$$
		
		Using the estimate $K_H(t,r)\geq C_H(t-r)^{H-1/2}$ (see \cite[Proposition B.2]{butkovsky2023stochastic} for instance) and the local non-determinism $\sigma_r\geq C_H(t-r)^{2H}$, we have
		$$(K_H(t,r))^{2(1-\rho)}\sigma_r^{-\frac{(2-\rho)}2}\leq C_H(t-r)^{\rho(1-H)-1}\leq C_{H,\rho,T}.$$
		Choosing a pair $(c,\theta)$ such that
		$$\alpha-\frac{(2-\rho)(1-\theta)}{2}>0,\qquad 0<\theta<\frac1c,$$
		and using the estimate $\Vert F\Vert_{L_x^\infty}\leq C_{d,p}\sigma_r^{-d/2}\Vert b\Vert_{L_x^p}^p(t)$, together with the local non-determinism and \eqref{appendix-eq.5}, we obtain
		\begin{align*}
			\limsup_{\delta\downarrow0}\Phi_\delta(r)&\leq C_{H,d,p,\rho}\left(\sigma_r^{-\frac d2}\Vert b\Vert_{L_x^p}^p(t)\right)^{\frac{(2-\rho)(1-\theta)}{2}}
			F^{\alpha-\frac{(2-\rho)(1-\theta)}{2}}|b(r,X_r)|^\rho\\
			&\leq C_{H,d,p,\rho}\left(\sigma_r^{-\frac d2}\Vert b\Vert_{L_x^p}^p(t)\right)^\alpha|b(r,X_r)|^\rho\\
			&\leq C_{H,d,p,\rho}(t-r)^{-\beta}\Vert b\Vert_{L_x^p}^\rho(t)|b(r,X_r)|^\rho\quad\mathbb{P}\text{-a.s.}
		\end{align*}
		By similar reasoning and the estimate $F_\delta\leq C_{d,p}\sigma_r^{-d/2}\Vert b\Vert_{L_x^p}^p(t)+1$ for all $0<\delta\leq1$, we have
		$$\Phi_\delta(r)\leq C_{H,d,p,\rho}\left(1+(t-r)^{-\beta}\Vert b\Vert_{L_x^p}^\rho(t)\right)|b(r,X_r)|^\rho=:\Psi(r).$$
		Since $b$ is bounded and $\beta<1$, it holds $\mathbb{P}$-a.s. that $\int_s^t\E_s[\Psi(r)]\d r<\infty$.
		Then, applying Fatou's lemma to $\Psi-\Phi_\delta$, we obtain
		\begin{align*}
			\limsup_{\delta\downarrow0}\int_s^t\E_s[\Phi_\delta(r)]\d r&\leq \int_s^t\E_s\big[\limsup_{\delta\downarrow0}\Phi_\delta(r)\big]\\
			&\leq Cf(t)\int_s^t(t-r)^{-\beta}\E_s[|b(r,X_r)|^\rho]\d r,\quad\mathbb{P}\text{-a.s.}
		\end{align*}
		where $f(t)=\Vert b\Vert_{L_x^p}^\rho(t)$. Furthermore, the same heat-kernel estimate gives $F(s,Z_s^t)^\alpha\leq (t-s)^{-\beta}f(t)$. Then, letting $\delta\downarrow0$ in \eqref{sec4-eq.9} leads to
		$$\E_s[|b(t,X_t)|^\rho]\leq C(t-s)^{-\beta}f(t)+Cf(t)\int_s^t(t-r)^{-\beta}\E_s[|b(r,X_r)|^\rho]\d r\quad\mathbb{P}\text{-a.s.}$$
		For $N\in\mathbb{N}$, iterating the above inequality $N$ times, we obtain
		\begin{align*}
			\E_s[|b(t,X_t)|^\rho]
			&\leq C(t-s)^{-\beta}f(t)\\
			&+f(t)\sum_{n=1}^N C^{n+1}\int_{\Delta_{s,t}^n}(t-s_n)^{-\beta}
			\prod_{j=1}^n(s_j-s_{j-1})^{-\beta}f(s_j)\d s+R_N(t),\quad\mathbb{P}\text{-a.s.}
		\end{align*}
		where $s_0=s$ and
        $$R_N(t):=C^{N+2}f(t)\int_{\Delta_{s,t}^{N+1}}\left[(t-s_{N+1})^{-\beta}\prod_{j=2}^{N+1}(s_j-s_{j-1})^{-\beta}f(s_j)\right]\E_s[|b(s_1,X_{s_1})|^\rho]\d s.$$
        Since $b$ is bounded, we have $\E_s[|b(s_1,X_{s_1})|^\rho]\leq\Vert b\Vert_\infty^\rho$. Note that $f\in L^{q_\rho}([0,T])$ with $q_\rho=q/\rho$. Then, by H\"older's inequality and Lemma \ref{appendix-lm.4},
        \begin{align*}
            R_N(t)&\leq C^{N+2}\Vert b\Vert_\infty^\rho f(t)\frac{\Gamma(1+q_\rho'(1-\beta))^{\frac1{q_\rho'}}\Gamma(1-q_\rho'\beta)^{\frac{N}{q_\rho'}}}{\Gamma((N+1)(1-q_\rho'\beta)+q_\rho')^{\frac1{q_\rho'}}(N!)^{\frac1{q_\rho}}}\\
            &\qquad\times(t-s)^{1-\beta+N(1-\beta-\frac1{q_\rho})}\Vert f\Vert_{L_t^{q_\rho}}^N.
        \end{align*}
        Since $\beta<1-1/q_\rho$, Stirling's formula shows that $R_N(t)\rightarrow0$ as $N\rightarrow\infty$. Therefore, letting $N\rightarrow\infty$, we obtain
        \begin{align*}
	       \E_s[|b(t,X_t)|^\rho]&\leq C(t-s)^{-\beta}f(t)\\
           &+f(t)\sum_{n=1}^\infty C^{n+1}\int_{\Delta_{s,t}^n}(t-s_n)^{-\beta}\prod_{j=1}^n(s_j-s_{j-1})^{-\beta}f(s_j)\d\mathbf{s},\quad\mathbb{P}\text{-a.s.}
        \end{align*}
        Similarly, the same argument gives
        \begin{align*}
			\E_s[|b(t,X_t)|^\rho]
			&\leq C(t-s)^{-\beta}f(t)+f(t)\sum_{n=1}^\infty
			\frac{C^{n+1}\Gamma(1-q_\rho'\beta)^{\frac{n+1}{q_\rho'}}}
			{\Gamma((n+1)(1-q_\rho'\beta))^{\frac1{q_\rho'}}(n!)^{\frac1{q_\rho}}}\\
			&\qquad\qquad\times(t-s)^{-\beta+n(1-\beta-\frac1{q_\rho})}
			\Vert f\Vert_{L_t^{q_\rho}}^n\\
			&\leq C(t-s)^{-\beta}f(t)\left[1+\sum_{n=1}^\infty
			\frac{C^n\Gamma(1-q_\rho'\beta)^{\frac{n+1}{q_\rho'}}}
			{\Gamma((n+1)(1-q_\rho'\beta))^{\frac1{q_\rho'}}(n!)^{\frac1{q_\rho}}}
			\Vert b\Vert_{L_x^pL_t^q}^{n\rho}\right]
		\end{align*}
        and the convergence of the series appearing in the last inequality. Consequently, we have
		$$\E_s[|b(t,X_t)|^\rho]\leq C(t-s)^{-\frac{Hd\rho}{p}}\Vert b\Vert_{L_x^p}^\rho(t)\quad\mathbb{P}\text{-a.s.}$$
		The proof is complete.
	\end{proof}
	
	Proposition \ref{sec4-prop.5} is the key input for the
uniform moment estimates below. With Proposition \ref{sec4-prop.5} in hand, we can follow the same proof pipeline of Lemma \ref{sec4-lm.1} and \ref{sec4-lm.2} to establish the following results.
	\begin{corollary}\label{sec4-col.7}
		Let $v_t=K_H^{-1}(\int_0^\cdot b(s,X_s)\d s)(t)$, where $b:[0,T]\times\R^d\rightarrow\R^d$ is a continuously differentiable function with compact support and $X$ is the strong solution to the corresponding SDE. Assume \eqref{cond1}, $p,q\in((1-H)^{-1},\infty)$ and $H<1/2$. Suppose that $\Vert b\Vert_{L_x^pL_t^q}\leq M$. Then, for $m\in\mathbb{N}$, there exists a constant $C=C(H,d,p,q,T,M)$ such that
		\begin{equation}
			\E[\Vert v_\cdot\Vert_{L_t^2}^{2m}]\leq C^{2m}m^{\frac{2Hd}p m}\Vert b\Vert_{L_x^pL_t^q}^{2m}.
		\end{equation}
	\end{corollary}
	\begin{proof}
		Set $\rho=(1-H)^{-1}$. Since $H<1/2$, we have $\rho<2$. Moreover, $p,q>\rho$ by assumption and
		$$\rho\left(\frac{Hd}{p}+\frac1q\right)<1$$
		by \eqref{cond1}. Thus, Proposition \ref{sec4-prop.5} is applicable. Set $p_H=p(1-H)$ and $q_H=q(1-H)$. As in the proof of Lemma \ref{sec4-lm.2}, the Hardy-Littlewood inequality gives
		\begin{align*}
			\E[\Vert v_\cdot\Vert_{L_t^2}^{2m}]
			&\leq C^m\left((2m)!\E\left[\int_{\Delta_{0,T}^{2m}}\prod_{j=1}^{2m}|b(s_j,X_{s_j})|^{\frac1{1-H}}\d s\right]\right)^{1-H}.
		\end{align*}
		By conditioning successively on $\mathcal{F}_{s_{2m-1}},\cdots,\mathcal{F}_{s_1}$ and using Proposition \ref{sec4-prop.5}, we have
		$$\E\left[\prod_{j=1}^{2m}|b(s_j,X_{s_j})|^{\frac1{1-H}}\right]\leq C^{2m}\prod_{j=1}^{2m}\Vert b\Vert_{L_x^p}^{\frac1{1-H}}(s_j)(s_j-s_{j-1})^{-\frac{Hd}{p(1-H)}}.$$
		Therefore, Fubini's theorem, H\"older's inequality and Lemma \ref{appendix-lm.4} yield
		\begin{align*}
			\E[\Vert v_\cdot\Vert_{L_t^2}^{2m}]
			&\leq C^{2m}\Vert b\Vert_{L_x^pL_t^q}^{2m}(2m!)^{1-H-\frac1q}
			\left(\frac{\Gamma(1-q_H'\frac{Hd}{p_H})^{\frac{2m}{q_H'}}}
			{\Gamma\left(2m(1-q_H'\frac{Hd}{p_H})+1\right)^{\frac1{q_H'}}}\right)^{1-H}.
		\end{align*}
		The desired estimate now follows from Stirling's formula, exactly as in the proof of Lemma \ref{sec4-lm.2}. The proof is complete.
	\end{proof}
	\begin{corollary}\label{sec4-col.8}
		Let $v_t$ be the same quantity as in Corollary \ref{sec4-col.7}. Assume \eqref{cond1}, $p,q\in((1-H)^{-1},\infty)$ and $H<1/2$.  Suppose that $\Vert b\Vert_{L_x^pL_t^q}\leq M$. Then, for $0<\delta<\min\{1/2-H,\kappa_H,1-H-1/p,1-H-1/q\}$, $\E[\Vert v_\cdot\Vert_{\mathbb{H}^\delta}^{2m}]$ is bounded by a constant $C=C(H,d,p,q,T,m,M)$ for all $m\in\mathbb{N}$.
	\end{corollary}
	\begin{proof}
		The proof follows the same lines as that of Corollary
		\ref{sec4-col.3}. Choose
		$\delta<\delta_0<\min\{1/2-H,\kappa_H,1-H-1/p,1-H-1/q\}$. Set $\bar{H}=H+\delta_0$ and $\rho=(1-\bar{H})^{-1}$. Then, the choice of $\delta_0$ and $\rho$ allows us to apply Proposition
		\ref{sec4-prop.5} to derive the required moment bound for $\mathcal{R}_{\bar H}|b(\cdot,X_\cdot)|$. The remaining argument is identical to that of Corollary \ref{sec4-col.3}. 
	\end{proof}
	
	\section{Uniform Estimates for $\E[\Vert \mathcal{I}^n_{\cdot,t}\Vert_{\mathcal{H}_\beta}^4]$ and $\E^{\mathbb{Q}_n}[\Vert \D^H v^n\Vert_{\mathbb{H}^{-2}([0,T];\mathcal{H}_\beta)}^8]$}\label{sec5}
	\subsection{Block Integral Estimates for $\sup_n\E[\|\mathcal{I}^n_{\cdot,t}\|_{\mathcal{H}_\beta}^4]$}
	Recall the definition \eqref{sec3-eq.2} of $\mathcal{I}_{\theta,t}^n$. For any integer $r\geq1$, the triangle inequality gives
	$$\big\Vert\Vert\mathcal{I}_{\cdot,t}^n\Vert_{\mathcal{H}_\beta}\big\Vert_{L^{2r}(\Omega)}\leq \sum_{m=1}^\infty\big\Vert\Vert G_m^n(\cdot,t)\Vert_{\mathcal{H}_\beta}\big\Vert_{L^{2r}(\Omega)},$$ 
	where
	\begin{align*}
		G_m^n(\theta,t)
		&:=\int_{\Delta_{0,t}^m}\mathbbm{1}_{[0,s_1]}(\theta)\nabla b_n(s_m,x_0+B_{s_m}^H)\cdots \nabla b_n(s_1,x_0+B_{s_1}^H)\d s.
	\end{align*}
	For notational simplicity, we let $\mathcal{A}_m^n(s):=\nabla b_n(s_m,x_0+B_{s_m}^H)\cdots \nabla b_n(s_1,x_0+B_{s_1}^H)$. Note that $\mathcal{K}_H^*\mathbbm{1}_{[0,s]}(\cdot)=K_H(s,\cdot)$. Then, by the definition of $\Vert\cdot\Vert_{\mathcal{H}_\beta}$ and Fubini's theorem, we have
	\begin{align*}
		\Vert G_m^n(\cdot,t)\Vert_{\mathcal{H}_\beta}^2=\int_{\Delta_{0,t}^m}\int_{\Delta_{0,t}^m}\langle K_H(s_1,\cdot),K_H(s_1',\cdot)\rangle_{W^{\beta,2}}\left\langle \mathcal{A}_m^n(s),\mathcal{A}_m^n(s')\right\rangle_F\d s\d s'.
	\end{align*}
	 Write  $s^k:=(s_1^k,\cdots,s_m^k)$ and denote by $(\mathcal{A}_{m}^n)_{ij}$ the $(i,j)$-th entry of $\mathcal{A}_{m}^n$. Expanding the $r$-th power of the above equality, we obtain
	\begin{align*}
		\E[\Vert 
		G_m^n(\cdot,t)\Vert_{\mathcal{H}_\beta}^{2r}]&=\sum_{i_1,j_1,\cdots,i_r,j_r=1}^d\int_{(\Delta_{0,t}^m)^{2r}}\prod_{k=1}^r\langle K_H(s_1^{2k-1},\cdot),K_H({s_1}^{2k},\cdot)\rangle_{W^{\beta,2}}\\
		&\qquad\times\E\left[\prod_{k=1}^r \left(\mathcal{A}_m^n(s^{2k-1})\right)_{i_kj_k}\left(\mathcal{A}_m^n(s^{2k})\right)_{i_kj_k}\right]\prod_{k=1}^{2r}\d s^k\\
		&\leq C_H^{2r}t^{2r(H-\beta)}\\
		&\qquad\times\sum_{i_1,j_1,\cdots,i_r,j_r}^d\int_{(\Delta_{0,t}^m)^{2r}}\left|\E\left[\prod_{k=1}^r \left(\mathcal{A}_m^n(s^{2k-1})\right)_{i_kj_k}\left(\mathcal{A}_m^n(s^{2k})\right)_{i_kj_k}\right]\right|\prod_{k=1}^{2r}\d s^k\\
	\end{align*}
	where for $0<\beta<H$, the above inequality follows from Lemma \ref{appendix-lm.6} and the Cauchy-Schwarz inequality:
	\begin{align*}
		|\langle K_H(s_1,\cdot),K_H(s_1',\cdot)\rangle_{W^{\beta,2}}|
		\leq \Vert K_H(s_1,\cdot)\Vert_{W^{\beta,2}}
		\Vert K_H(s_1',\cdot)\Vert_{W^{\beta,2}}\leq C_Ht^{2H-2\beta}.
	\end{align*}
	Expanding the matrix products in $\mathcal{A}_m^n$, we have
	\begin{align*}
		&\left|\E\left[\prod_{k=1}^{2r}\left(\mathcal{A}_m^n(s^k)\right)_
		{i_{\lceil k/2\rceil}j_{\lceil k/2\rceil}}
		\right]\right|\\
		&\qquad\leq\sum_{\ell_1^1,\cdots,\ell_{m-1}^1=1}^d\cdots
		\sum_{\ell_1^{2r},\cdots,\ell_{m-1}^{2r}=1}^d\\
		&\qquad\quad\times\Bigg|\E\Bigg[\prod_{k=1}^{2r}\Bigg(
		\frac{\partial}{\partial x_{\ell_1^k}}b_n^{(i_{\lceil k/2\rceil})}
		(s_m^k,x_0+B_{s_m^k}^H)
		\frac{\partial}{\partial x_{\ell_2^k}}b_n^{(\ell_1^k)}
		(s_{m-1}^k,x_0+B_{s_{m-1}^k}^H)\cdots\\
		&\qquad\quad\quad\times
		\frac{\partial}{\partial x_{\ell_{m-1}^k}}
		b_n^{(\ell_{m-2}^k)}(s_2^k,x_0+B_{s_2^k}^H)
		\frac{\partial}{\partial x_{j_{\lceil k/2\rceil}}}
		b_n^{(\ell_{m-1}^k)}(s_1^k,x_0+B_{s_1^k}^H)
		\Bigg)\Bigg]\Bigg|.
	\end{align*}
	For each fixed choice of $i_1,j_1,\cdots,i_r,j_r$ and
	$\ell_1^k,\cdots,\ell_{m-1}^k$ for $k=1,\cdots,2r$, by shuffling the $2r$ simplices,
	\begin{align*}
		&\int_{(\Delta_{0,t}^m)^{2r}}\Bigg|\E\Bigg[\prod_{k=1}^{2r}\Bigg(
		\frac{\partial}{\partial x_{\ell_1^k}}b_n^{(i_{\lceil k/2\rceil})}
		(s_m^k,x_0+B_{s_m^k}^H)
		\frac{\partial}{\partial x_{\ell_2^k}}b_n^{(\ell_1^k)}
		(s_{m-1}^k,x_0+B_{s_{m-1}^k}^H)\cdots\\
		&\qquad\times
		\frac{\partial}{\partial x_{\ell_{m-1}^k}}
		b_n^{(\ell_{m-2}^k)}(s_2^k,x_0+B_{s_2^k}^H)
		\frac{\partial}{\partial x_{j_{\lceil k/2\rceil}}}
		b_n^{(\ell_{m-1}^k)}(s_1^k,x_0+B_{s_1^k}^H)
		\Bigg)\Bigg]\Bigg|\prod_{k=1}^{2r}\d s^k
	\end{align*}
	can be rewritten as a sum of at most $(2r)^{2rm}$ terms of the form
	\begin{align*}
		\int_{\Delta_{0,t}^{2rm}}\left|\E\left[
		\prod_{k=1}^{2rm}\frac{\partial}{\partial x_{\alpha_k'}}
		b_n^{(\alpha_k)}(s_k,x_0+B_{s_k}^H)\right]\right|\d s.
	\end{align*}
	This is because there are 
	$$\binom{2rm}{m}\binom{(2r-1)m}{m}\cdots\binom{m}{m}
	=\frac{(2rm)!}{(m!)^{2r}}\leq(2r)^{2rm}$$
	ways to choose $m$ elements at a time from $2r\times m$ elements,
	repeated $2r$ times. Consequently, we obtain
	\begin{align}\label{sec5-eq.0}
		&\E[\Vert G_m^n(\cdot,t)\Vert_{\mathcal{H}_\beta}^{2r}]\nonumber\\
		&\qquad\leq C_H^{2r}t^{2r(H-\beta)}\sum_{i_1,j_1,\cdots,i_r,j_r=1}^d
		\sum_{\ell_1^1,\cdots,\ell_{m-1}^1=1}^d\cdots\sum_{\ell_1^{2r},\cdots,\ell_{m-1}^{2r}=1}^d\sum_{\alpha,\alpha'}\nonumber\\
		&\qquad\qquad\times
		\int_{\Delta_{0,t}^{2rm}}\left|\E\left[\prod_{\ell=1}^{2rm}
		\frac{\partial}{\partial x_{\alpha_\ell'}}b_n^{(\alpha_\ell)}(s_\ell,x_0+B_{s_\ell}^H)\right]\right|\d s.
	\end{align}
	
	The first proposition of this subsection is devoted to estimating the expectation of the simplex integral on the right-hand side of the above inequality. In the subsequent subsection, we will see that this estimate also plays a crucial role in controlling $\E^{\mathbb{Q}_n}[\Vert \D^H v^n\Vert_{\mathbb{H}^{-2}([0,T];\mathcal{H}_\beta)}^8]$.
	\begin{proposition}\label{sec5-prop.4}
		Assume \eqref{cond1} and $2\leq p,q<\infty$. Then, for indices $\alpha_1,\cdots,\alpha_n\in\{1,\cdots,d\}$ and continuously differentiable functions $b_1,\cdots,b_n:[0,T]\times \R^d\rightarrow \R$ with compact support, we have
		\begin{align}\label{sec5-eq.13}
			\int_{\Delta_{0,t}^n}\left|\E\left[\prod_{j=1}^n
			\partial_{x^{(\alpha_j)}}b_j(s_j,x_0+B_{s_j}^H)\right]\right|\d s\leq C^n(n\kappa_H)^{-n\kappa_H}t^{n\kappa_H}
			\prod_{j=1}^n\Vert b_j\Vert_{L_x^pL_t^q},
		\end{align}
		where the constant $C$ depends only on $H,d,p,q$ and $T$.
	\end{proposition}
	
	Without loss of generality, we assume that $x_0=0$. Indeed, $\tilde{b}_j(t,\cdot):=b_j(t,x_0+\cdot)$ is continuously differentiable with compact support and $\Vert b_j\Vert_{L_x^pL_t^q}=\Vert\tilde{b}_j\Vert_{L_x^pL_t^q}$ for each $j$. We first make some preparations for the proof of this proposition. By the integration-by-parts formula and Remark \ref{sec2-rm.9}, we have
	\begin{align*}
		&\int_{\Delta_{0,t}^n}\left|\E\left[\prod_{j=1}^n
		\partial_{x^{(\alpha_j)}}b_j(s_j,B_{s_j}^H)\right]\right|\d s\\
		&\qquad=\int_{\Delta_{0,t}^n}\left|\int_{(\R^d)^n}
		\prod_{j=1}^nb_j(s_j,x_j)\partial_{x_1^{(\alpha_1)}}\cdots
		\partial_{x_n^{(\alpha_n)}}p_s(x_1,\cdots,x_n)\d x\right|\d s\\
		&\qquad\leq C^n\int_{\Delta_{0,t}^n}\int_{(\R^d)^n}
		\prod_{j=1}^n|b_j(s_j,x_j)|\Delta_j^{-Hd}\ell_j^{-H}
		\exp\left(-\frac{|x_j-x_{j-1}|^2}{8\Delta_j^{2H}}\right)\d x\d s.
	\end{align*}
	Denote by $\mathcal{I}$ the right-hand side of the above inequality. Note that $\ell_j^{-H}\leq \Delta_j^{-H}+\Delta_{j+1}^{-H}$ for $j=1,\cdots,n-1$. Then, $\mathcal{I}$ is bounded by $C^n$ multiplied by a sum containing at most $2^{n-1}$ terms of the form 
	\begin{align*}
		\mathcal{I}(\beta_1,\cdots,\beta_n)
		&:=\int_{\Delta_{0,t}^n}\int_{(\R^d)^n}
		\prod_{j=1}^n|b_j(s_j,x_j)|\Delta_j^{-H(d+\beta_j)}\\
		&\qquad\times\exp\left(-\frac{|x_j-x_{j-1}|^2}{8\Delta_j^{2H}}\right)\d x\d s.
	\end{align*}
	where the $n$-tuple $(\beta_1,\cdots,\beta_n)$ can be decomposed into blocks of sizes $n_1,\cdots,n_\ell$ such that if 
	$$m_0=0,\;m_1=n_1,\;m_2=n_1+n_2,\cdots,m_\ell=n_1+\cdots+n_\ell=n,$$
	then each block $(\beta_{m_{i-1}+1},\cdots,\beta_{m_i})$ satisfies the following:
	\begin{itemize}
		\item If $n_i=1$, then $\beta_{m_{i-1}+1}=1$;
		\item If $n_i=2$, then $\beta_{m_{i-1}+1}=0$ and $\beta_{m_i}=2$;
		\item If $n_i>2$, then $\beta_{m_{i-1}+1}=0$, $\beta_{m_{i-1}+2}=\cdots=\beta_{m_i-1}=1$ and $\beta_{m_i}=2$.
	\end{itemize}
	Therefore, it suffices to consider the following three types of block integrals for $\alpha\geq0$:
	\begin{itemize}
		\item For $k=1$, $\mathcal{I}_1(b_1):=\int_{s_0}^t\int_{\R^d}b_1(s_1,x_1)(s_1-s_0)^{-Hd-H}e^{-\frac{|x_1-x_0|^2}{8(s_1-s_0)^{2H}}}(t-s_1)^\alpha\d x_1\d s_1;$
		\item For $k=2$, define
		\begin{align*}
			\mathcal{I}_2(b_1,b_2):=&\int_{\Delta_{s_0,t}^2}\int_{\R^{2d}}b_1(s_1,x_1)(s_1-s_0)^{-Hd}e^{-\frac{|x_1-x_0|^2}{8(s_1-s_0)^{2H}}}\\
			&\times b_2(s_2,x_2)(s_2-s_1)^{-Hd-2H}e^{-\frac{|x_2-x_1|^2}{8(s_2-s_1)^{2H}}}(t-s_2)^\alpha\d x\d s;
		\end{align*}
		\item For $k>2$, define
		\begin{align*}
			\mathcal{I}_k(b_1,\cdots,b_k)
			:=&\int_{\Delta_{s_0,t}^k}\int_{(\R^d)^k}
			b_1(s_1,x_1)(s_1-s_0)^{-Hd}e^{-\frac{|x_1-x_0|^2}{8(s_1-s_0)^{2H}}}\\
			&\times\prod_{j=2}^{k-1}\Bigg\{b_j(s_j,x_j)(s_j-s_{j-1})^{-Hd-H}
			e^{-\frac{|x_j-x_{j-1}|^2}{8(s_j-s_{j-1})^{2H}}}\Bigg\}\\
			&\times b_k(s_k,x_k)(s_k-s_{k-1})^{-Hd-2H}
			e^{-\frac{|x_k-x_{k-1}|^2}{8(s_k-s_{k-1})^{2H}}}(t-s_k)^\alpha\d x\d s.
		\end{align*}
	\end{itemize}
	Here, $b_1(t,\cdot),\cdots,b_k(t,\cdot)\in L_x^pL_t^q$ are non-negative functions.
	
	The principal difficulty lies in handling the singular term $(s_2-s_1)^{-Hd-2H}$ in $\mathcal{I}_2(b_1,b_2)$. Once the required estimate for $\mathcal{I}_2(b_1,b_2)$ is established under assumptions \eqref{cond1} and $2\leq p,q<\infty$, the proof of Proposition \ref{sec5-prop.4} follows straightforwardly. 
	
	\begin{lemma}\label{sec5-lm.1}
		Assume \eqref{cond1} and $2\leq p,q<\infty$. Let $b_1(t,\cdot)$ and $b_2(t,\cdot)$ be continuously differentiable non-negative functions with compact support, and let $\gamma\geq1$ and $\alpha,\beta\geq0$. Then, we have
		\begin{align}\label{sec5-eq.7}
			I:&=\int_{\Delta_{s_0,t}^2}\int_{\R^{2d}}b_1(s_1,x_1)(s_1-s_0)^{-Hd+\beta}e^{-\frac{|x_1-x_0|^2}{8\gamma(s_1-s_0)^{2H}}}\nonumber\\
			&\qquad\times b_2(s_2,x_2)(s_2-s_1)^{-Hd-2H}e^{-\frac{|x_2-x_1|^2}{8(s_2-s_1)^{2H}}}(t-s_2)^\alpha\d x\d s\nonumber\\
			&\leq C\frac{\alpha^\alpha(\beta+2\kappa_H)^{\beta+2\kappa_H}\gamma^{d}}{(\alpha+\beta+2\kappa_H)^{\alpha+\beta+2\kappa_H}}(t-s_0)^{\alpha+\beta+2\kappa_H}\Vert b_1\Vert_{L_x^pL_t^q}\Vert b_2\Vert_{L_x^pL_t^q},
		\end{align}
		where the constant $C$ depends only on $H,d,p,q$ and $T$.
	\end{lemma}
    
	The proof of Lemma 5.2 is lengthy and largely follows the argument of Rezakhanlou \cite[Lemma 4.1]{rezakhanlou2014regular}. Since the details are mainly technical, we postpone the proof to Appendix B in order to keep the main line of the argument transparent.
	
	With the help of Lemma \ref{sec5-lm.1}, we can derive an upper bound for $\mathcal{I}_k(b_1,\cdots,b_k)$.
	\begin{lemma}\label{sec5-lm.3}
		Assume \eqref{cond1} and $2\leq p,q<\infty$. Let $b_1(t,\cdot),\cdots,b_k(t,\cdot)$ be continuously differentiable non-negative functions with compact support, and let $\alpha\geq0$. Then, we have
		\begin{align}\label{sec5-eq.11}
			\mathcal{I}_k(b_1,\cdots,b_k)
			&\leq C^k\frac{\alpha^\alpha}{(\alpha+k\kappa_H)^{\alpha+k\kappa_H}}
			(t-s_0)^{\alpha+k\kappa_H}\prod_{j=1}^k\Vert b_j\Vert_{L_x^pL_t^q},
		\end{align}
		where the constant $C$ depends only on $H,d,p,q$ and $T$.
	\end{lemma}
	\begin{proof}
		Without loss of generality, we assume that $s_0=0$ and $x_0=0$. When $k=1$, the proof of \eqref{sec5-eq.11} is an immediate consequence of H\"older's inequality, Lemma \ref{appendix-lm.4} and Stirling's formula. In addition, Lemma \ref{sec5-lm.1} gives the desired result \eqref{sec5-eq.11} for $k=2$. Therefore, it remains to consider the case $k\geq 3$. First, we rewrite $\mathcal{I}_k(b_1,\cdots,b_k)$ as
		\begin{align}\label{sec5-eq.12}
			\mathcal{I}_k(b_1,\cdots,b_k)=&\int_0^t\int_0^{s_k}\int_{\R^{2d}}A(s_{k-1},x_{k-1})b_k(s_k,x_k)(s_k-s_{k-1})^{-Hd-2H}\nonumber\\
			&\qquad \times e^{-\frac{|x_k-x_{k-1}|^2}{8(s_k-s_{k-1})^{2H}}}(t-s_k)^\alpha\d x_{k-1}\d x_k\d s_{k-1}\d s_k,
		\end{align}
		where $A(s_{k-1},x_{k-1})$ is defined by
		\begin{align*}
			A(s_{k-1},x_{k-1}):=&\,b_{k-1}(s_{k-1},x_{k-1})\int_{(\R^d)^{k-2}}\int_{\Delta_{0,s_{k-1}}^{k-2}}b_1(s_1,x_1)s_1^{-Hd}e^{-\frac{|x_1|^2}{8s_1^{2H}}}\\
			&\qquad\times \prod_{j=2}^{k-2}\Bigg\{b_j(s_j,x_j)(s_j-s_{j-1})^{-Hd-H}e^{-\frac{|x_j-x_{j-1}|^2}{8(s_j-s_{j-1})^{2H}}}\Bigg\}\\
			&\qquad\times(s_{k-1}-s_{k-2})^{-Hd-H}e^{-\frac{|x_{k-1}-x_{k-2}|^2}{8(s_{k-1}-s_{k-2})^{2H}}}\prod_{j=1}^{k-2}\d s_j\d x_j.
		\end{align*}
		Define the heat kernel for fBm by
		$$p^H(s,x)=\frac1{(8\pi)^{\frac d2}s^{Hd}}\exp\left(-\frac{|x|^2}{8s^{2H}}\right)$$
		for $s>0$. Then, by H\"older's inequality, we have
		\begin{align*}
			&\int_{(\R^d)^{k-2}}\int_{\Delta_{0,s_{k-1}}^{k-2}}b_1(s_1,x_1)s_1^{-Hd}e^{-\frac{|x_1|^2}{8s_1^{2H}}}\\
			&\qquad\times \prod_{j=2}^{k-2}\Bigg\{b_j(s_j,x_j)(s_j-s_{j-1})^{-Hd-H}e^{-\frac{|x_j-x_{j-1}|^2}{8(s_j-s_{j-1})^{2H}}}\Bigg\}\\
			&\qquad\times(s_{k-1}-s_{k-2})^{-Hd-H}e^{-\frac{|x_{k-1}-x_{k-2}|^2}{8(s_{k-1}-s_{k-2})^{2H}}}\d s\d x\\
			&\leq \prod_{j=1}^{k-2}\Vert b_j\Vert_{L_x^pL_t^q}\left(\int_{\Delta_{0,s_{k-1}}^{k-2}}|B(s_1,\cdots,s_{k-2},s_{k-1},x_{k-1})|^{\frac {q'}{p'}}\prod_{j=1}^{k-2}\d s_j\right)^{\frac1{q'}},
		\end{align*}
		where
		\begin{align*}
			B(s_1,\cdots,s_{k-2},s_{k-1},x_{k-1}):=&\left(\frac{8\pi}{p'}\right)^{\frac{(k-1)d}2}s_1^{-Hd(p'-1)}\prod_{j=2}^{k-1}(s_j-s_{j-1})^{-Hp'-Hd(p'-1)}\\
			&\times \int_{(\R^d)^{k-2}}\prod_{j=1}^{k-1}p^H((s_j-s_{j-1})/(p')^{\frac1{2H}},x_j-x_{j-1})\prod_{j=1}^{k-2}\d x_j.
		\end{align*}
		By the Gaussian convolution formula, we have
		\begin{align*}
			&\int_{(\R^d)^{k-2}}\prod_{j=1}^{k-1}p^H((s_j-s_{j-1})/(p')^{\frac1{2H}},x_j-x_{j-1})\prod_{j=1}^{k-2}\d x_j\\
			&\qquad=\left(\frac{p'}{8\pi}\right)^{\frac d2}\left(\sum_{j=1}^{k-1}(s_j-s_{j-1})^{2H}\right)^{-\frac d2}\exp\left(-\frac{p'|x_{k-1}|^2}{8\sum_{j=1}^{k-1}(s_j-s_{j-1})^{2H}}\right).
		\end{align*}
		Given that 
		$$s_{k-1}^{2H}\leq \sum_{j=1}^{k-1}(s_j-s_{j-1})^{2H}\leq (k-1)^{1-2H}s_{k-1}^{2H}$$
		for $H<1/2$, we have
		\begin{align*}
			B(s_1,\cdots,s_{k-2},s_{k-1},x_{k-1})\leq&\left(\frac{8\pi}{p'}\right)^{\frac{(k-1)d}2}s_1^{-Hd(p'-1)}\prod_{j=2}^{k-1}(s_j-s_{j-1})^{-Hp'-Hd(p'-1)}\\
			&\times(k-1)^{\frac{(1-2H)d}{2}} p^H((k-1)^{\frac{1-2H}{2H}}s_{k-1}/(p')^{\frac1{2H}},x_{k-1}).
		\end{align*}
		Consequently, Lemma \ref{appendix-lm.4} yields
		\begin{align*}
			&\left(\int_{\Delta_{0,s_{k-1}}^{k-2}}|B(s_1,\cdots,s_{k-2},s_{k-1},x_{k-1})|^{\frac{q'}{p'}}\prod_{j=1}^{k-2}\d s_j\right)^{\frac1{q'}}\\
			&\qquad\leq C^{k-2} s_{k-1}^{-\frac{Hd}{p'}}e^{-\frac{|x_{k-1}|^2}{8(k-1)^{1-2H}s_{k-1}^{2H}}}\\
			&\qquad\qquad\times\left(\int_{\Delta_{0,s_{k-1}}^{k-2}}s_1^{-q'\frac{Hd}p}\prod_{j=2}^{k-1}(s_j-s_{j-1})^{-q'(H+\frac{Hd}p)}\prod_{j=1}^{k-2}\d s_j\right)^{\frac1{q'}}\\
			&\qquad\leq C^{k-2}\frac{\Gamma(1-q'\frac{Hd}p)^{\frac 1{q'}}\Gamma(q'\kappa_H)^{\frac{k-2}{q'}}}{\Gamma(q'(H+(k-1)\kappa_H))^{\frac1{q'}}}s_{k-1}^{(k-2)\kappa_H-Hd}e^{-\frac{|x_{k-1}|^2}{8(k-1)^{1-2H}s_{k-1}^{2H}}}.
		\end{align*}
		Then, by Stirling's formula, we obtain
		\begin{align*}
			A(s_{k-1},x_{k-1})
			&\leq\frac{C^{k-2}s_{k-1}^{(k-2)\kappa_H}}
			{((k-2)\kappa_H)^{(k-2)\kappa_H}}
			\prod_{j=1}^{k-2}\Vert b_j\Vert_{L_x^pL_t^q}\\
			&\qquad\times b_{k-1}(s_{k-1},x_{k-1})s_{k-1}^{-Hd}
			e^{-\frac{|x_{k-1}|^2}{8(k-1)^{1-2H}s_{k-1}^{2H}}}.
		\end{align*}
		Substituting the above upper bound of $A(s_{k-1},x_{k-1})$ into the expression \eqref{sec5-eq.12} of $\mathcal{I}_k(b_1,\cdots,b_k)$ and invoking Lemma \ref{sec5-lm.1} with $\beta=(k-2)\kappa_H$ and $\gamma=(k-1)^{1-2H}$, we obtain
		\begin{align*}
			\mathcal{I}_k(b_1,\cdots,b_k)&\leq C^{k-2}((k-2)\kappa_H)^{-(k-2)\kappa_H}\prod_{j=1}^{k-2}\Vert b_j\Vert_{L_x^pL_t^q}\\
			&\qquad\times\int_0^t\int_0^{s_k}\int_{\R^{2d}}b_{k-1}(s_{k-1},x_{k-1})s_{k-1}^{-Hd+(k-2)\kappa_H}e^{-\frac{|x_{k-1}|^2}{8(k-1)^{1-2H}s_{k-1}^{2H}}}\nonumber\\
			&\qquad\qquad\times b_k(s_k,x_k)(s_k-s_{k-1})^{-Hd-2H}e^{-\frac{|x_k-x_{k-1}|^2}{8(s_k-s_{k-1})^{2H}}}(t-s_k)^\alpha\d x_{k-1}\d x_k\d s_{k-1}\d s_k\\
			&\leq C^k\prod_{j=1}^k\Vert b_j\Vert_{L_x^pL_t^q}t^{\alpha+k\kappa_H}(k-1)^{(1-2H)d}((k-2)\kappa_H)^{-(k-2)\kappa_H}\frac{\alpha^\alpha(k\kappa_H)^{k\kappa_H}}{(\alpha+k\kappa_H)^{\alpha+k\kappa_H}}\\
			&\leq C^k\frac{\alpha^\alpha}{(\alpha+k\kappa_H)^{\alpha+k\kappa_H}}t^{\alpha+k\kappa_H}\prod_{j=1}^k\Vert b_j\Vert_{L_x^pL_t^q}.
		\end{align*}
		This follows since there exists a constant $C$ such that $(k-1)^{(1-2H)d}\leq C^k$ and $(k\kappa_H)^{k\kappa_H}/((k-2)\kappa_H)^{(k-2)\kappa_H}\leq C^k$. The proof is complete.
	\end{proof}
	
	With the above preparations in place, we can give the proof of Proposition \ref{sec5-prop.4}.
	\\\textit{Proof of Proposition \ref{sec5-prop.4}}. Recall some notations we set previously. If $\ell=1$, then Lemma \ref{sec5-lm.3} implies the desired estimate \eqref{sec5-eq.13}. For $\ell\geq2$, we define
	\begin{align*}
		\mathcal{J}_{\ell-1}(s_{m_{\ell-1}},x_{m_{\ell-1}}):=&\int_{\Delta_{s_{m_{\ell-1}},t}^{n_\ell}}\int_{(\R^d)^{n_\ell}}\prod_{j=m_{\ell-1}+1}^n\Bigg\{|b_j(s_j,x_j)|\\
		&\times(s_j-s_{j-1})^{-Hd-H\beta_j}e^{-\frac{|x_j-x_{j-1}|^2}{8(s_j-s_{j-1})^{2H}}}\Bigg\}\prod_{j=m_{\ell-1}+1}^n\d x_j\d s_j.
	\end{align*}
	Inductively, we define
	\begin{align*}
		\mathcal{J}_{k-1}(s_{m_{k-1}},x_{m_{k-1}}):=&\int_{\Delta_{s_{m_{k-1}},t}^{n_k}}\int_{(\R^d)^{n_k}}\mathcal{J}_{k}(s_{m_k},x_{m_k})\prod_{j=m_{k-1}+1}^{m_k}\Bigg\{|b_j(s_j,x_j)|\\
		&\times(s_j-s_{j-1})^{-Hd-H\beta_j}e^{-\frac{|x_j-x_{j-1}|^2}{8(s_j-s_{j-1})^{2H}}}\Bigg\}\prod_{j=m_{k-1}+1}^{m_k}\d x_j\d s_j.
	\end{align*}
	for $k=\ell-1,\cdots,2$. Then, $\mathcal{I}(\beta_1,\cdots,\beta_n)$ can be expressed as
	\begin{align*}
		\mathcal{I}(\beta_1,\cdots,\beta_n)=&\int_{\Delta_{0,t}^{n_1}}\int_{(\R^d)^{n_1}}\mathcal{J}_1(s_{m_1},x_{m_1})\prod_{j=1}^{m_1}\Bigg\{|b_j(s_j,x_j)|\\
		&\times(s_j-s_{j-1})^{-Hd-H\beta_j}e^{-\frac{|x_j-x_{j-1}|^2}{8(s_j-s_{j-1})^{2H}}}\Bigg\}\prod_{j=1}^{m_1}\d x_j\d s_j.
	\end{align*}
	By Lemma \ref{sec5-lm.3}, we have
	$$
	\begin{aligned}
		\mathcal{J}_{\ell-1}(s_{m_{\ell-1}},x_{m_{\ell-1}})
		&\leq C^{n_\ell}(n_\ell\kappa_H)^{-n_\ell\kappa_H}
		(t-s_{m_{\ell-1}})^{n_\ell\kappa_H}\\
		&\qquad\times\prod_{j=m_{\ell-1}+1}^n\Vert b_j\Vert_{L_x^pL_t^q}.
	\end{aligned}
	$$
	Then, applying Lemma \ref{sec5-lm.3} again with $\alpha=n_{\ell}\kappa_H$ leads to
	\begin{align*}
		\mathcal{J}_{\ell-2}(s_{m_{\ell-2}},x_{m_{\ell-2}})
		&\leq C^{n_\ell}(n_\ell\kappa_H)^{-n_\ell\kappa_H}
		\prod_{j=m_{\ell-1}+1}^n\Vert b_j\Vert_{L_x^pL_t^q}\\
		&\quad\times\int_{\Delta_{s_{m_{\ell-2}},t}^{n_{\ell-1}}}
		\int_{(\R^d)^{n_{\ell-1}}}(t-s_{m_{\ell-1}})^{n_\ell\kappa_H}
		\prod_{j=m_{\ell-2}+1}^{m_{\ell-1}}\Bigg\{|b_j(s_j,x_j)|\\
		&\qquad\times(s_j-s_{j-1})^{-Hd-H\beta_j}
		e^{-\frac{|x_j-x_{j-1}|^2}{8(s_j-s_{j-1})^{2H}}}\Bigg\}
		\prod_{j=m_{\ell-2}+1}^{m_{\ell-1}}\d x_j\d s_j\\
		&\leq C^{n_{\ell-1}+n_\ell}
		((n_{\ell-1}+n_\ell)\kappa_H)^{-(n_{\ell-1}+n_\ell)\kappa_H}\\
		&\qquad\times(t-s_{m_{\ell-2}})^{(n_{\ell-1}+n_\ell)\kappa_H}
		\prod_{j=m_{\ell-2}+1}^n\Vert b_j\Vert_{L_x^pL_t^q}.
	\end{align*}
	Consequently, we complete the proof of Proposition \ref{sec5-prop.4} by iteration.
	
	We now state the main proposition of this subsection.
	\begin{proposition}\label{sec5-prop.6}
		Assume \eqref{cond1} and $2\leq p,q<\infty$. Let $0<\beta<H$. Then, for any $t\in[0,T]$ and all integers $r\geq1$, there exists a constant $C=C(H,d,p,q,T,\beta,r,\Vert b\Vert_{L_x^pL_t^q})$ such that
		\begin{equation}
			\sup_n\E[\Vert \mathcal{I}_{\cdot,t}^n\Vert_{\mathcal{H}_\beta}^{2r}]\leq Ct^{2r(H-\beta)}
		\end{equation}
	\end{proposition}
	\begin{proof}
		Recall the estimate \eqref{sec5-eq.0} and note that $\sup_n\Vert b_n\Vert_{L_x^pL_t^q}\leq \Vert b\Vert_{L_x^pL_t^q}$ by construction. Since $\sum_{i_1,j_1,\cdots,i_r,j_r=1}^d
		\sum_{\ell_1^1,\cdots,\ell_{m-1}^1=1}^d\cdots
		\sum_{\ell_1^{2r},\cdots,\ell_{m-1}^{2r}=1}^d=d^{2rm}$ and $\sum_{\alpha,\alpha'}1\leq(2r)^{2rm}$, Proposition \ref{sec5-prop.4} implies
		$$\sup_n\big\Vert\Vert G_m^n(\cdot,t)\Vert_{\mathcal{H}_\beta}\big\Vert_{L^{2r}(\Omega)}\leq C^mt^{H-\beta+m\kappa_H}(m\kappa_H)^{-m\kappa_H}\Vert b\Vert_{L_x^pL_t^q}^m.$$
		Then, the convergence of the series
		$$\sum_{m=1}^\infty C^m(m\kappa_H)^{-m\kappa_H}\Vert b\Vert_{L_x^pL_t^q}^m,$$
		completes the proof.
	\end{proof}
	\subsection{Controlling $\sup_n\E^{\mathbb{Q}_n}[\Vert \D^Hv^n\Vert_{\mathbb{H}^{-2}([0,T];\mathcal{H}_\beta)}^8]$}
	
	We begin by computing $\D_\theta^Hv_t^n$ for $0<\theta<t$. By the definition of $K_H^{-1}$, we have
	$$v_t^n=\frac{1}{\Gamma(\frac12-H)}t^{H-\frac12}\int_0^t(t-s)^{-\frac12-H}s^{\frac12-H}b_n(s,X_s^n)\d s.$$
	Then, it holds a.e. that
	$$\D_\theta^Hv_t^n=\frac{1}{\Gamma(\frac12-H)}t^{H-\frac12}\int_0^t(t-s)^{-\frac12-H}s^{\frac12-H}\nabla b_n(s,X_s^n)\D_\theta^HX_s^n\d s.$$
	Since $X^n=x_0+\tilde{B}^{H,n}$ under the probability measure $\mathbb{Q}_n$, it holds for all integers $r\geq1$ that
	\begin{equation}\label{sec5-eq.16}
		\E^{\mathbb{Q}_n}[\Vert\D^Hv^n\Vert_{\mathbb{H}^{-2}([0,T];\mathcal{H}_\beta)}^{2r}]=\E[\Vert K_H^{-1}(\mathcal{I}^n)\Vert_{\mathbb{H}^{-2}([0,T];\mathcal{H}_\beta)}^{2r}]
	\end{equation}
	To derive a uniform estimate for $\E[\Vert K_H^{-1}(\mathcal{I}^n)\Vert_{\mathbb{H}^{-2}([0,T];\mathcal{H}_\beta)}^{2r}]$, we need the following lemma.
	\begin{lemma}\label{sec5-lm.5}
		Let $E$ be a Hilbert space, $H<1/2$ and $\rho>H$. Assume that $f:[0,T]\rightarrow E$ is absolutely continuous with $f(0)=0$ and
$ t^{-\rho}f(t)\in L^2([0,T];E).$
In the present statement, $K_H^{-1}f$ is understood through the
fractional-derivative formula
$$
    (K_H^{-1}f)(t)
    :=
    c_H t^{H-\frac12}
    \int_0^t
        (t-s)^{-\frac12-H}s^{\frac12-H}f'(s)\d s,
$$
initially as an $E$-valued distribution on $(0,T)$.  Then,
		\begin{equation}\label{sec5-eq.15}
			\Vert K_H^{-1}f\Vert_{\mathbb{H}^{-2}([0,T];E)}\leq C_{H,T,\rho}\Vert t^{-\rho}f\Vert_{L^2([0,T];E)}.
		\end{equation}
If $f$ belongs to the original operator domain of $K_H^{-1}$, the distribution defined above coincides with the usual inverse.
	\end{lemma}	
\begin{proof}
Let $g:=K_H^{-1}f$ in the distributional sense specified above.
Let $(\varphi_k,\lambda_k)_{k\ge1}$ be the Dirichlet eigenbasis and
eigenvalues on $[0,T]$. Formally, the $k$-th spectral coefficient of
$g$ is
$$
\begin{aligned}
    g_k &=
    c_H\int_0^T
        f'(t)t^{\frac12-H}
        \int_t^T
            (s-t)^{-\frac12-H}
            s^{H-\frac12}\varphi_k(s)\,ds\d t.
\end{aligned}
$$
All the integrals in this proof are Bochner integrals in $E$. The
preceding identity can first be written on $[\varepsilon,T]$ and then
passed to the limit $\varepsilon\downarrow0$.

Define
$$
    F_{\varphi_k}(t)
    :=
    c_Ht^{\frac12-H}
    \int_t^T
        (s-t)^{-\frac12-H}s^{H-\frac12}\varphi_k(s)\d s.
$$
Then
$$
    g_k=\int_0^T f'(t)F_{\varphi_k}(t)\d t.
$$
One has
\begin{align}\label{sec5-eq.14}
    |F_{\varphi_k}(t)|
    \le
    C_{H,T}\|\varphi_k\|_{L_t^\infty}
    t^{\frac12-H}(1+|\log t|),
\end{align}
and hence $F_{\varphi_k}(t)\to0$ as $t\downarrow0$. Moreover,
$F_{\varphi_k}(T)=0$. Since $f(0)=0$, the boundary terms in the
integration-by-parts formula vanish, and therefore
$$
    g_k=-\int_0^T f(t)F_{\varphi_k}'(t)\d t.
$$
		By a straightforward calculation, we have
		\begin{align*}
			F_{\varphi_k}'(t)&=c_H\left(\frac12-H\right)t^{-\frac12-H}\int_0^{T-t}r^{\frac12-H}(t+r)^{H-\frac32}\varphi_k(t+r)\d r\\
			&+c_Ht^{\frac12-H}\int_0^{T-t}r^{-\frac12-H}(t+r)^{H-\frac12}\varphi_k'(t+r)\d r
		\end{align*}
		Following a similar argument as in the estimate \eqref{sec5-eq.14}, we obtain
		\begin{align*}
			|F_{\varphi_k}'(t)|\leq C_{H,T}(1+|\log t|)\left(t^{-\frac12-H}\Vert\varphi_k\Vert_{L_t^\infty}+t^{\frac12-H}\Vert\varphi_k'\Vert_{L_t^\infty}\right).
		\end{align*}
		Given that $\rho>H$, $\rho-1/2-H>-1/2$, which implies $t^\rho F_{\varphi_k}'(t)$ is $L^2$ integrable with the estimate
		$$\int_0^Tt^{2\rho}|F_{\varphi_k}'(t)|^2\d t\leq C_{H,T,\rho}\left(\Vert\varphi_k\Vert_{L_t^\infty}^2+\Vert\varphi_k'\Vert_{L_t^\infty}^2\right).$$
		By the Cauchy-Schwarz inequality, we obtain
		\begin{align*}
			\Vert g_k\Vert_E&\leq \int_0^T\Vert t^{-\rho}f\Vert_E\,t^\rho |F_{\varphi_k}'(t)|\d t\\
			&\leq \Vert t^{-\rho}f\Vert_{L^2([0,T];E)}\left(\int_0^Tt^{2\rho}|F_{\varphi_k}'(t)|^2\d t\right)^{\frac12}\\
			&\leq C_{H,T,\rho}\Vert t^{-\rho}f\Vert_{L^2([0,T];E)}\sqrt{1+\lambda_k},
		\end{align*} 
		where we use the fact that $\Vert\varphi_k\Vert_{L_t^\infty}\leq \sqrt{2/T}$ and $\Vert\varphi_k'\Vert_{L_t^\infty}\leq\sqrt{2\lambda_k/T}$ in the last inequality. Consequently, we obtain
		\begin{align*}
			\Vert g\Vert_{\mathbb{H}^{-2}([0,T];E)}^2=\sum_{k=1}^\infty\lambda_k^{-2}\Vert g_k\Vert_E^2
			\leq  C_{H,T,\rho}\Vert t^{-\rho}f\Vert_{L^2([0,T];E)}^2\sum_{k=1}^\infty\lambda_k^{-2}(1+\lambda_k).
		\end{align*}
		\eqref{sec5-eq.15} follows from the fact that $\sum_{k=1}^\infty\lambda_k^{-2}(1+\lambda_k)<\infty$. 
The preceding square-summability estimate also proves that the distribution defined by the fractional-derivative formula belongs to $H^{-2}([0,T];E)$, thereby justifying the interpretation of $K_H^{-1}f$ used in the statement.
 The proof is complete.
	\end{proof}
	
	Applying the above lemma and Proposition \ref{sec5-prop.6}, we can now control the term  $\E^{\mathbb{Q}_n}[\Vert \D^Hv^n\Vert_{\mathbb{H}^{-2}([0,T];\mathcal{H}_\beta)}^8]$.
	\begin{proposition}\label{sec5-prop.8}
		Assume \eqref{cond1} and $2\leq p,q<\infty$. Let $0<\beta<H$. Then, for all integers $r\geq1$, we have
		\begin{equation}
			\sup_n \E^{\mathbb{Q}_n}[\Vert \boldsymbol{\mathrm{D}}^Hv^n\Vert_{\mathbb{H}^{-2}([0,T];\mathcal{H}_\beta)}^{2r}]<\infty.
		\end{equation}
	\end{proposition}
    \begin{proof}
		Choose $H<\rho<H+1/2-\beta$. Note that 
		$$\mathcal{I}_{\cdot,t}^n=\int_0^t\nabla b_n(s,x_0+B_s^H)(\mathbbm{1}_{[0,s]}(\cdot)+\mathcal{I}_{\cdot,s}^n)\d s.$$
		Since $b_n\in C_c^\infty([0,T]\times \R^d;\R^d)$, it follows from Lemma \ref{appendix-lm.6} and Proposition \ref{sec5-prop.6} that
		$$\E\int_0^T\Vert\nabla b_n(t,x_0+B_t^H)(\mathbbm{1}_{[0,t]}(\cdot)+\mathcal{I}_{\cdot,t}^n)\Vert_{\mathcal{H}_\beta}\d t
		\leq C\Vert\nabla b_n\Vert_\infty\int_0^Tt^{H-\beta}\d t<\infty.$$
		Hence, the integrand $\nabla b_n(t,x_0+B_t^H)(\mathbbm{1}_{[0,t]}(\cdot)+\mathcal{I}_{\cdot,t}^n)$ is $\mathbb{P}\text{-a.s.}$ integrable in $\mathcal{H}_\beta$ and 
		$$\frac{\d}{\d t}\mathcal{I}^n_{\cdot,t}=\nabla b_n(t,x_0+B_t^H)(\mathbbm{1}_{[0,t]}(\cdot)+\mathcal{I}_{\cdot,t}^n)$$
		$\mathbb{P}\text{-a.s.}$ for a.e. $t\in[0,T]$, which implies that $t\mapsto\mathcal{I}^n_{\cdot,t}$ is $\mathbb{P}\text{-a.s.}$ absolutely continuous as an $\mathcal{H}_\beta$-valued map. Then, by \eqref{sec5-eq.16}, Lemma \ref{sec5-lm.5} and Proposition \ref{sec5-prop.6}, we have
		\begin{align*}
			&\sup_n \E^{\mathbb{Q}_n}[\Vert \boldsymbol{\mathrm{D}}^Hv^n\Vert_{\mathbb{H}^{-2}([0,T];\mathcal{H}_\beta)}^{2r}]\\
			&\qquad\leq C \sup_n\left\Vert \left(\int_0^Tt^{-2\rho}\Vert\mathcal{I}^n_{\cdot,t}\Vert_{\mathcal{H}_\beta}^2\d t\right)^{\frac12}\right\Vert_{L^{2r}(\Omega)}^{2r}\\
			& \qquad\leq C\sup_n\left(\int_0^Tt^{-2\rho}\big\Vert\Vert\mathcal{I}^n_{\cdot,t}\Vert_{\mathcal{H}_\beta}\big\Vert_{L^{2r}(\Omega)}^2\d t\right)^r\leq C\left(\int_0^Tt^{2H-2\beta-2\rho}\d t\right)^r<\infty,
		\end{align*}
		where we use Minkowski's inequality in the second and third inequalities. The proof is complete. 
	\end{proof}
    
	\section{Proof of the Main Result}\label{sec6}
	
	Combining Lemmas \ref{sec3-lm.2}-\ref{sec3-lm.4}, Corollaries \ref{sec4-col.3} and \ref{sec4-col.8}, Propositions \ref{sec5-prop.6} and \ref{sec5-prop.8}, we can now establish the relative compactness of $\{X_t^n\}_{n=1}^\infty$ for each $t$.
	\begin{lemma}\label{sec6-lm.1}
		Let $\{b_n\}_{n=1}^\infty\subset C_c^\infty([0,T]\times\R^d)$ be an approximating sequence of $b$ and $\{X^n\}_{n=1}^\infty$ be the strong solutions to the corresponding SDEs. Assume \eqref{cond1} and $2\leq p,q<\infty$. Then, for each $t\in[0,T]$, $\{X_t^n\}_{n=1}^\infty$ is relatively compact in $L^2(\Omega)$.
	\end{lemma}
	\begin{proof}
		Since $X^n$ under $\mathbb{Q}_n$ has the same law as $x_0+B^H$ under $\mathbb{P}$, Corollaries \ref{sec4-col.3} and \ref{sec4-col.8} demonstrate that
		$$\sup_n\E^{\mathbb{Q}_n}[\Vert v_\cdot^n\Vert_{\mathbb{H}^\delta}^{2m}]<\infty,\quad\forall m\in\mathbb{N}\quad\text{and}\quad\sup_n\E[\Vert v_\cdot^n\Vert_{\mathbb{H}^\delta}^2]<\infty$$
		for some $\delta>0$. Hence, Propositions \ref{sec5-prop.6} and \ref{sec5-prop.8}, together with Corollary \ref{sec4-col.3}, imply that the three conditions in Lemma \ref{sec3-lm.3} are satisfied. As a consequence, the truncated sequence $\{F^n_L\}_{n=1}^\infty$ is relatively compact in $L^2(\Omega)$. Finally, Corollary \ref{sec4-col.8} and Lemma \ref{sec3-lm.2} give the desired result. The proof is complete.
	\end{proof}
	
	\noindent\textit{Proof of Theorem \ref{sec1-thm.1}}. Let $\mathbb{Q}_T:=([0,T]\cap\mathbb{Q})\cup\{T\}$. By Lemma \ref{sec6-lm.1} and a standard diagonal argument, there exists a subsequence of $\{X^n\}_{n=1}^\infty$ (still denoted by $\{X^n\}_{n=1}^\infty$) and a family of random variables $\{X_t\}_{t\in\mathbb{Q}_T}$ such that
	\begin{equation}\label{sec6-eq.1}
		X_t^n\xrightarrow{L^2(\Omega)\text{ and a.s.}}X_t,\qquad\forall\,t\in\mathbb{Q}_T.
	\end{equation}
	Set
	$$\gamma:=H\wedge\left(1-H-\frac{Hd}{p}-\frac{1-H}{q}\right)>0.$$
	For any $0\leq s<t\leq T$ and $m\geq2$, Lemma \ref{sec3-lm.1}, together with the fact that $\sup_n\Vert b_n\Vert_{L_x^pL_t^q}\leq\Vert b\Vert_{L_x^pL_t^q}$, yields
	\begin{align*}
		\E|X_t^n-X_s^n|^m
		&\leq 2^{m-1}\E\left|\int_s^tb_n(r,X_r^n)\d r\right|^m+2^{m-1}\E|B_t^H-B_s^H|^m\\
		&\leq C\left(|t-s|^{m(1-H-\frac{Hd}{p}-\frac{1-H}{q})}+|t-s|^{mH}\right)
		\leq C_T|t-s|^{m\gamma},
	\end{align*}
	where the constants are independent of $n$. By Fatou's lemma and \eqref{sec6-eq.1}, the same estimate holds for $X$ on $\mathbb{Q}_T$. Choosing $m$ sufficiently large such that $m\gamma>1$, the Kolmogorov continuity theorem shows that $X$ admits a continuous extension to $[0,T]$, still denoted by $X$.
	
	Fix $\vartheta\in(0,\gamma-1/m)$. By the quantitative Kolmogorov continuity theorem, we have
	$$\sup_n\E[\Vert X^n\Vert_{C^\vartheta([0,T])}^m]<\infty.$$
	Consequently, for all $\varepsilon>0$,
	\begin{equation}\label{sec6-eq.2}
		\lim_{\delta\downarrow0}\sup_n\mathbb{P}\left(\sup_{\substack{s,t\in[0,T]\\|t-s|\leq\delta}}|X_t^n-X_s^n|>\varepsilon\right)=0.
	\end{equation}
	Let $0=t_0<\cdots<t_N=T$ be a partition with $t_j\in\mathbb{Q}_T$ and mesh size at most $\delta$. Then,
	\begin{align*}
		&\mathbb{P}\left(\sup_{t\in[0,T]}|X_t^n-X_t|>3\varepsilon\right)\\
		&\qquad\leq\mathbb{P}\left(\max_{0\leq j\leq N}|X_{t_j}^n-X_{t_j}|>\varepsilon\right)
		+\mathbb{P}\left(\sup_{\substack{s,t\in[0,T]\\|t-s|\leq\delta}}|X_t^n-X_s^n|>\varepsilon\right)\\
		&\qquad\qquad+\mathbb{P}\left(\sup_{\substack{s,t\in[0,T]\\|t-s|\leq\delta}}|X_t-X_s|>\varepsilon\right).
	\end{align*}
	For fixed $\delta$, the first term tends to zero as $n\rightarrow\infty$ by \eqref{sec6-eq.1}. Letting first $n\rightarrow\infty$ and then $\delta\downarrow0$, and using \eqref{sec6-eq.2} and the continuity of $X$, we obtain
	$$X^n\xrightarrow{n\rightarrow\infty}X\quad\text{in probability in }C([0,T];\R^d).$$
	Passing to a further subsequence, we may assume without loss of generality that
	\begin{equation}\label{sec6-eq.3}
		\sup_{t\in[0,T]}|X_t^n-X_t|\xrightarrow{n\rightarrow\infty}0,\qquad\mathbb{P}\text{-a.s.}
	\end{equation}
	
	We next establish a Krylov estimate for the limit process. Let $\varphi\in C_c^\infty([0,T]\times\R^d)$ be a non-negative function. By \eqref{sec6-eq.3}, the dominated convergence theorem and Lemma \ref{sec3-lm.1}, we have
	\begin{align*}\label{sec6-eq.4}
		\E\int_0^T\varphi(s,X_s)\d s
		&=\lim_{n\rightarrow\infty}\E\int_0^T\varphi(s,X_s^n)\d s\\
		&\leq\liminf_{n\rightarrow\infty}\left\Vert\int_0^T\varphi(s,X_s^n)\d s\right\Vert_{L^2(\Omega)}
		\leq C\Vert\varphi\Vert_{L_x^pL_t^q}.
	\end{align*}
	By a standard approximating argument, the above estimate holds for all non-negative $\varphi\in L_x^pL_t^q$, i.e.
	\begin{equation}\label{sec6-eq.4}
		\E\int_0^T\varphi(s,X_s)\d s\leq C\Vert\varphi\Vert_{L_x^pL_t^q},
		\qquad\forall\,\varphi\in L_x^pL_t^q,\quad\varphi\geq0.
	\end{equation}
	Taking $\varphi=|b|$ in \eqref{sec6-eq.4}, we further obtain $\mathbb{E}\int_0^T |b(s,X_s)|\d s\leq C\|b\|_{L_x^pL_t^q}$
    Consequently,
    $$\int_0^T |b(s,X_s)|\d s<\infty,\qquad \mathbb{P}\text{-a.s.}$$
    
	For each $N\in\mathbb{N}$, we have
	\begin{align*}
		&\E\sup_{t\in[0,T]}\left|\int_0^tb(s,X_s)\d s-\int_0^tb_n(s,X_s^n)\d s\right|\\
		&\qquad\leq\E\int_0^T|b-b_N|(s,X_s)\d s+\E\int_0^T|b_N-b_n|(s,X_s^n)\d s\\
		&\qquad\qquad+\E\int_0^T|b_N(s,X_s)-b_N(s,X_s^n)|\d s.
	\end{align*}
	By \eqref{sec6-eq.4} and Lemma \ref{sec3-lm.1}, the first and second terms on the right-hand side are bounded by
	$C\Vert b-b_N\Vert_{L_x^pL_t^q}$ and $C\Vert b_N-b_n\Vert_{L_x^pL_t^q}$, respectively. Since $b_N$ is uniformly continuous on $[0,T]\times\R^d$, \eqref{sec6-eq.3} implies
	$$\sup_{s\in[0,T]}|b_N(s,X_s^n)-b_N(s,X_s)|\xrightarrow{n\rightarrow\infty}0,\quad\mathbb{P}\text{-a.s.}$$
	Moreover,
	$$\sup_{s\in[0,T]}|b_N(s,X_s^n)-b_N(s,X_s)|\leq2\Vert b_N\Vert_{L^\infty([0,T]\times\R^d)}.$$
	Therefore, by the dominated convergence theorem, we have
	\begin{equation}\label{sec6-eq.5}
		\E\int_0^T|b_N(s,X_s)-b_N(s,X_s^n)|\d s\xrightarrow{n\rightarrow\infty}0.
	\end{equation}
	Consequently, upon taking the limit as $n\rightarrow\infty$ first, followed by letting $N\rightarrow\infty$, we obtain
	\begin{equation}\label{sec6-eq.6}
		\E\sup_{t\in[0,T]}\left|\int_0^tb(s,X_s)\d s-\int_0^tb_n(s,X_s^n)\d s\right|\longrightarrow0.
	\end{equation}
	Passing to a further subsequence, we may assume that the convergence in \eqref{sec6-eq.6} holds almost surely. Combining this with \eqref{sec6-eq.3}, we have, on a common event of probability one,
    \begin{align*}
    &\sup_{t\in[0,T]}\left|X_t-x_0-\int_0^tb(s,X_s)\d s-B_t^H\right|\\
    &\qquad\leq\sup_{t\in[0,T]}|X_t-X_t^n|+\sup_{t\in[0,T]}\left|\int_0^tb(s,X_s)\d s-\int_0^tb_n(s,X_s^n)\d s\right|\xrightarrow{n\rightarrow\infty} 0.
    \end{align*}
    Therefore,
    $$X_t-x_0-\int_0^tb(s,X_s)\d s=B_t^H$$
    holds simultaneously for all $t\in[0,T]$, $\mathbb{P}$-a.s. Hence, we conclude that the limit process $X$ is a solution to the SDE \eqref{sec1-eq.1}. Finally, for any $t\in[0,T]$, $X_t^n$ is measurable with respect to the completed natural filtration $\mathcal{F}^{B^H}_t$. Since $X_t^n\rightarrow X_t$ in probability, $X_t$ is measurable with respect to the same filtration. Therefore, $X$ is a strong solution to \eqref{sec1-eq.1}. The proof is complete.
	
	\begin{appendices}
		\section{Technical Lemmas}\label{A}
		\begin{lemma}\label{appendix-lm.1}
			For $0<\beta\leq1,0<y<x<\infty$, it holds that
			\begin{equation}
				\frac{|y^{\alpha}-x^{\alpha}|}{(x-y)^\beta}\lesssim\left\{\begin{array}{l}
					y^{\alpha-\beta},\qquad \alpha<0;\\
					x^{\alpha-\beta},\qquad \alpha\geq0.
				\end{array}\right.
			\end{equation}
		\end{lemma}
		\begin{proof}
			The proof is elementary.
		\end{proof}
		
		\begin{lemma}\label{appendix-lm.4}
			For all $m\in\mathbb{N}$ and $\alpha_0,\alpha_1,\cdots,\alpha_m>0$,
			\begin{align}
	\int_{s_0<s_1<\cdots<s_m<s_{m+1}}\prod_{j=0}^m(s_{j+1}-s_j)^{\alpha_j-1}\prod_{j=1}^m\d s_j
				=(s_{m+1}-s_0)^{\sum_{j=0}^m\alpha_j-1}\frac{\prod_{j=0}^m\Gamma(\alpha_j)}{\Gamma\left(\sum_0^m\alpha_j\right)}
			\end{align}
		\end{lemma}
		\begin{proof}
			By a straightforward calculation.
		\end{proof}
		\begin{lemma}\label{appendix-lm.5}
			Let $p_\sigma$ be the $d$-dimensional Gaussian density with mean zero and covariance matrix $\sigma I_d$, that is
			$$p_\sigma(x)=(2\pi\sigma)^{-\frac d2}\exp\left(-\frac{|x|^2}{2\sigma}\right).$$ For $1\leq r\leq\infty$, $k=0,1$, and $h\in L^r(\R^d)$, we have
			\begin{equation}\label{appendix-eq.heat-kernel}
				\Vert\nabla^k(p_\sigma*h)\Vert_{L_x^\infty}\lesssim\sigma^{-\frac{d}{2r}-\frac{k}{2}}\Vert h\Vert_{L_x^r}.
			\end{equation}
			Now let $h\in L^1(\R^d)$ be a non-negative function. Set $w(x)=(p_\sigma*h)(x)$ and $w_c(x)=(p_{c\sigma}*h)(x)$ for $c>1$. Then, we have
			\begin{equation}\label{appendix-eq.4}
				|\nabla w|^2\lesssim_c \sigma^{-1}ww_c.
			\end{equation}
			Moreover, for $0<\theta<1/c$,
			\begin{equation}\label{appendix-eq.5}
				w_c\lesssim_{c,\theta} \sigma^{-\frac d2(1-\theta)}\Vert h\Vert_{L_x^1}^{1-\theta}w^\theta
			\end{equation}
		\end{lemma}
		\begin{proof}
			Estimate \eqref{appendix-eq.heat-kernel} follows directly from Young's inequality and scaling. For \eqref{appendix-eq.4}, the Cauchy-Schwarz inequality yields 
			\begin{align*}
				|\nabla w|^2&=\left|\int_{\R^d}\nabla p_\sigma(z)h(x-z)\d z\right|^2=\left|\int_{\R^d}\frac{-z}{\sigma} p_\sigma(z)h(x-z)\d z\right|^2\\
				&\leq \sigma^{-1}\left(\int_{\R^d}p_\sigma(z)h(x-z)\d z\right)\left(\int_{\R^d}\frac{|z|^2}{\sigma}p_\sigma(z)h(x-z)\d z\right)\lesssim_c \sigma^{-1}ww_c.
			\end{align*}
			For \eqref{appendix-eq.5}, we first note that
			$$\frac{p_{c\sigma}(x)}{\sigma^{-\frac d2(1-\theta)}p_\sigma^\theta(x)}\leq (2\pi)^{-\frac d2(1-\theta)}c^{-\frac d2}.$$
			As a result,
			\begin{align*}
				w_c&\lesssim_{c,\theta}\sigma^{-\frac d2(1-\theta)}\int_{\R^d}p_\sigma^\theta(z)h(x-z)\d z\\
				&=\sigma^{-\frac d2(1-\theta)}\Vert h\Vert_{L_x^1}\int_{\R^d}p_\sigma^\theta(z)\frac{h(x-z)}{\Vert h\Vert_{L_x^1}}\d z\\
				&\leq \sigma^{-\frac d2(1-\theta)}\Vert h\Vert_{L_x^1}\left(\int_{\R^d}p_\sigma(z)\frac{h(x-z)}{\Vert h\Vert_{L_x^1}}\d z\right)^\theta=\sigma^{-\frac d2(1-\theta)}\Vert h\Vert_{L_x^1}^{1-\theta}w^\theta.
			\end{align*}
			where we use Jensen's inequality in the last inequality, thanks to the condition $0<\theta<1$. The proof is complete.
		\end{proof}
		\begin{lemma}\label{appendix-lm.6}
			Let $K_H(t,\cdot)$ be the kernel of fBm for $H<1/2$. Then, for $0<\beta<H$, we have
			\begin{equation}
				\Vert K_H(t,\cdot)\Vert_{W^{\beta,2}}\lesssim t^{H-\beta}.
			\end{equation}
		\end{lemma}
		\begin{proof}
			Set $K_H(u):=K_H(1,u)$ for $0<u<1$ and $\delta(u)=u\wedge(1-u)$. Then,
			$$K_H(u)=C_H\left[u^{\frac12-H}(1-u)^{H-\frac12}+\left(\frac12-H\right)u^{\frac12-H}\int_u^1v^{H-\frac32}(v-u)^{H-\frac12}\d v\right].$$
			By changing the variable $w=1-u/v$, we obtain
			$$u^{\frac12-H}\int_u^1v^{H-\frac32}(v-u)^{H-\frac12}\d v=u^{H-\frac12}\int_0^{1-u}w^{H-\frac12}(1-w)^{-2H}\d w\lesssim u^{H-\frac12}.$$
			Hence, $K_H(u)\lesssim\delta(u)^{H-1/2}$. In addition, by a straightforward calculation, we have $\frac{\d}{\d u}K_H(u)\lesssim \delta(u)^{H-3/2}$. Subsequently, we estimate
			\begin{align*}
				&\int_0^1\int_0^1\frac{|K_H(u)-K_H(v)|^2}{|u-v|^{1+2\beta}}\d u\d v\\
				&\qquad=2\left(\int_{0<v<u<\frac12}+\int_{0<v<\frac12<u}+\int_{\frac12<v<u<1}\right)\frac{|K_H(u)-K_H(v)|^2}{|u-v|^{1+2\beta}}\d u\d v=:I_1+I_2+I_3
			\end{align*}
			for $0<\beta<H$. For the $I_1$ term, on the domain $|u-v|\leq \delta(v)$, we have
			\begin{align*}
				&\int_0^\frac12\int_v^{2v\wedge\frac12}\frac{|K_H(u)-K_H(v)|^2}{|u-v|^{1+2\beta}}\d u\d v\leq\int_0^\frac12\int_v^{2v\wedge\frac12}|K_H'(\xi)|^2(u-v)^{1-2\beta}\d u\d v\\
				&\qquad\lesssim \int_0^\frac12v^{2H-3}\d v\int_v^{2v}(u-v)^{1-2\beta}\d u=\int_0^\frac12 v^{2H-1-2\beta}\d v<\infty.
			\end{align*}
			On the domain $|u-v|>\delta(v)$, using the estimate $K_H(u)\lesssim\delta(u)^{H-1/2}$, we have
			\begin{align*}
				&\int_0^{\frac14}\int_{2v}^{\frac12}\frac{|K_H(u)-K_H(v)|^2}{|u-v|^{1+2\beta}}\d u\d v\lesssim \int_0^\frac14\int_{2v}^{\frac12}\frac{|K_H(u)|^2+|K_H(v)|^2}{|u-v|^{1+2\beta}}\d u\d v\\
				&\qquad\lesssim \int_0^\frac14v^{2H-1}\d v\int_{2v}^{\frac12}(u-v)^{-1-2\beta}\d u\lesssim \int_0^\frac14 v^{2H-1-2\beta}\d v<\infty.
			\end{align*}
			Thus, we conclude $I_1<\infty$. The $I_3$ term follows analogously. For the $I_2$ term, it holds that
			\begin{align*}
				&\int_{\frac12}^1\int_0^{\frac12}\frac{|K_H(u)-K_H(v)|^2}{|u-v|^{1+2\beta}}\d v\d u\\
				&\qquad\lesssim \int_{\frac12}^1(1-u)^{2H-1}\d u\int_0^{\frac12}(u-v)^{-1-2\beta}\d v+\int_0^{\frac12}v^{2H-1}\d v\int_{\frac12}^1(u-v)^{-1-2\beta}\d u\\
				&\qquad\lesssim\int_{\frac12}^1(1-u)^{2H-1}\left(u-\frac12\right)^{-2\beta}\d u + \int_0^{\frac12}v^{2H-1}\left(\frac12-v\right)^{-2\beta}\d v<\infty.
			\end{align*}
			In addition, we have
			\begin{align*}
				\int_0^t|K_H(t,s)|^2\int_t^T\frac{\d r}{|r-s|^{1+2\beta}}\d s
				&\lesssim\int_0^t|K_H(t,s)|^2(t-s)^{-2\beta}\d s\\
				&=t^{2H-2\beta}\int_0^1|K_H(u)|^2(1-u)^{-2\beta}\d u\lesssim t^{2H-2\beta},
			\end{align*}
			where the last integral is finite because $\beta<H$. Consequently, combining this estimate with the preceding bounds, we obtain
			\begin{align*}
				\Vert K_H(t,\cdot)\Vert_{W^{\beta,2}([0,T])}^2
				&\lesssim t^{2H}+\int_0^t\int_0^t\frac{|K_H(t,s)-K_H(t,r)|^2}{|s-r|^{1+2\beta}}\d s\d r\\
				&\qquad\quad+2\int_0^t|K_H(t,s)|^2\int_t^T\frac{\d r}{|r-s|^{1+2\beta}}\d s\\
				&\lesssim t^{2H-2\beta}.
			\end{align*}
			Here we use the fact that $K_H(t,s)=t^{H-1/2}K_H(s/t)$ and change variables $u=s/t$ and $v=r/t$. The proof is complete.
		\end{proof}
        \section{Proof of Lemma \ref{sec5-lm.1}}\label{B}
        In this appendix, we provide the proof of Lemma \ref{sec5-lm.1}. Without loss of generality, we assume that $s_0=0$ and $x_0=0$. Before giving the proof, we first present some consequences of Lemma \ref{appendix-lm.4} and Stirling's formula. For convenience, we set
	$$\zeta(\alpha,\beta):=
	\frac{\alpha^\alpha(\beta+2\kappa_H)^{\beta+2\kappa_H}}
	{(\alpha+\beta+2\kappa_H)^{\alpha+\beta+2\kappa_H}}.$$
	The first consequence we will use repeatedly in the proof of Lemma \ref{sec5-lm.1} is that for $T>0$, we have
	\begin{align}\label{sec5-eq.1}
		\left(\int_{\Delta_{0,T}^2}s_1^{q'(\beta-H-\frac{Hd}{p})}
		(s_2-s_1)^{-q'(H+\frac{Hd}{p})}(T-s_2)^{q'\alpha}\d s_1\d s_2\right)^{\frac1{q'}}
		\leq C_{H,d,p,q}\zeta(\alpha,\beta)
		T^{\alpha+\beta+2\kappa_H}.
	\end{align}
	In fact, given that
	$$q'\left(\beta-H-\frac{Hd}{p}\right)+1=q'(\beta+\kappa_H),\qquad 1-q'\left(H+\frac{Hd}{p}\right)=q'\kappa_H,$$
	Lemma \ref{appendix-lm.4} yields
	\begin{align*}
		&\int_{\Delta_{0,T}^2}s_1^{q'(\beta-H-\frac{Hd}{p})}(s_2-s_1)^{-q'(H+\frac{Hd}{p})}(T-s_2)^{q'\alpha}\d s_1\d s_2\\
		&\qquad=T^{q'(\alpha+\beta+2\kappa_H)}
		\frac{\Gamma(q'(\beta+\kappa_H))\Gamma(q'\kappa_H)\Gamma(q'\alpha+1)}{\Gamma(q'(\alpha+\beta+2\kappa_H)+1)}.
	\end{align*}
	Therefore, by Stirling's formula, we obtain \eqref{sec5-eq.1}.
	
	We shall also use the following integral estimate in the proof of Lemma \ref{sec5-lm.1}. For any $f(t,\cdot)\in L_x^pL_t^q$, we have 
	\begin{equation}\label{sec5-eq.8}
		\int_0^T \Vert f\Vert_{L_x^p}^2(s_1) s_1^{\beta+1-2H-\frac{2Hd}{p}}(T-s_1)^\alpha\d s_1 \leq C_{H,d,p,q}\zeta(\alpha,\beta)T^{\alpha+\beta+2\kappa_H}\Vert f\Vert_{L_x^pL_t^q}^2.
	\end{equation}
	Indeed, if $q>2$, then H\"older's inequality with exponent $(q/2)'=q/(q-2)$ gives 
	\begin{align*} 
		&\int_0^T \Vert f\Vert_{L_x^p}^2(s_1) s_1^{\beta+1-2H-\frac{2Hd}{p}}(T-s_1)^\alpha\d s_1\\ &\qquad\leq\Vert f\Vert_{L_x^pL_t^q}^2 \left(\int_0^T s_1^{\frac{q}{q-2}(\beta+1-2H-\frac{2Hd}{p})} (T-s_1)^{\frac{q}{q-2}\alpha}\d s_1 \right)^{\frac{q-2}{q}}. 
	\end{align*} 
	Since 
	$$\frac{q}{q-2}\left(\beta+1-2H-\frac{2Hd}{p}\right)+1 =\frac{q}{q-2}(\beta+2\kappa_H),$$ 
	Lemma \ref{appendix-lm.4} together with Stirling's formula implies 
	$$\left(\int_0^T s_1^{\frac{q}{q-2}(\beta+1-2H-\frac{2Hd}{p})} (T-s_1)^{\frac{q}{q-2}\alpha}\d s_1\right)^{\frac{q-2}{q}} \leq C_{H,d,p,q}\zeta(\alpha,\beta)T^{\alpha+\beta+2\kappa_H}.$$ 
	If $q=2$, then the same conclusion follows from 
	$$\int_0^T \Vert f\Vert_{L_x^p}^2(s_1) s_1^{\beta+1-2H-\frac{2Hd}{p}}(T-s_1)^\alpha\d s_1 \leq \Vert f\Vert_{L_x^pL_t^2}^2 \sup_{0<s_1<T}s_1^{\beta+2\kappa_H}(T-s_1)^\alpha $$ and the elementary estimate 
	$$\sup_{0<s_1<T}s_1^{\beta+2\kappa_H}(T-s_1)^\alpha\leq \zeta(\alpha,\beta)T^{\alpha+\beta+2\kappa_H}.$$
	We now turn to the proof of the lemma. First, we partition the domain of integration into two parts
	$\Delta_{0,t}^2=\Delta_1\cup\Delta_2$, where the sets $\Delta_1$ and $\Delta_2$ are defined by
	\begin{align*}
		&\Delta_1:=\{(s_1,s_2)\in\Delta_{0,t}^2:s_1\leq s_2-s_1\},\\
		&\Delta_2:=\{(s_1,s_2)\in\Delta_{0,t}^2:s_2-s_1<s_1\}.
	\end{align*}
	Denote by $I_i$ the integral over $\Delta_i$. For $I_1$, applying H\"older's inequality and the estimate \eqref{sec5-eq.1} gives
	\begin{align}\label{sec5-eq.2}
		I_1
		&\leq C_{H,d,p}\gamma^{\frac{d}{2p'}}
		\int_{\Delta_1}
		\Vert b_1\Vert_{L^p_x}(s_1)\Vert b_2\Vert_{L^p_x}(s_2)
		s_1^{\beta-\frac{Hd}{p}}
		(s_2-s_1)^{-2H-\frac{Hd}{p}}
		(t-s_2)^\alpha\d s_1\d s_2\nonumber\\
		&\leq C_{H,d,p}\gamma^{\frac{d}{2p'}}
		\int_{\Delta_1}
		\Vert b_1\Vert_{L^p_x}(s_1)\Vert b_2\Vert_{L^p_x}(s_2)
		s_1^{\beta-H-\frac{Hd}{p}}
		(s_2-s_1)^{-H-\frac{Hd}{p}}
		(t-s_2)^\alpha\d s_1\d s_2\nonumber\\
		&\leq C_{H,d,p,q}\gamma^{\frac{d}{2p'}}\zeta(\alpha,\beta)t^{\alpha+\beta+2\kappa_H}
		\Vert b_1\Vert_{L^p_xL^q_t}\Vert b_2\Vert_{L^p_xL^q_t},
	\end{align}
	where we use the fact that $s_2-s_1\geq s_1$ in the set $\Delta_1$ for the second inequality.
	
	Then, we decompose $I_2$ as $I_{21}+I_{22}$, where $I_{21}$ is obtained from $I_2$ by restricting the domain of $\d x_1\d x_2$-integration to a set of points $(x_1,x_2)$ such that $|x_1-x_2|/s_1^H$ stays away from zero. For this purpose, we define
	$$B_k(s):=\left\{x=(x^{(1)},\cdots,x^{(d)})\in\R^d:
	x^{(i)}/s^H\in[k^{(i)},k^{(i)}+1),\ i=1,\cdots,d\right\} $$
	for a $d$-dimensional vector $k=(k^{(1)},\cdots,k^{(d)})\in\mathbb Z^d$ and set
	$$\mathcal{S}:=\left\{(k,\ell):k,\ell\in\mathbb{Z}^d,\
	|k|_1\notin\Big[|\ell|_1-4d,2^H|\ell|_1+4d\Big]\right\},$$
	where we write $|k|_1=\sum_{i=1}^d|k^{(i)}|$ for the $L^1$-norm of $k$. In this way, $I_2$ can be decomposed as
	$$I_{21}:=\sum_{(k,\ell)\in\mathcal{S}}I(k,\ell),\qquad
	I_{22}:=\sum_{(k,\ell)\in\mathcal{S}^c}I(k,\ell),$$
	where
	\begin{align*}
		I(k,\ell):=
		&\int_{\Delta_2}\int_{B_k(s_1)\times B_\ell(s_2)}
		b_1(s_1,x_1)b_2(s_2,x_2)\\
		&\qquad\times
		s_1^{-Hd+\beta}e^{-\frac{|x_1|^2}{8\gamma s_1^{2H}}}
		(s_2-s_1)^{-Hd-2H}
		e^{-\frac{|x_2-x_1|^2}{8(s_2-s_1)^{2H}}}
		(t-s_2)^\alpha\d x\d s_1\d s_2.
	\end{align*}
	
	For $I_{21}$, note that $x\in B_k(s)$ implies
	$0\leq x^{(i)}-s^H k^{(i)}<s^H$ for $i=1,\cdots,d$, and hence
	$|x-s^Hk|_1<ds^H$. For $(x_1,x_2)\in B_k(s_1)\times B_{\ell}(s_2)$, it follows that either $|k|_1>2^{H}|\ell|_1+4d$ or $|\ell|_1>|k|_1+4d$. If the former occurs, then
	\begin{align*}
		|x_1-x_2|_1
		&=|(x_1-s_1^Hk)-(x_2-s_2^H\ell)+s_1^{H}k-s_2^{H}\ell|_1\\
		&\geq |s_1^Hk-s_2^H\ell|_1-|x_1-s_1^Hk|_1-|x_2-s_2^H\ell|_1\\
		&\geq s_1^H|k|_1-s_2^H|\ell|_1-d(s_1^H+s_2^H)\\
		&\geq s_1^H|k|_1-2^Hs_1^H|\ell|_1-d(2^H+1)s_1^H\\
		&\geq\left(|k|_1-2^H|\ell|_1-(2^H+1)d\right)s_1^{H}
		\geq ds_1^{H},
	\end{align*}
	where we use $s_2<2s_1$ in the set $\Delta_2$ for the third inequality. If the latter occurs, similarly, we have
	\begin{align*}
		|x_1-x_2|_1
		&\geq |\ell|_1s_2^H-|k|_1s_1^H-d(s_1^H+s_2^H)\\
		&\geq \left(|\ell|_1-|k|_1-(2^H+1)d\right)s_1^H
		\geq ds_1^{H}.
	\end{align*}
	In summary, we have
	$$
	|x_1-x_2|^2\geq d^{-1}|x_2-x_1|_1^2\geq ds_1^{2H}
	$$
	when $(x_1,x_2)\in B_k(s_1)\times B_{\ell}(s_2)$ for some $(k,\ell)\in\mathcal{S}$. From this and H\"older's inequality, we have
	\begin{align}\label{sec5-eq.3}
		I_{21}
		&=\sum_{(k,\ell)\in\mathcal{S}}\int_{\Delta_2}\int_{B_k(s_1)\times B_\ell(s_2)}
		b_1(s_1,x_1)b_2(s_2,x_2)\nonumber\\
		&\qquad\times
		s_1^{-Hd+\beta}e^{-\frac{|x_1|^2}{8\gamma s_1^{2H}}}
		(s_2-s_1)^{-Hd-2H}
		e^{-\frac{|x_2-x_1|^2}{8(s_2-s_1)^{2H}}}
		(t-s_2)^\alpha\d x\d s_1\d s_2\nonumber\\
		&=\sum_{(k,\ell)\in\mathcal{S}}\int_{\Delta_2}\int_{B_k(s_1)\times B_\ell(s_2)}
		b_1(s_1,x_1)b_2(s_2,x_2)\nonumber\\
		&\qquad\times
		s_1^{-Hd+\beta}e^{-\frac{|x_1|^2}{8\gamma s_1^{2H}}}
		\frac{1}{|x_1-x_2|}
		\frac{|x_1-x_2|}{(s_2-s_1)^{Hd+2H}}
		e^{-\frac{|x_2-x_1|^2}{8(s_2-s_1)^{2H}}}
		(t-s_2)^\alpha\d x\d s_1\d s_2\nonumber\\
		&\leq C_{H,d}\int_{\Delta_{0,t}^2}\int_{\R^d}
		b_1(s_1,x_1)s_1^{-Hd+\beta-H}
		e^{-\frac{|x_1|^2}{8\gamma s_1^{2H}}}\d x_1\nonumber\\
		&\qquad\times
		\int_{\R^d}b_2(s_2,x_2)
		\frac{|x_1-x_2|}{(s_2-s_1)^{Hd+2H}}
		e^{-\frac{|x_2-x_1|^2}{8(s_2-s_1)^{2H}}}
		(t-s_2)^\alpha\d x_2\d s_1\d s_2\nonumber\\
		&\leq C_{H,d,p}\gamma^{\frac{d}{2p'}}
		\int_{\Delta_{0,t}^2}
		\Vert b_1\Vert_{L^p_x}(s_1)\Vert b_2\Vert_{L^p_x}(s_2)
		s_1^{\beta-H-\frac{Hd}{p}}
		(s_2-s_1)^{-H-\frac{Hd}{p}}
		(t-s_2)^\alpha\d s_1\d s_2\nonumber\\
		&\leq C_{H,d,p,q}\gamma^{\frac{d}{2p'}}\zeta(\alpha,\beta)t^{\alpha+\beta+2\kappa_H}
		\Vert b_1\Vert_{L^p_xL^q_t}\Vert b_2\Vert_{L^p_xL^q_t}.
	\end{align}
	We obtain the desired bound for $I_{21}$. It remains to estimate $I_{22}$.
	
	For $I_{22}$, we introduce the kernel
	$$\mathcal K^H(s,x):=s^{-Hd-2H}
	\exp\left(-\frac{|x|^2}{8s^{2H}}\right)$$
	for $s>0$ and $x\in\R^d$, whose Fourier transform satisfies
	$$
	\widehat{\mathcal K^H(s,\cdot)}(\xi)=(8\pi)^{\frac d2}s^{-2H}\exp\left(-2s^{2H}|\xi|^2\right).
	$$
	Set
	\begin{align*}
		b_{1k}(s_1,x_1)&:=b_1(s_1,x_1)\mathbbm{1}_{\{x_1\in B_k(s_1)\}}
		e^{-\frac{|x_1|^2}{8\gamma s_1^{2H}}},\\
		b_{2\ell}(s_2,x_2)&:=b_2(s_2,x_2)\mathbbm{1}_{\{x_2\in B_\ell(s_2)\}}.
	\end{align*}
	Then, $I(k,\ell)$ can be expressed as
	\begin{align*}
		I(k,\ell)
		&=\int_{\Delta_2}s_1^{-Hd+\beta}(t-s_2)^\alpha
		\int_{\R^{2d}}b_{1k}(s_1,x_1)b_{2\ell}(s_2,x_2)\\
		&\qquad\times
		\mathcal K^H(s_2-s_1,x_2-x_1)\d x\d s_1\d s_2 .
	\end{align*}
	By Plancherel's formula, we obtain
	\begin{align*}
		I(k,\ell)
		&=(8\pi)^{\frac d2}\int_{\Delta_2}s_1^{-Hd+\beta}(t-s_2)^\alpha
		\int_{\R^d}
		\widehat b_{1k}(s_1,\xi)\overline{\widehat b_{2\ell}(s_2,\xi)}\\
		&\qquad\times
		(s_2-s_1)^{-2H}
		e^{-2(s_2-s_1)^{2H}|\xi|^2}
		\d \xi\d s_1\d s_2.
	\end{align*}
	Using the elementary inequality $|ab|\leq\delta^{-1}|a|^2+\delta|b|^2$ for $\delta>0$, we get
	\begin{align*}
		I(k,\ell)
		&\leq (8\pi)^{\frac d2}\int_{\Delta_2}s_1^{-Hd+\beta}(t-s_2)^\alpha
		\int_{\R^d}
		\left(\delta^{-1}|\widehat b_{1k}(s_1,\xi)|^2
		+\delta|\widehat b_{2\ell}(s_2,\xi)|^2\right)\\
		&\qquad\times
		(s_2-s_1)^{-2H}
		e^{-2(s_2-s_1)^{2H}|\xi|^2}
		\d \xi\d s_1\d s_2\\
		&=:I_1(k)+I_2(\ell).
	\end{align*}
	
			For the $I_1(k)$ term, let $|B_k(s)|$ denote the Lebesgue measure of $B_k(s)$. By definition, $|B_k(s)|=s^{Hd}$. Using the fact that $s_2-s_1<s_1$ and $H<1/2$, we have
	\begin{align*}
		I_1(k)
		&\leq C_d\delta^{-1}\int_0^t s_1^{-Hd+\beta}(t-s_1)^\alpha
		\int_{\R^d}|\widehat b_{1k}(s_1,\xi)|^2
		\int_{s_1}^{(2s_1)\wedge t}
		(s_2-s_1)^{-2H}
		e^{-2(s_2-s_1)^{2H}|\xi|^2}\d s_2\d \xi\d s_1\\
		&\leq C_{H,d}\delta^{-1}\int_0^t s_1^{-Hd+\beta+1-2H}
		(t-s_1)^\alpha
		\int_{\R^d}|b_{1k}(s_1,x_1)|^2\d x_1\d s_1\\
		&\leq C_{H,d}\delta^{-1}e^{-\frac{|k|_1^2}{16d\gamma}}
		\int_0^t s_1^{-Hd+\beta+1-2H}(t-s_1)^\alpha
		\int_{B_k(s_1)}|b_1(s_1,x_1)|^2\d x_1\d s_1\\
		&\leq C_{H,d}\delta^{-1}e^{-\frac{|k|_1^2}{16d\gamma}}
		\int_0^t s_1^{-Hd+\beta+1-2H}(t-s_1)^\alpha\\
		&\qquad\times |B_k(s_1)|^{1-\frac2p}
		\left(\int_{B_k(s_1)}|b_1(s_1,x_1)|^p\d x_1\right)^{\frac2p}
		\d s_1\\
		&\leq C_{H,d,p}\delta^{-1}e^{-\frac{|k|_1^2}{16d\gamma}}
		\int_0^t\Vert b_1\Vert^2_{L^p_x}(s_1)
		s_1^{\beta+1-2H-\frac{2Hd}{p}}(t-s_1)^\alpha\d s_1\\
		&\leq C_{H,d,p,q}\delta^{-1}e^{-\frac{|k|_1^2}{16d\gamma}}\zeta(\alpha,\beta)
		t^{\alpha+\beta+2\kappa_H}
		\Vert b_1\Vert^2_{L^p_xL^q_t}.
	\end{align*}
	The second inequality follows from Parseval's identity, while the fourth and final inequalities follow from H\"older's inequality and the integral estimate \eqref{sec5-eq.8}. We now proceed to explain the third inequality. Note that
	$$
	|x/s^H|_1=|x/s^H-k+k|_1
	\geq\big||k|_1-|x/s^H-k|_1\big|.
	$$
	By the definition of $B_k(s_1)$, we know that $|x/s_1^H-k|_1<d$ when $x\in B_k(s_1)$. Hence, there are only finitely many $k\in\mathbb{Z}^d$ and corresponding sets $B_k(s_1)$ for which $\big||k|_1-|x/s_1^H-k|_1\big|\geq |k|_1/2$ fails for $x\in B_k(s_1)$. Their contribution can be absorbed into a constant $C$. Overall, we have
	$$
	|b_{1k}(s_1,x_1)|^2
	\leq C|b_1(s_1,x_1)|^2
	\mathbbm{1}_{\{x_1\in B_k(s_1)\}}e^{-\frac{|k|_1^2}{16d\gamma}},
	$$
	which gives the third inequality.
	
	For the $I_2(\ell)$ term, similarly,
	\begin{align*}
		I_2(\ell)
		&\leq C_{H,d}\delta\int_0^t s_2^{-Hd+\beta}(t-s_2)^\alpha
		\int_{\R^d}|\widehat b_{2\ell}(s_2,\xi)|^2
		\int_{s_2/2}^{s_2}
		(s_2-s_1)^{-2H}
		e^{-2(s_2-s_1)^{2H}|\xi|^2}\d s_1\d \xi\d s_2\\
		&\leq C_{H,d}\delta\int_0^t s_2^{-Hd+\beta+1-2H}
		(t-s_2)^\alpha
		\int_{B_\ell(s_2)}|b_2(s_2,x_2)|^2\d x_2\d s_2\\
		&\leq C_{H,d,p}\delta\int_0^t
		\Vert b_2\Vert^2_{L^p_x}(s_2)
		s_2^{\beta+1-2H-\frac{2Hd}{p}}(t-s_2)^\alpha\d s_2\\
		&\leq C_{H,d,p,q}\delta\zeta(\alpha,\beta)
		t^{\alpha+\beta+2\kappa_H}
		\Vert b_2\Vert^2_{L^p_xL^q_t}.
	\end{align*}
	Here we use the fact that $s_2-s_1<s_1$ in the set $\Delta_2$.
	
	In summary, choosing
	$$\delta=e^{-\frac{|k|_1^2}{32d\gamma}}
	\frac{\Vert b_1\Vert_{L^p_xL^q_t}}{\Vert b_2\Vert_{L^p_xL^q_t}},$$
	with the natural convention if one of the norms is zero, we have
	$$I(k,\ell)\leq C_{H,p,d,q}e^{-\frac{|k|_1^2}{32d\gamma}}
	\zeta(\alpha,\beta)t^{\alpha+\beta+2\kappa_H}
	\Vert b_1\Vert_{L^p_xL^q_t}\Vert b_2\Vert_{L^p_xL^q_t}.$$
	For $(k,\ell)\in\mathcal{S}^c$, we know that
	$$\frac{|k|_1-4d}{2^H}\leq|\ell|_1\leq |k|_1+4d.$$
	Therefore, for each fixed $k\in\mathbb Z^d$, the number of admissible $\ell$ is bounded by
	\begin{align*}
		\#\{\ell\in\mathbb Z^d:(k,\ell)\in\mathcal S^c\}
		&\leq \#\{\ell\in\mathbb Z^d:|\ell|_1\leq |k|_1+4d\}\\
		&\leq C_d(1+|k|_1)^d.
	\end{align*}
	Indeed, the condition $|\ell|_1\leq |k|_1+4d$ implies
	$|\ell^{(i)}|\leq |k|_1+4d$ for $i=1,\cdots,d$, and hence the number of such $\ell$ is at most
	$(2(|k|_1+4d)+1)^d$. Given that $\gamma\geq1$, we have
	\begin{align*}
		&\sum_{(k,\ell)\in \mathcal{S}^c}e^{-\frac{|k|_1^2}{32d\gamma}}\leq\sum_{k\in\mathbb{Z}^d}(1+|k|_1)^de^{-\frac{|k|_1^2}{32d\gamma}}=\sum_{n=0}^\infty\sum_{|k|_1=n}(1+n)^de^{-\frac{n^2}{32d\gamma}}\\
		&\qquad\leq
		C_d\sum_{n=0}^{\infty}
		(1+n)^{2d-1}e^{-\frac{n^2}{32d\gamma}}\leq
		C_d\int_0^\infty
		(1+u)^{2d-1}e^{-\frac{u^2}{32d\gamma}}\d u\leq C_d\gamma^d,
	\end{align*}
	where the last inequality follows from the change of variable $v=u/\sqrt{\gamma}$. Consequently, we obtain
	\begin{align}\label{sec5-eq.4}
		I_{22}=\sum_{(k,\ell)\in\mathcal{S}^c}I(k,\ell)\leq C_{H,d,p,q}
		\gamma^d\zeta(\alpha,\beta)t^{\alpha+\beta+2\kappa_H}
		\Vert b_1\Vert_{L^p_xL^q_t}\Vert b_2\Vert_{L^p_xL^q_t}.
	\end{align}
	Combining \eqref{sec5-eq.2}-\eqref{sec5-eq.4} and using the fact that $\gamma\geq1$ gives the desired estimate \eqref{sec5-eq.7} for $I$. The proof is complete.
	\end{appendices}
	\bibliographystyle{imsart-number}
\bibliography{ref}
\end{document}